\documentclass[12pt]{article}
\usepackage[margin =1in]{geometry}
\usepackage{amssymb,amsthm}
\usepackage{amsmath}
\numberwithin{equation}{section}

\newcommand{\E}{{\bf E}}
\newcommand{\R}{{\mathbb R}}

\newcommand{\cU}{\mathcal{U}}

\newcommand{\norm}[1] {\left \| #1 \right \|}
\newcommand{\inclu}[0] {\ar@{^{(}->}}

\newcommand{\Null}{\text{Null}}
\newcommand{\tr}{\text{tr}}

\newcommand{\spann}{\text{span}}

\newcommand{\dist}{{\rm dist}}

\newcommand{\cA}{\mathcal{A}}

\newcommand{\EE}{\mathbb{E}}
\newcommand{\trace}[1]{\mathrm{trace}\left( #1\right)}

\newcommand{\RR}{\mathbb{R}}

\newcommand{\cT}{\mathcal{T}}

\newcommand{\cM}{\mathcal{M}}

\newcommand{\range}{\mathrm{range}}

\newcommand{\rank}{\mathrm{rank}}

\newcommand{\abs}[1]{\left| #1 \right|}

\newcommand{\fnorm}[1]{\left\| #1 \right\|_{F}}

\newcommand{\proj}{P}

\newcommand{\argmin}{\operatornamewithlimits{argmin}}

\newcommand{\opnorm}[1]{\left\|#1\right\|_{op}}

\newcommand{\dotp}[1]{\left\langle #1\right\rangle}

\newtheorem{theorem}{Theorem}[section]

\newtheorem{definition}[theorem]{Definition}
\newtheorem{proposition}[theorem]{Proposition}
\newtheorem{lemma}[theorem]{Lemma}
\newtheorem{corollary}[theorem]{Corollary}

\newtheorem{assumption}{Assumption}

\newcommand{\lb}{\mathtt{lb}}
\newcommand{\ub}{\mathtt{ub}}

\newcommand{\paren}[1]{ \left( #1 \right) }
\theoremstyle{remark}

\usepackage{mathtools}

\usepackage{hyperref}

\usepackage{color}
\usepackage{algorithm}
\usepackage{enumitem}
\usepackage{algpseudocode}
\usepackage{caption}
\usepackage{subcaption}
\usepackage[numbers,sort&compress]{natbib}
\newcommand{\gdstep}{\eta}

\newcommand{\gdrds}{K}
\newcommand{\etax}{\eta_x}
\newcommand{\cS}{\mathcal{S}}
\newcommand{\gd}{\mathtt{GD}}
\newcommand{\gdpk}{\mathtt{GDPolyak}}
\newcommand{\gdpklb}{\mathtt{GDPolyakLB}}
\newcommand{\polyak}{\mathtt{Polyak}}
\newcommand{\gammalb}{\gamma_\lb}
\newcommand{\gammaub}{\gamma_\ub}
\newcommand{\betalb}{\beta_\lb}
\newcommand{\betaub}{\beta_\ub}

\newcommand{\dlb}{D_\lb}
\newcommand{\dub}{D_\ub}
\newcommand{\etay}{\eta_y}

\newcommand{\deltaU}{\delta_{\mathtt{U}}}
\newcommand{\deltagdpk}{\delta_{\mathtt{init}}}
\newcommand{\xout}{x_{\mathtt{out}}}

\newcommand{\ckl}{c_{\mathtt{L}}}
\newcommand{\lipf}{L_f}

\definecolor{blue}{rgb}{0,0,0}
\title{Gradient descent with adaptive stepsize converges (nearly) linearly under fourth-order growth}

\author{Damek Davis\thanks{Wharton Department of Statistics and Data Science, University of Pennsylvania,
		Philadelphia, PA 19104, USA;
		\texttt{www.damekdavis.com}. Research of Davis supported by an Alfred P. Sloan research fellowship and NSF DMS award 2047637. Research was completed while Davis was visiting the Simons Institute for the Theory of Computing.}
	\and Dmitriy Drusvyatskiy\thanks{Department of Mathematics, U. Washington, Seattle, WA 98195; \texttt{www.math.washington.edu/$\sim$ddrusv}.
		Research of Drusvyatskiy was supported by NSF DMS-2306322, NSF CCF 1740551, and AFOSR FA9550-24-1-0092 awards.}
	\and Liwei Jiang\thanks{Edwardson School of Industrial Engineering, Purdue University, West Lafayette, IN 47906, USA;
		\texttt{liwei-jiang97.github.io/}.}
}	

	\date{}

\begin{document}
	\maketitle
		\begin{abstract}A prevalent belief among optimization specialists is that linear convergence of gradient descent is contingent on the function growing quadratically away from its minimizers.  In this work, we argue that this belief is inaccurate. We show that gradient descent with an adaptive stepsize converges at a local (nearly) linear rate on any smooth function that merely exhibits fourth-order growth away from its minimizer. The adaptive stepsize we propose arises from an intriguing decomposition theorem: any such function admits a smooth manifold around the optimal solution---which we call the ravine---so that the function grows at least quadratically away from the ravine and has constant order growth along it. The ravine allows one to interlace many short gradient steps with a single long Polyak gradient step, which together ensure rapid convergence to the minimizer. We illustrate the theory and algorithm on the problems of  matrix sensing and factorization and learning a single neuron in the overparameterized regime.
\end{abstract}

\section{Introduction}	
	Classical optimization literature shows that gradient descent converges linearly when applied to smooth convex functions that grow quadratically away from their minimizers. See, for example, the seminal work \cite{polyak1963gradient} in this regard. 
Numerous extensions of such results to more sophisticated proximal algorithms have been established recently, for example, in \cite{Luo1993,drusvyatskiy2018error,karimi2016linear,zhou2017unified,necoara2019linear}. This expansive body of literature suggests that linear convergence of gradient descent is contingent upon quadratic growth. Indeed, simple examples confirm this intuition for constant stepsize gradient descent.
In contrast, this work shows that this popular belief is no longer accurate when the stepsizes can be chosen adaptively.

As motivation, let us perform the following thought experiment. Consider minimizing the univariate function $f(x)=\tfrac{1}{4}x^4$, which grows only quartically away from its minimizer. Then gradient descent with stepsize $\eta_t$ generates the iterates: 
$x_{t+1}=(1-\eta_t x^2_t)x_t.$ Clearly, if the stepsize is constant $\eta_t\equiv\eta$, then gradient descent converges at a sublinear rate. On the other hand, if we set the stepsize adaptively, say according to Polyak's rule $\eta_t=f(x_t)/(f'(x_t))^2=1/4x^2$, then the iterates  $x_{t+1}=\tfrac{3}{4}x_t$ converge linearly to zero. 
This rudimentary example suggests that adaptively chosen long steps may endow gradient descent with a local linear rate of convergence even for highly degenerate functions. We show that this is indeed the case, at least when the function grows quartically away from its solution set.
Intriguingly, this suggests that a ``good" stepsize for the function $f$ is one that grows exponentially with the iteration counter, a phenomenon we will observe repeatedly. 
The stepsize sequence we use is epoch-based: the algorithm takes multiple constant size gradient steps in each epoch followed by a single long Polyak step. This strategy is different from other adaptive stepsize/preconditioning methods popular in the machine learning literature, such as AdaGrad~\cite{duchi2011adaptive} and Adam~\cite{kingma2014adam}. More formally, 
	recall that gradient descent with constant stepsize $\eta$ is simply the algorithm:
	\begin{equation} \label{eqn:grad_desc_const}
		x_{k+1} = x_k - \gdstep \nabla f(x_k).
	\end{equation}
	Henceforth, we  let $\gd(x,  \gdstep,\gdrds)=x_K$ denote the $K$-th iterate of the gradient descent sequence \eqref{eqn:grad_desc_const} when initialized at $x$. The algorithm we propose, summarized as Algorithm~\ref{alg:GD-Polyak}, proceeds by alternating between $K$ steps of constant stepsize gradient descent and a single Polyak step. 
	
	\begin{algorithm}[H]
		\caption{$\gdpk(x_0, \gdstep, \gdrds, I$)}
		\label{alg:GD-Polyak}
		\begin{algorithmic}[1]
			\State {\bfseries Input}   $x_0, \gdstep, K, I$.
			\For {$i = 1,\ldots, I$}
			\State $\tilde x_{i} = \gd(x_{i-1}, \gdstep, \gdrds)$
			\State $x_{i} = \tilde x_{i} -\frac{f(\tilde x_{i}) - f^*}{\|\nabla f(\tilde x_{i})\|^2} \nabla f(\tilde x_{i})$.
			\EndFor
			\State $x_{\mathtt{out}} =  \argmin \{f(x_i), f(\tilde x_i) \colon i = 1,\ldots, I\} $
			\State \Return $\xout$
		\end{algorithmic}
\end{algorithm}
The following is our main theorem. For simplicity, we state it when the minimizer is unique; see Theorem~\ref{thm: main_thm_intro} for the general result.

\begin{theorem}[informal]\label{thm:basic_scifi}
	Consider a smooth function $f$ satisfying $f(x)-\inf f\geq \Omega( \|x-\bar x\|^4)$ for all $x$ near the minimizer $\bar x$. Then, when initialized sufficiently close to $\bar x$ with sufficiently small $\eta$, Algorithm~\ref{alg:GD-Polyak} reaches any $\varepsilon$-ball around $\bar x$ after $O(\log^2(1/\varepsilon))$ gradient evaluations.
\end{theorem}

\begin{figure}[h]
	\centering
	\begin{subfigure}[b]{0.55\textwidth}
		\includegraphics[width=\textwidth]{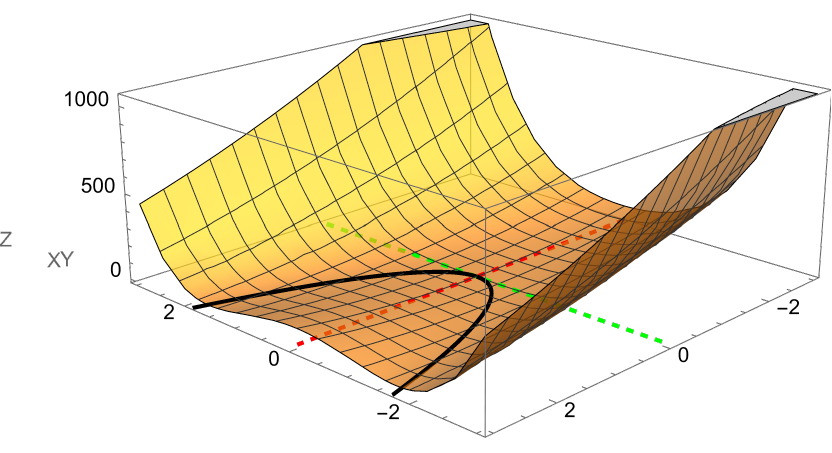}
		\caption{The ravine $\cM=\{(x,y):y=x^2\}$ (black), and the tangent/normal space (green/red)}
	\end{subfigure}
	\hfill
	\begin{subfigure}[b]{0.4\textwidth}
		\includegraphics[width=\textwidth]{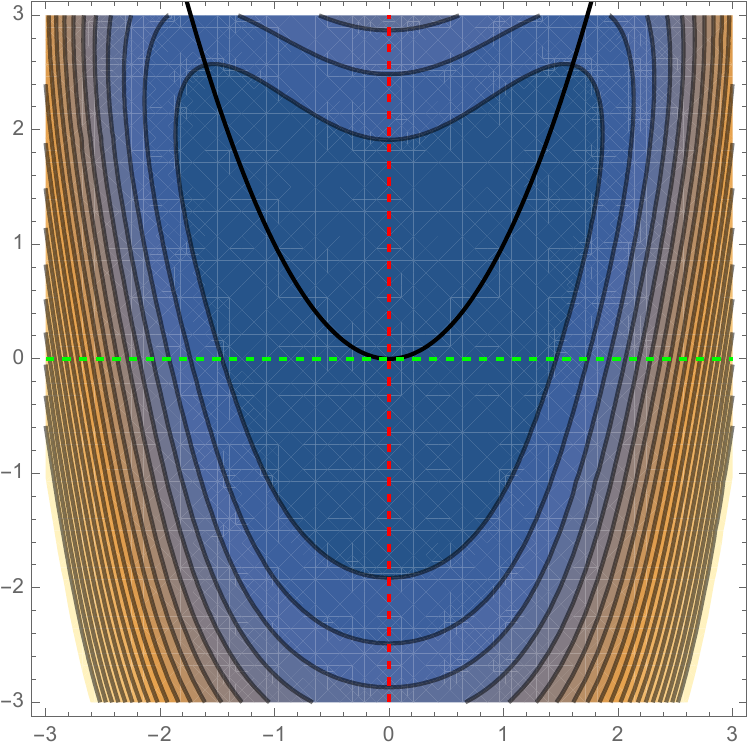}
		\caption{Contour plot}
		\begin{minipage}{.1cm}
			\vfill
		\end{minipage}
	\end{subfigure}
		\caption{The function $f(x,y)=x^4+10(y-x^2)^2$}
		\label{fig:combined_intro}
	\end{figure}
	
	The motivation behind our stepsize sequence is best illustrated with an example. Consider the Rosenbrock function $f(x,y)=x^4+10(y-x^2)^2$ depicted in Figure~\ref{fig:combined_intro}, which has the origin as its minimizer. The Rosenbrock function is designed to make first-order methods perform poorly since gradient descent has the tendency to jump back and forth across the parabola $\cM=\{(x,y): y=x^2\}$, depicted in black in Figure~\ref{fig:combined_intro}. In contrast, a faster algorithm such as Newton's method would traverse $\cM$ tangentially. The manifold $\cM$ is geometrically distinctive in that $f$ has a valley along $\cM$, and hence we will call $\cM$ the {\em ravine}. There is a long history of such geometric structures in optimization going back at least to the ``ravine method" of Gelfand and Tsetlin \cite{gelfand1961printszip}, which explicitly tries to move tangent to the ravine. Notably, the ravine method was a precursor to Polyak's heavy ball algorithm \cite{polyak1964some} and Nesterov's accelerated method \cite{nest_orig}, and has gained some recent attention \cite{attouch2022ravine,shi2022understanding}.

	The ravine in this example plays a special role, stagnating the performance of gradient descent with constant stepsize. Roughly speaking, gradient descent struggles because $f$ grows rapidly (quadratically) away from $\cM$ and slowly (quartically) along $\cM$. Given the importance of the ravine in this example, it is natural to study analogous objects for general smooth functions. For a smooth function $f$ with a minimizer $\bar x$, we introduce the definition:
	
	\begin{quote}
		A manifold $\cM$ is called a {\em ravine} for $f$ at $\bar x$ if $\cM$ is tangent to ${\rm Null}(\nabla^2 f(\bar x))$ at $\bar x$ and there is a retraction\footnote{By a retraction, we simply mean a smooth map $R\colon U\to\cM$ defined on a neighborhood $U$ of $\bar x$ that restricts to identity on $\cM$ and such that the Jacobian $\nabla R(\bar x)$ coincides with the projection onto the tangent space of $\cM$ at $\bar x$.} $ R(\cdot)$ onto $\cM$ satisfying growth lower bound: 
		\begin{equation}\label{eqn:prop_intro}
			f(x)-f(R(x)) \geq \Theta(1)\cdot \|x-R(x)\|^2.
		\end{equation}
	\end{quote}
	In words, the definition stipulates that there is a retraction $R$ onto $\cM$ such that the Function gap $f(x)-f(R(x))$ is lower bounded by a constant multiple of the square distance $\|x-R(x)\|^2$. Although the ideal retraction is the nearest-point projection $P_{\cM}$, stipulating the equality $R=P_{\cM}$ would be quite stringent. For example, the ravine of the  Rosenrock function at the origin is simply the parabola $\cM=\{(x,y)\colon y=x^2\}$ and the retraction is the map $R(x,y)=(x,x^2)$. Moreover, it is straightforward to see that the projection $P_{\cM}$ does not satisfy the requisite property \eqref{eqn:prop_intro}; indeed, the gap $f(x)-f(P_{\cM}(x))$ can be negative. Reassuringly, we will show that a ravine always exists due to the so-called Morse Lemma with parameters.

	With the ravine $\cM$ at hand, we can  decompose $f$ into  {\em normal} and {\em tangent parts}:
	$$f(x)=f_{N}(x)+f_{T}(x),$$
	where we define $f_N(x):=f(x)-f(P_{\cM}(x))$ and $f_T(x):=f(P_{\cM}(x))$.
	We will see that condition~\ref{eqn:prop_intro} implies that the iterates of gradient descent with a constant stepsize will approach $\cM$ at a linear rate up to a point when $f$ behaves similarly to its tangent part $f_{T}$. At this point, if we assume that $f_T$ behaves like a power function on $\cM$, a single Polyak gradient step will move the iterate significantly closer to the optimal solution. Unfortunately, the Polyak step causes the next iterate to move far away from the ravine. Therefore, we repeat the process, running multiple constant size gradient steps again, followed by a single Polyak step, and so forth. Interestingly, we show that if 
	$f$ grows quartically away from a unique minimizer, then $f_T$ automatically has constant order growth on $\cM$; therefore, the logic above applies. Combining all the ingredients yields the main Theorem~\ref{thm:basic_scifi}.\footnote{If the set of the minimizers $S$ is not a singleton, the same conclusion applies if in addition, we assume that the Hessian $\nabla^2 f$ has constant rank along $S$.}

	\begin{figure}[t]
		\centering
		\begin{subfigure}[b]{0.325\textwidth}
			\centering
			\includegraphics[width=\textwidth]{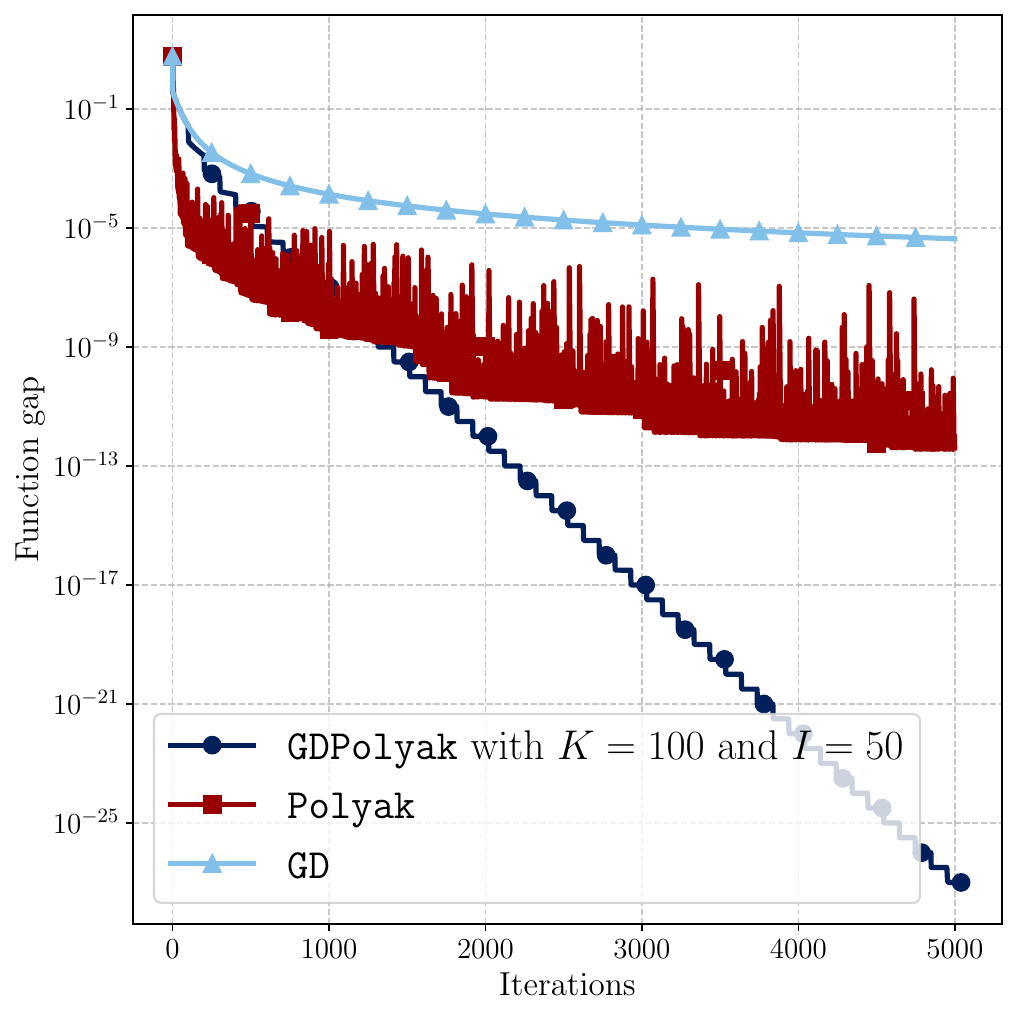} 
			\caption{Function gap}
			\label{fig:rosenbrock:gap}
		\end{subfigure}
		\begin{subfigure}[b]{0.325\textwidth}
			\centering
			\includegraphics[width=\textwidth]{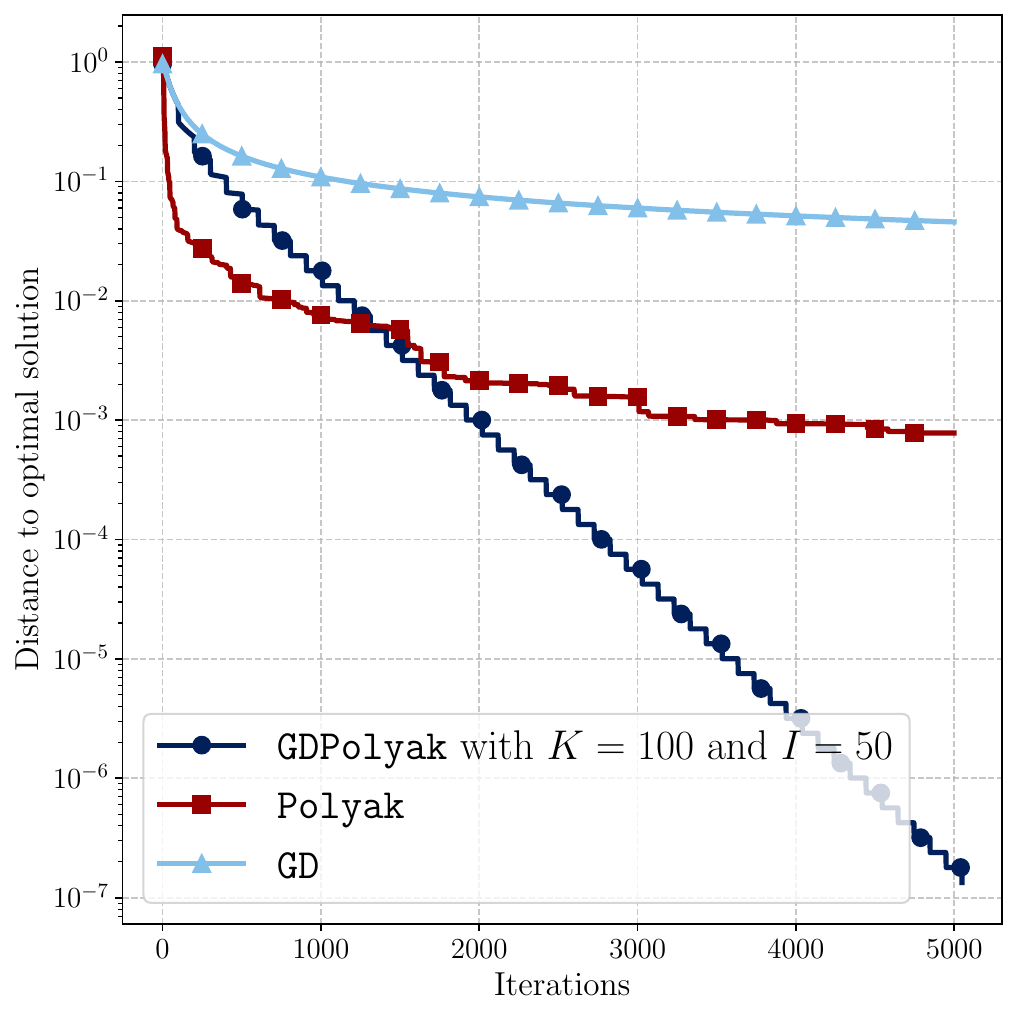} 
			\caption{Distance to optimal solution}
			\label{fig:rosenbrock:distance}
		\end{subfigure}
		\begin{subfigure}[b]{0.325\textwidth}
			\centering
			\includegraphics[width=\textwidth]{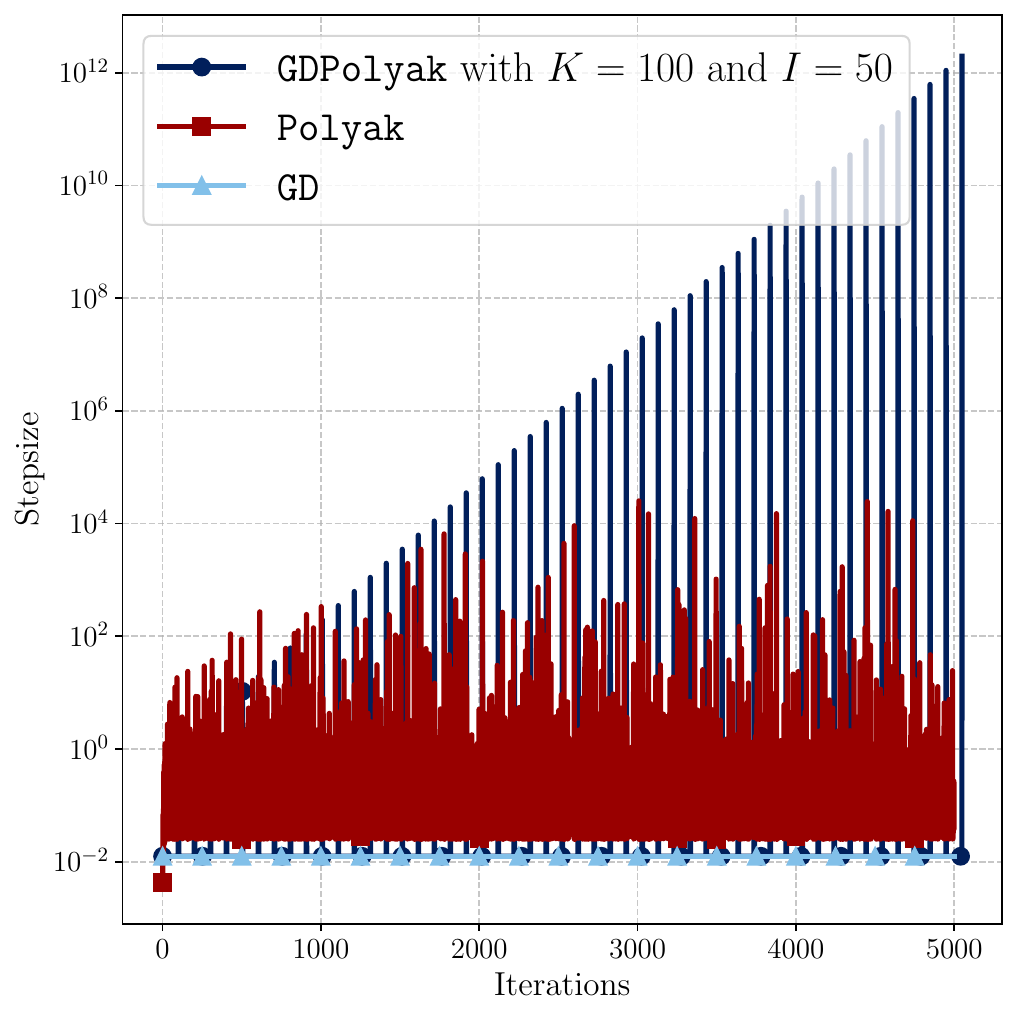} 
			\caption{Stepsize}
			\label{fig:rosenbrock:stepsize}
		\end{subfigure}
		
		\caption{Comparison of $\gdpk$ with $\gd$ and $\polyak$ on the Rosenbrock function. $\gdpk$ proceeds in $I = 50$ epochs of length $K = 100$. 
			During each epoch, $\gdpk$ uses the same short stepsize as $\gd$, i.e., $.0125$.
			After taking $K=100$ steps with short stepsizes, $\gdpk$ takes a step with the Polyak stepsize $\frac{f(x) - f^\ast}{\|\nabla f(x)\|^2}$.}
		\label{fig:rosenbrock}
	\end{figure}
	Figure~\ref{fig:rosenbrock} illustrates the performance of our proposed algorithm -- denoted $\gdpk$ -- on the Rosenbrock function.\footnote{Code is available at \href{https://github.com/damek/GDPolyak}{https://github.com/damek/GDPolyak}} We compare  $\gdpk$ to gradient descent with constant stepsize ($\gd$) and gradient descent with the Polyak stepsize ($\polyak$). 
	The plots show that while $\gd$ and $\polyak$ converge sublinearly both in terms of function value and distance to the optimal solution, $\gdpk$ converges linearly. 
	Looking at Figure~\eqref{fig:rosenbrock:stepsize}, we see that the adaptive stepsize of $\gdpk$ is substantially different from the stepsize taken by $\gd$ and $\polyak$.
	Intriguingly, the plot shows that this stepsize is growing exponentially.

	We apply our techniques to two applications: matrix factorization/sensing and overparameterized training of a student-teacher neural network.
	
	\paragraph{Matrix sensing.}
	\begin{figure}[h]
		\centering
		\begin{subfigure}[b]{0.325\textwidth}
			\centering
			\includegraphics[width=\textwidth]{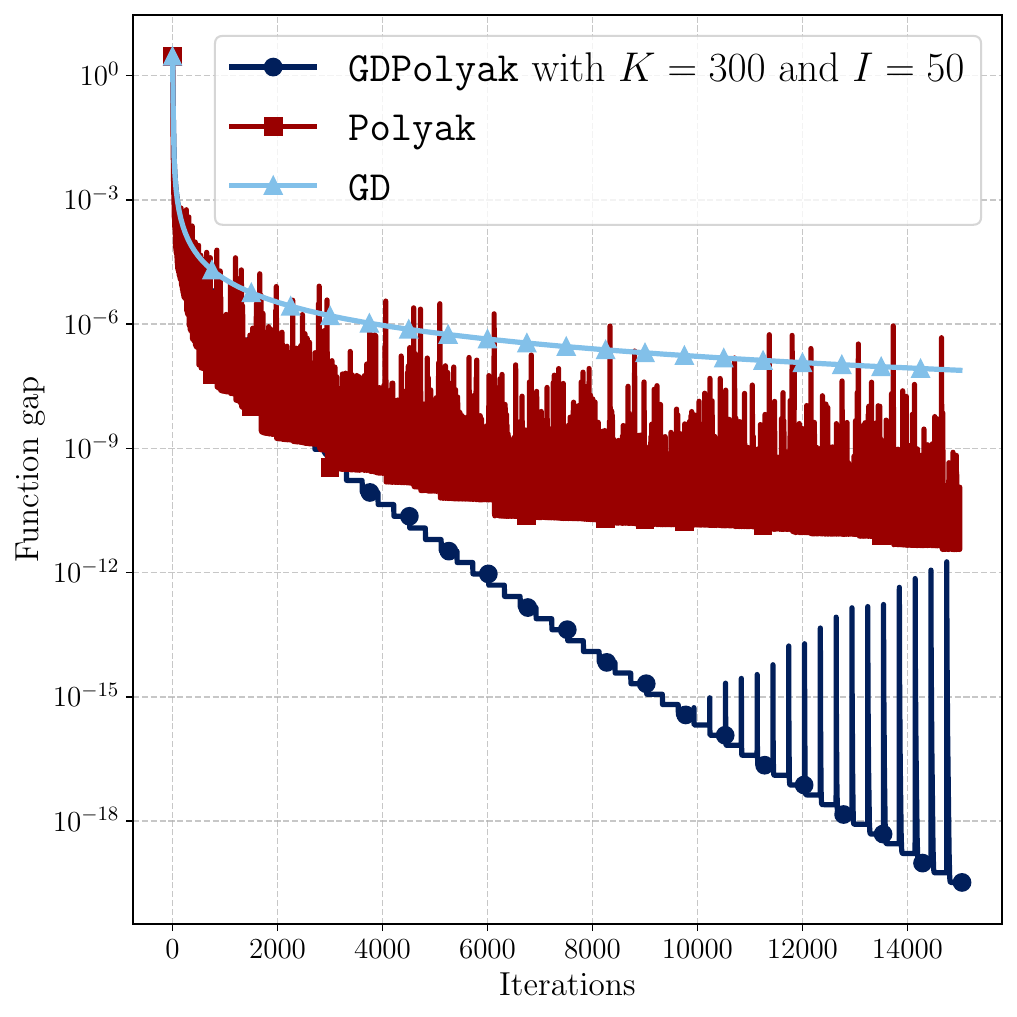} 
			\caption{Function gap}
		\end{subfigure}
		\hfill
		\begin{subfigure}[b]{0.325\textwidth}
			\centering
			\includegraphics[width=\textwidth]{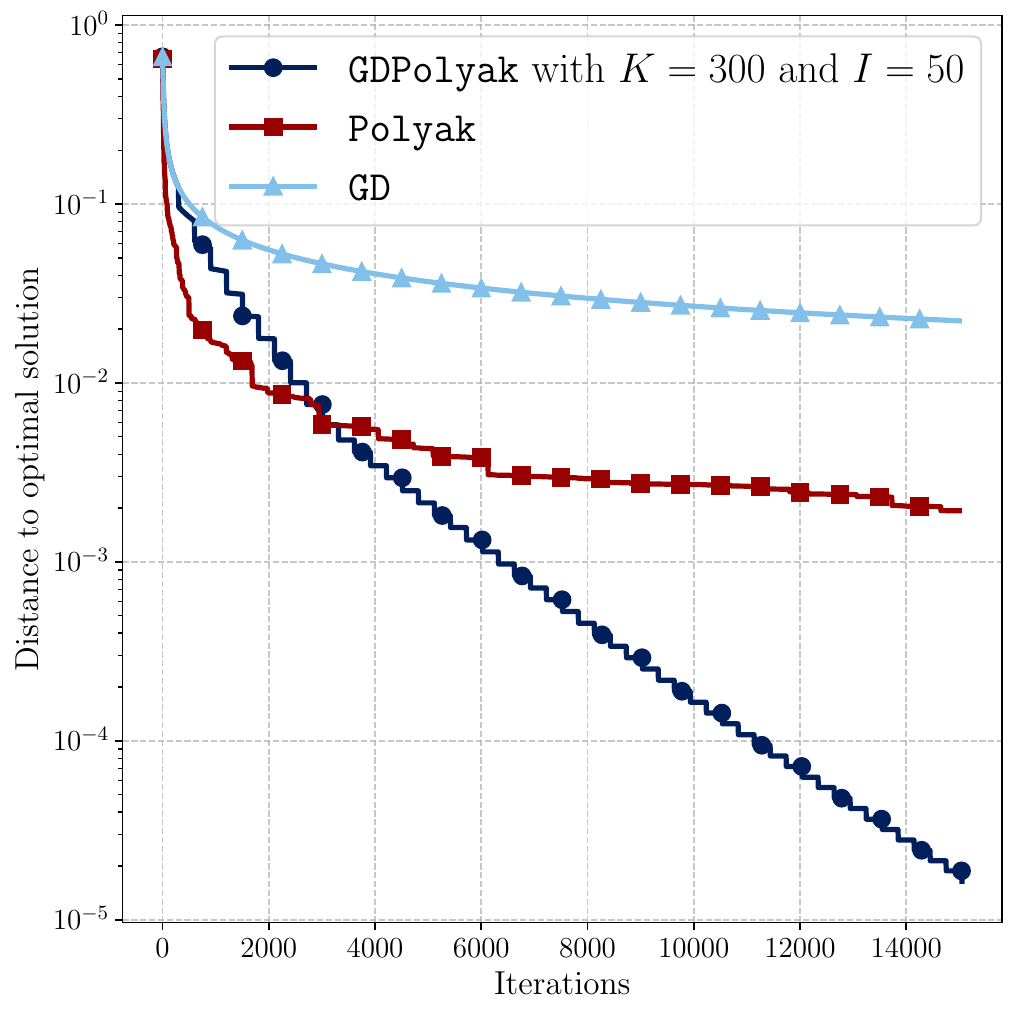} 
			\caption{Distance to optimal solution}
		\end{subfigure}
		\hfill
		\begin{subfigure}[b]{0.325\textwidth}
			\centering
			\includegraphics[width=\textwidth]{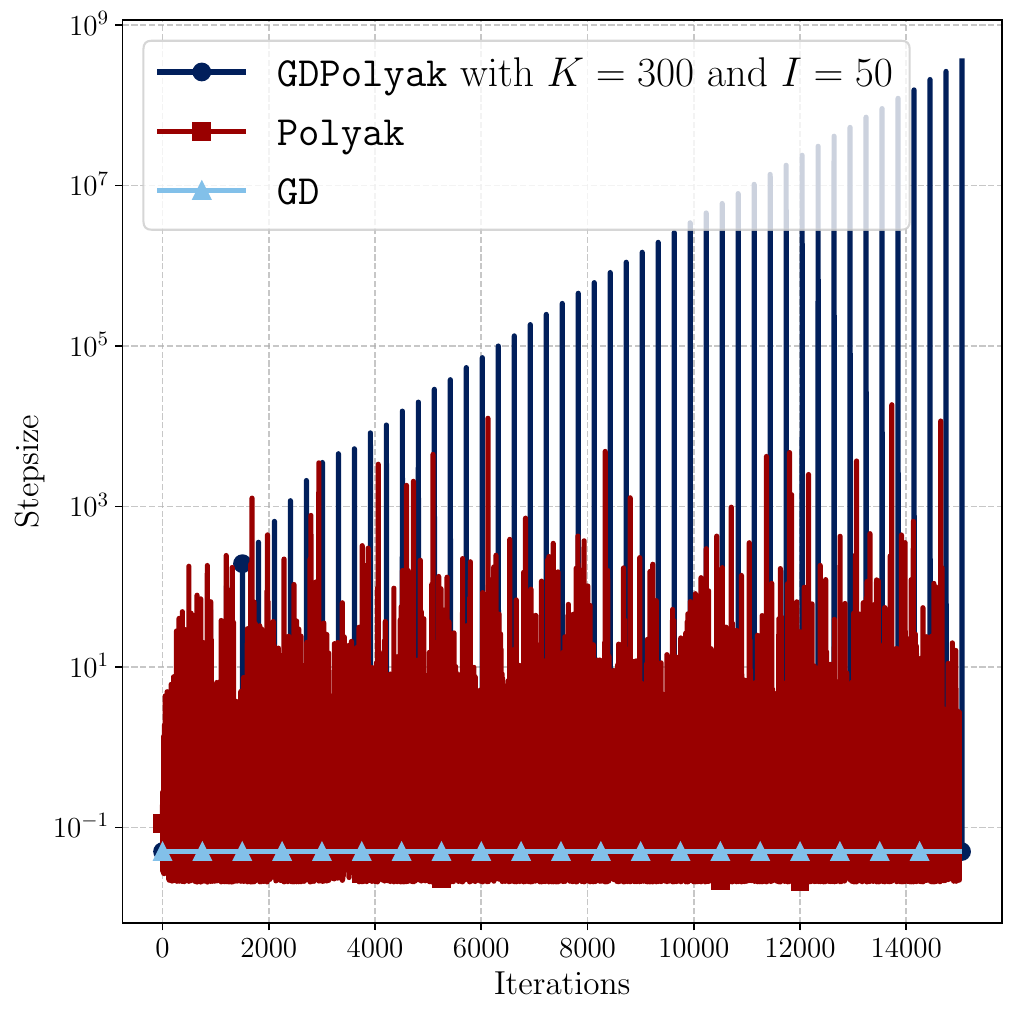} 
			\caption{Stepsize}
		\end{subfigure}
		
		\caption{Comparison of $\gdpk$ with $\gd$ and $\polyak$ on an overaparameterized quadratic matrix sensing problem. Each measurement matrix is of the form $A_i = a_ia_i^T - \tilde a_i \tilde a_i^T$ where $a_i$ and $\tilde a_i$ are $d$-dimensional standard Gaussians.
			In this experiment, $d=100$, the unknown rank is $r = 2$, and the overparameterized rank is $k = 4$. 
			For $\gdpk$, we run the method for $I=50$ epochs of size $K = 300$. 
			In each epoch, $\gdpk$ uses constant stepsize $.05$.}
		\label{fig:gaussian_sensing}
	\end{figure}
	Low-rank matrix sensing problems appear in a wide variety of applications, such as quantum state tomography, image processing, multi-task regression, and metric embeddings \cite{recht2010guaranteed,candes2011tight,liu2011universal,flammia2012quantum,chi2019nonconvex}. The goal is to recover a symmetric positive semidefinite matrix $X \in \RR^{d\times d}$ with low rank $r \ll d$ from a set of linear measurements $y_i = \dotp{A_i, X}$, where $A_i \in \RR^{d\times d}$ are known matrices. A popular approach (e.g., \cite{chi2019nonconvex}) is to form a low-rank factorization of the matrix $X=BB^\top$ and optimize the mean-square-error over the factors:
	\begin{equation}\label{eqn: matrixsensingobjintro}
		\min_{B\in \R^{d\times k}}f(B) = \frac{1}{4m}\sum_{i=1}^{m} (y_i - \dotp{A_i, BB^\top})^2.
	\end{equation}
	This factorized approach was pioneered by Burer-Monteiro in~\cite{burer2003nonlinear,burer2005local}. If the true rank $r$ of the underlying matrix $X$ is known, then $k=r$ is the ideal choice. In this exact regime, and under a ``restricted isometry property," the objective function $f$ grows quadratically away from its solution set  \cite{tu2016low,ge2017no,zhu2018global}. The rank $r$, however, is rarely known exactly, and instead, one resorts to an overestimated $k>r$. In this so-called rank-overparameterized regime, the objective function $f$ only grows quartically from the solution set, and therefore, standard gradient methods converge sublinearly at best~\cite{zhuo2021computational}. It is in this rank overparametrized regime that all of our results apply and gradient descent with  adaptive stepsize converges at a local (nearly) linear rate.

	As a numerical illustration, Figure~\ref{fig:gaussian_sensing} compares the performance of $\gdpk$, gradient descent with constant stepsize $.05$ ($\gd$), and gradient descent with Polyak stepsize ($\polyak$) on a matrix factorization problem.
	Again, the figures show that $\gd$ and $\polyak$ converge sublinearly, while $\gdpk$ converges nearly linearly. 
	In addition, the long stepsizes taken by $\gdpk$ grow exponentially in the iteration counter.

	Let us briefly take a closer look at an idealized version of the problem \eqref{eqn: matrixsensingobjintro} where the measurement operator is the identity: 
	\begin{align}\label{eqn: matrixfactorizationintro}
		\min_{B\in \RR^{d\times k}}~f(B) = \|BB^\top -X \|_F^2,
	\end{align}
	In this case, a ravine takes a straightforward form. Namely, assume without loss of generality that
	$X=
	\begin{pmatrix}
		D & 0 \\
		0 & 0
	\end{pmatrix},
	$ where $D \in \RR^{r\times r}$ is a diagonal matrix with positive diagonal elements and write variable $B$ in block form $B = \begin{pmatrix}
		P^\top & Q^\top
	\end{pmatrix}^\top$, where $P \in \RR^{r\times k}$ and $Q \in \RR^{(d-r)\times k}$.
	Then we will show that the set of minimizers $S$ of \eqref{eqn: matrixfactorizationintro} and a ravine $\cM$  are simply 
	$$S = \left\{\begin{pmatrix}
		P\\
		Q
	\end{pmatrix} \colon PP^\top = D, Q=0 \right\},\qquad\cM = \left\{\begin{pmatrix}
		P\\
		Q
	\end{pmatrix} \colon PP^\top =D, PQ^\top = 0 \right\}.$$ 
	For example, in the rank one setting $X=e_1e_1^\top$ with $d=k=2$, the ravine and the solution set are diffeomorphic to a cylinder $\cM\simeq \mathbb{S}^1\times \mathbb{R}$ and a circle $S\simeq\mathbb{S}^1\times \{0\}$, respectively.

	\paragraph{Learning a single neuron.}
	As the second application of our techniques, we consider learning a single neuron in the overparametrized regime. That is, following \cite{xu2023over}, we focus on the problem 
	$$
	\min_{w}~ f(w) = \EE_{x\sim N(0,I)}\left[ \frac{1}{2}\left(\sum_{i=1}^{n}[w_i^\top x]_+ - [v^\top x]_+ \right)^2\right],
	$$
	where $w = (w_1^\top, w_2^\top,\ldots, w_n^\top)^\top \in \RR^{n\times d}$ denotes the decision variable. Variants of this problem have also been studied in~\cite{tian2017analytical,brutzkus2017globally, yehudai2020learning,du2017convolutional,wu2018no}. In the exact parameterization regime ($n=1$), gradient descent is shown to converge at a linear rate~\cite[Theorem 5.3]{yehudai2020learning}. In contrast, in the overparametrized regime ($n\geq 2$), the objective function $f(w)$ only grows cubicly away from the solution set, and gradient descent converges at a sublinear rate \cite{xu2023over}. 
	In this work, we focus on the simplest overparameterized setting $n=2$ and show that our results apply. In particular, gradient descent with an adaptive stepsize converges at a nearly linear rate. This result suggests that when training neural networks in the mildly overparameterized regime, the adaptive choice of the stepsize can exponentially speed up convergence.

	Figure~\ref{fig:neuralnetwork} compares the performance of $\gdpk$, gradient descent with constant stepsize $1.5$ ($\gd$), and gradient descent with Polyak stepsize ($\polyak$).
	As in the previous two examples, the plots show that $\gd$ and $\polyak$ converge sublinearly, while $\gdpk$ converges nearly linearly, and the long steps taken by $\gdpk$ grow exponentially with the iterations.
	
	\begin{figure}[H]
		\centering
		\begin{subfigure}[b]{0.325\textwidth}
			\centering
			\includegraphics[width=\textwidth]{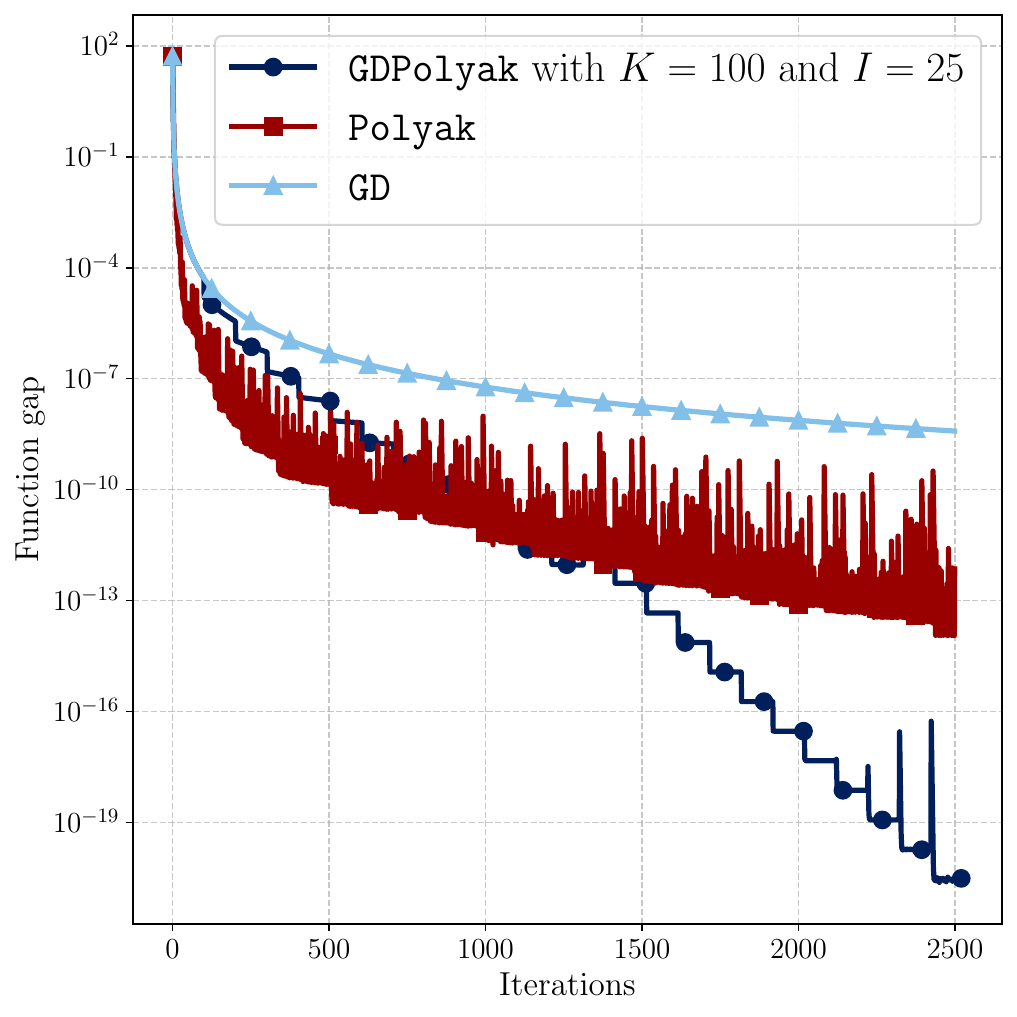} 
			\caption{Function gap}
		\end{subfigure}
		\hfill
		\begin{subfigure}[b]{0.325\textwidth}
			\centering
			\includegraphics[width=\textwidth]{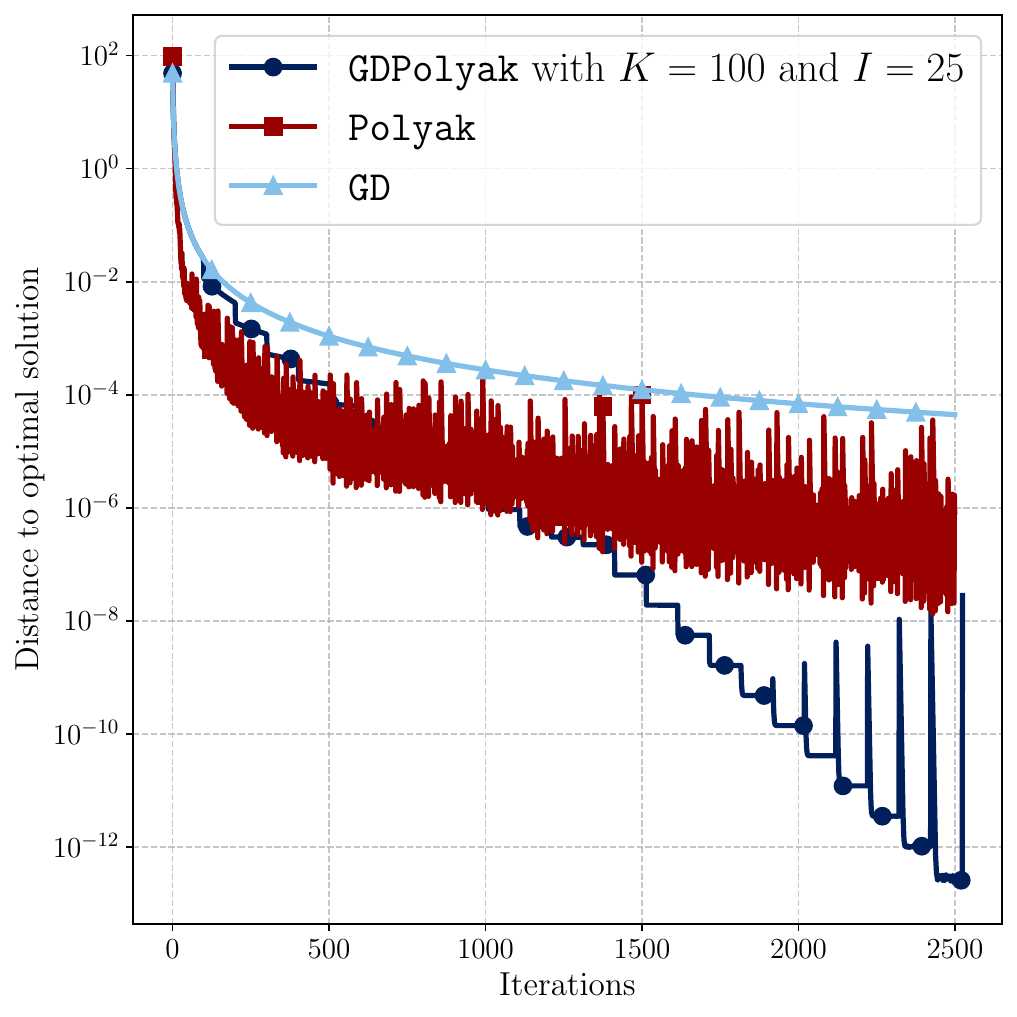} 
			\caption{Distance to optimal solution}
		\end{subfigure}
		\hfill
		\begin{subfigure}[b]{0.325\textwidth}
			\centering
			\includegraphics[width=\textwidth]{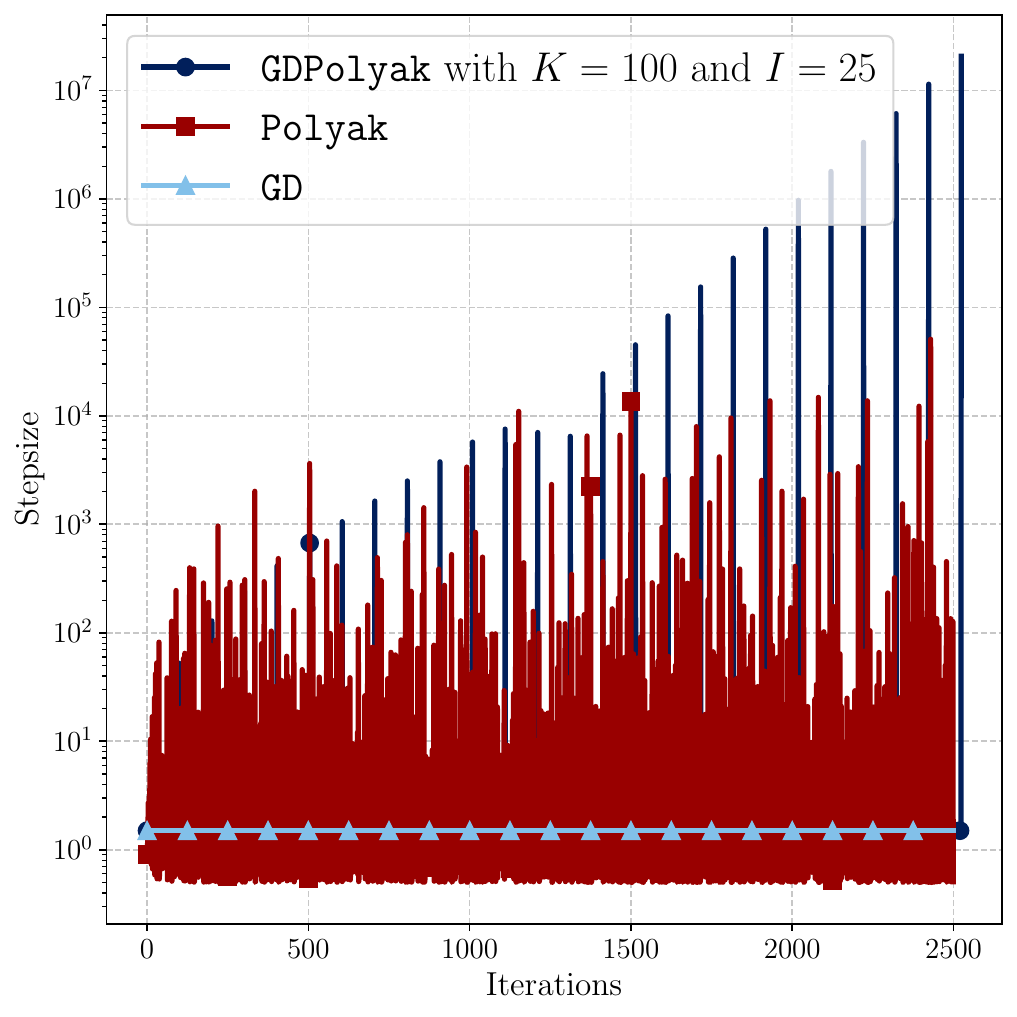} 
			\caption{Stepsize}
		\end{subfigure}
		
		\caption{Comparison of $\gdpk$ with $\gd$ and $\polyak$ for the problem of learning a single neuron in the overparameterized regime. In the experiment, we set $d=100$ and $n = 2$.
			For $\gdpk$, we run the method for $I=50$ epochs of size $K = 100$. 
			In each epoch, $\gdpk$ uses constant stepsize $1.5$. We note that since it is difficult to compute the exact distance to the set of minimizers of $S$ (defined explicitly in~\eqref{eq:solutionsetdu}), we instead compute a penalty, which can be shown to be proportional to $\dist(x_k, S)$. }
		\label{fig:neuralnetwork}
	\end{figure}

	\subsection{Related literature.}
	
	\paragraph{Ravines, partial smoothness, and local linear convergence.}
	The notion of a ravine introduced in this work nicely connects to nonsmooth optimization. Indeed, it has been classically known that critical points of typical nonsmooth functions lie on a certain manifold that captures the activity of the problem in the sense that critical points of slight linear tilts of the function do not leave the manifold. Such manifolds have been modeled in a variety of ways, including identifiable surfaces \cite{wright1993identifiable}, partial smoothness~\cite{lewis2002active}, $\mathcal{UV}$-structures \cite{lemarecha2000,mifflin2005algorithm}, $g\circ F$ decomposable functions \cite{shapiroreducible}, minimal identifiable sets \cite{drusvyatskiy2014optimality}, and active manifolds~\cite{davis2022proximal, davis2021active,davis2023asymptotic,davis2024local}. Roughly speaking, the active manifold at a minimizer is a manifold such that the function grows rapidly (linearly) away from it and varies smoothly along it. In this sense, the ravine can be understood as a higher-order active manifold for smooth optimization. The gradient method with adaptive stepsize conceptually resembles the more sophisticated Normal Tangent Descent (NTD) algorithm for nonsmooth optimization in~\cite{davis2024local}. When the nonsmooth function has exactly quadratic growth along the active manifold---a generic property for typical nonsmooth functions~\cite{drusvyatskiy2016generic}---NTD converges at a local nearly linear rate. A key feature of NTD is the switching between constant and diminishing step sizes, determined by the ratio of the distance to the active manifold and the distance to the minimizer.

	\paragraph{Overparameterized matrix sensing}
	Over the past decade, low-rank matrix sensing has been the subject of extensive study. The factorized approach, initially introduced by Burer and Monteiro in~\cite{burer2003nonlinear,burer2005local}, has been widely studied in recent work on matrix sensing (e.g., \cite{chi2019nonconvex,chen2015fast,bhojanapalli2016global,tu2016low,ge2017no,zhu2018global,li2019non, li2020nonconvex, charisopoulos2021low,zhang2021general,ma2023geometric, chandrasekher2022alternating}) and its ``population version," matrix factorization~\cite{ye2021global,josz2022nonsmooth}. When the exact target rank $k=r$ is known, the objective function \eqref{eqn: matrixsensingobjintro} has a benign optimization landscape \cite{ge2017no,zhu2018global} and simple gradient-based methods can find the ground-truth matrix with a statistical error that is minimax optimal up to log factors \cite{chen2015fast,bhojanapalli2016global,chi2019nonconvex} with a local linear rate of convergence. That being said, the ground-truth rank $r$ is usually \emph{unknown a priori}. To ensure recovery, one may choose a relatively large rank $k>r$. Recent studies have shown that in this rank-overparameterized setting, gradient descent with a constant stepsize can still find the ground truth up to a statistical error, but the local rate of convergence is only sublinear due to poor local geometry (lower growth) caused by overparametrization~\cite{zhuo2021computational,ding2021rank}. 
	
	Several papers have proposed ways to accelerate the convergence of gradient descent. For example, it has been shown that gradient descent with small initialization and early stopping only requires $O(\log(1/\epsilon))$ iterations to achieve an $\epsilon$ solution for overparameterized matrix factorization~\cite{jiang2022algorithmic, wind2023asymmetric} and matrix sensing~\cite{li2018algorithmic, stoger2021small,ma2023global,ding2022validation, jin2023understanding, xu2023power, soltanolkotabi2023implicit, maros2024decentralized}. That being said, the final error  depends on the initialization scale, and linear convergence is no longer observed after a certain number of iterations, which depends on the initialization size. The only exception in this line of work is~\cite{xiong2023over}, where one can observe indefinite linear convergence for the asymmetric variant of the problem and small initialization. The linear rate of convergence, however, depends on the initialization size and can be extremely slow when the initialization is small. A parallel line of work  leverages the specific structure of the problem in order to design methods that go beyond gradient descent. The work~\cite{ward2023convergence} shows that, by setting the stepsize according to the final accuracy $\epsilon$, alternating minimization only requires $O(\log(1/\epsilon))$ iterations to achieve an $\epsilon$-optimal solution for overparametrized matrix factorization. Interestingly, the works~\cite{zhang2022preconditioned,zhang2024fast} show that when using the preconditioner $B_t^\top B_t +\lambda_t I$ for suitable choices of $\lambda_t$, one can obtain local linear convergence for overparameterized matrix sensing. Moreover, the convergence rate is independent of the condition number of the ground truth matrix.
	
	\paragraph{Gradient descent with alternating short and long steps.}
	It has been observed that nonconstant stepsize schedules can improve the convergence of gradient descent. We review the literature on this subject, which has received renewed interest. In 1953, Young~\cite{young1953richardson} showed that one can achieve an optimal accelerated convergence rate for minimizing strongly convex quadratic functions when applying gradient descent with stepsizes dictated by the roots of Chebyshev polynomials. More recently, the work~\cite[Chap.\ 8]{altschuler2018greed} showed that alternating between short and long stepsizes achieves a faster rate for optimizing smooth, strongly convex functions. 
	In addition, \cite[Chap.\ 6]{altschuler2018greed} showed that certain random stepsizes achieve the optimal rate of convergence $O(\sqrt{\kappa}\log(1/\epsilon))$ for separable smooth strongly convex functions, where $\kappa$ is the condition number of the function.
	Another related work~\cite{oymak2021provable} shows a faster convergence rate when using nonconstant stepsize for smooth, strongly convex functions with bimodal structured Hessians. The paper \cite{kelner2022big} shows that for multi-scale strongly convex functions, which take the form of a sum of multiple non-interacting, smooth, and strongly convex functions, a recursive long-step short-step schedule for gradient descent can significantly improve dependence of the convergence rate on the condition number of the objective function. 
	
	The recent work \cite{grimmer2024provably} shows that alternating between short and long stepsizes also improves convergence rates for smooth convex functions, which are not necessarily strongly convex. 
	The subsequent works~\cite{altschuler2023acceleration, altschuler2023accelerationii} show that by applying the so-called ``silver stepsize schedule," gradient descent finds an $\epsilon$-minimizer after $O(\kappa^{\log_{1+\sqrt{2}}(2)}\log(1/\epsilon))$ and $O(\epsilon^{-\log_{1+\sqrt{2}}(2)})$ iterations for smooth strongly convex and smooth convex functions, respectively.  Finally, the subsequent work \cite{grimmer2024accelerated} shows that with a stepsize schedule similar to that of~\cite{altschuler2023acceleration, altschuler2023accelerationii}, gradient descent has the same convergence rate in function value and gradient norm for smooth convex functions.

	\section{Notation and preliminaries}
	Throughout, $\E$ will denote a Euclidean space equipped with an inner product $\langle \cdot,\cdot\rangle$ and the induced Euclidean norm $\|x\|=\sqrt{\langle x,x\rangle}$. The symbol $\mathbb{S}$ will denote the unit sphere in $\E$. For us the two main examples will be the Euclidean space of real $n$-dimensional vectors endowed with the usual dot-product and the $\ell_2$-norm and the space of real $n\times k$ matrices $\R^{n\times k}$ endowed with the trace inner product $\langle X,Y\rangle=\tr(X^\top Y)$ and the induced Frobenius norm $\|X\|_F=\sqrt{\tr(X^2)}$. The Euclidean subspace of $n\times n$ symmetric matrices will be written as $\mathcal{S}^{n}$, while the symbol $\mathcal{S}^{n}_+$ will denote the cone of $n\times n$ positive semidefinite matrices. The distance and the nearest point projection to any set $Q\subset\E$ are defined, respectively, as
	$$\dist(x,Q)=\inf_{y\in Q} \|y-x\|\qquad \textrm{and}\qquad\proj_{Q}(x)=\argmin_{y\in Q} \|y-x\|.$$
	
	We make heavy use of the $O(\cdot), \Theta(\cdot), o(\cdot)$ notation throughout this work. In particular, fix a subset $\Omega \subseteq \E$, a basepoint $\bar x \in \Omega$, and functions $g \colon \Omega \rightarrow \RR$ and $h \colon \Omega \rightarrow \R$. Then we say that $g(x) = O(h(x))$  for $x$ near  $\bar x$ if and only if on some neighborhood $U\subseteq \E$ of $\bar x$, there exists $c \geq 0$ such that $|g(x)| \leq c |h(x)|$ for all $x \in U\cap \Omega$. Next, we use the symbol $o(1)$ as $x$ tends to $\bar x$ to denote any function (positive or negative) that tends to $0$ as $x$ tends to $\bar x$. Finally, we use the symbol $\Theta(1)$ as $x$ tends to $\bar x$ to denote any function that is bounded between two fixed constants $0< c < C < \infty$ as $x$ tends to $\bar x$.
	For brevity, we often drop the phrase ``as $x$ tends to $\bar x$," when the limiting point is clear from context.

	Throughout the paper, we use the standard machinery of smooth manifolds as set out, for example, in the monographs \cite{lee2013smooth} and \cite{boumal2020introduction}. In particular, all smooth manifolds $\cM$ that we will consider are embedded in some Euclidean space $\E$, and we regard the tangent and normal spaces to $\cM$ as subspaces of $\E$. Thus a set $\cM\subset \E$ is a $C^p$-smooth manifold ($p\geq 1$) if for any point $x\in \cM$ there exists an open neighborhood $U\subset \E$ and a $C^p$-smooth map $F\colon U\to {\bf Y}$ mapping into some Euclidean space ${\bf Y}$ such  that the Jacobian $\nabla F(x)$ is surjective and equality $\cM\cap U=F^{-1}(0)$ holds. Then the tangent and normal spaces to $\cM$ at $x$ are defined simply as $T_{\cM}(x):=\Null(\nabla F(x))$ and $N_{\cM}(x):=(T_{\cM}(x))^{\perp}$, respectively. We note that on any neighborhood of a point $\bar x$ in a $C^2$-smooth manifold $\cM$, the projection $y=P_{\cM}(x)$ is characterized by the inclusion $x-y\in N_{\cM}(y).$ In particular, the function $P_{\cM}(x+t(y-x))=y$ is constant for all small $t$ and therefore the equality holds:
	\begin{equation}\label{eqn:proj_null_needed}
		\nabla P_{\cM}(x)(y-x)=0
	\end{equation}
	We will  use the following two elementary facts, which we record here for ease of reference. 
	\begin{proposition}[Range] \label{prop:range}
		For any smooth map $R\colon U\to \cM$ from an open set $U\subset\E$ to a smooth manifold $\cM$, the inclusion ${\rm Range}(\nabla R(x))\subset T_{\cM}(R(x))$ holds for all points $x\in U$.
	\end{proposition}

	\begin{proposition}[Tangents]\label{prop:tangent_close}
		Let $\cM$ be a $C^2$-smooth manifold around a point $\bar x$. Then the estimate, 
		$\dist\left(\tfrac{y-x}{\|y-x\|},T_{\cM}(x)\cap \mathbb{S}\right)=O(\|y-x\|)$,
		holds for all $x,y\in \cM$ near $\bar x$.
	\end{proposition}

	We will often encounter maps from a Euclidean space $\E$ to a submanifold $\cM\subset \E$. The following definition isolates a particularly nice type of such maps, called retractions. 
	\begin{definition}[Retraction] 
		{\rm  Let $\cM\subset \E$ be a $C^1$-smooth manifold around a point $\bar x$. Then a  {\em $C^p$-retraction onto $\cM$ around $\bar x$} is any 
			$C^p$-smooth map $R\colon U\to\cM$, defined on a neighborhood $U$ of $\bar x$,  satisfying $\nabla R(\bar x)=P_{T_{\cM}(\bar x)}$ and $R(x)=x$ for all $x\in \cM\cap U$.}
	\end{definition}

	Note that the definition requires the equality $\nabla R( x)=P_{T_{\cM}( x)}$ to hold only at $x=\bar x$. This is in contrast to the usual definition of a retraction in the optimization literature \cite[Section 3.6]{boumal2020introduction}, which requires this equality to hold for all $x\in \cM\cap \cU$. The extra flexibility, however, will be important in what follows. We will encounter two examples of retractions. First, if $\cM$ is a $C^{p+1}$-manifold, then the projection $P_{\cM}$ is a $C^p$-retraction around any point in $\cM$ \cite[Theorem 5.53]{boumal2020introduction}. Second, if $\cM\subset\E\times {\bf Y}$ can be written as a graph of some $C^p$-smooth map $F\colon \E\to{\bf Y}$ locally around $\bar x=(0,0)$ satisfying $F(0)=0$ and $\nabla F(0)=0$, then the map $R(v,u)=(v,F(v))$ is a $C^p$ retraction onto $\cM$ around $\bar x=0$. 
	
	All retractions can be understood as approximate projections in the following sense.
	\begin{proposition}[Retractions as approximate projections]\label{prop_dist_comp}
		Let $R(\cdot)$ be a $C^1$-retraction onto a $C^2$-smooth manifold $\cM$ at a point $\bar x$. Then the estimates hold:
		$$\|R(x)-P_{\cM}(x)\|=o(1)\cdot \dist(x,\cM)\qquad \textrm{as}~ x\to \bar x.$$
		In particular, we have $\|x-R(x)\|=\Theta(1)\cdot \dist(x,\cM)$ as $x$ tends to $\bar x$.
	\end{proposition}
	\begin{proof}
		Since the retraction $R$ is $C^1$-smooth near $\bar x$ we estimate:
		\begin{align*}
			P_{\cM}(x)-R(x)=R(P_{\cM}(x))-R(x)&=\nabla R(x)(P_{\cM}(x)-x)+ o(1)\|P_{\cM}(x)-x\|.
		\end{align*}
		Using continuity of $\nabla R$ and the equality $\nabla R(\bar x)=P_{T_{\cM}(\bar x)}$, we compute
		\begin{align*}\nabla R(x)(P_{\cM}(x)-x)&=\nabla R(\bar x)(P_{\cM}(x)-x)+\underbrace{(\nabla R(x)-\nabla R(\bar x))(P_{\cM}(x)-x)}_{=o(1)\cdot \dist(x,\cM)}.
		\end{align*}
		We now estimate the first term on the right side.
		To this end, observe the estimate $$\|\nabla R(\bar x)-\nabla P_{\cM}(x)\|_{\rm op}=\opnorm{\nabla P_{\cM}(\bar x) - \nabla P_{\cM}(x)} =  o(1).$$
		We therefore deduce 
		$$\nabla R(\bar x)(P_{\cM}(x)-x)=o(1)\cdot \dist(x,\cM)+\underbrace{\nabla P_{\cM}(x)(P_{\cM}(x)-x)}_{=0},$$
		where the expression in the underbrace follows from the inclusion $P_{\cM}(x)-x\in N_{\cM}(x)$.
		This completes the proof.
	\end{proof}
	
	\section{Ravines: definition, existence, and examples}
	The nullspace of the Hessian $\nabla^2 f(\bar x)$ at a minimizer $\bar x$ defines a set of problematic directions for the constant stepsize gradient method. The key idea of our work is to focus on certain distinguished manifolds $\cM$ that are tangent to the nullspace at $\bar x$. We will then decompose $f$ into its tangent and normal parts
	$$f(x)=f_{N}(x)+f_{T}(x),$$
	where we define $f_N(x):=f(x)-f(P_{\cM}(x))$ and $f_T(x):=f(P_{\cM}(x))$. The idea is to analyze the behavior of gradient methods using the distinctive properties of $f_T$ and $f_N$. In particular, we impose conditions on $\cM$, which ensure that $f_N$ is well-controlled by the square distance $\dist^2(x,\cM)$, which ensures that constant step gradient descent rapidly approaches $\cM$ up to a well-specified error. The following is the key definition of the paper.
	
	\begin{definition}[Ravine]\label{defn:ravine_g}{\rm 
			Consider a $C^{p}$-smooth function $f\colon\E\to\R$ $(p\geq 2)$ and let $\bar x$ be a minimizer of $f$. We say that a $C^{p}$-smooth manifold $\cM$ is a {\em $C^p$-ravine at $\bar x\in \cM$} if it satisfies the equality $T_{\cM}(\bar x)={\rm Null}(\nabla^2 f(\bar x))$ and there exists a $C^{p}$-smooth retraction $R\colon U\to \cM$ around $\bar x$ and a constant $C_{\rm lb}>0$ satisfying  
			\begin{equation}\label{eqn:def_prop_ravine}
				f(x)-f(R(x))\geq C_{\rm lb}\cdot \|x-R(x)\|^2\qquad \forall x\in U.
			\end{equation}
		}
	\end{definition}

	Note that since the function $f-f\circ R$ is $C^2$-smooth and is minimized by points in $\cM$ (due to \eqref{eqn:def_prop_ravine}), the reverse inequality $f(x)-f(R(x))=O(1)\cdot \|x-R(x)\|^2$ holds automatically near $\bar x$.
	As is readily seen from Figure~\ref{fig:combined_intro}, the ravine is geometrically distinctive because the function appears to have a valley along $\cM$. The ideal retraction would be the projection $P_{\cM}$ itself. Using more general retraction provides much greater flexibility. In particular (and surprisingly), any smooth function admits a ravine. This follows from the so-called Morse lemma with parameters \cite[Lemma C.6.1]{hormander2007analysis}. Indeed, there is one ravine---called the Morse ravine---that is canonically defined. We define it here in the case when $\bar x$ is zero for simplicity; the general case follows by considering the function $g(x)=f(x-\bar x)$. 
	
	\begin{definition}[Morse ravine]
		{\rm Consider a $C^p$-smooth function $f\colon\E\to\R$ {\color{blue} $(p\ge 2)$} and let $\bar x=0$ be a critical point of $f$. 
			Then the {\em Morse ravine of }$f$ {\em at} $\bar x$ is the set 
			$$\cM:=\{(u,v)\in \mathcal{T}\times \mathcal{T}^{\perp}: \nabla_v f(u,v)=0\},$$
			where $\cT={\rm Null}(\nabla^2 f(\bar x))$ denotes the nullspace of the Hessian.}
	\end{definition}
	
	In words, the Morse ravine is traced out by the critical points of the function $f(u,\cdot)$ as $u$ varies in $\cT$. It is straightforward to see that the manifold $\cM=\{(x,y): y=x^2\}$ in Figure~\ref{fig:combined_intro} is indeed a Morse ravine. 
	
	\subsection{The Morse ravine is a ravine}
	We will now show that the Morse ravine is indeed a ravine in the sense of Definition~\ref{defn:ravine_g}. We begin by showing that the Morse ravine is always a smooth manifold. Indeed, this follows directly from the implicit function theorem. 
	
	\begin{lemma}[Smoothness of the Morse ravine]\label{lem:implicit_graph_rep}
		Let $\cM$ be a Morse ravine of a $C^p$-smooth $(p\geq 2)$ function $f$ at a critical point $\bar x$. Then locally around $\bar x$, the Morse ravine $\cM$ coincides with the graph of some $C^{p-1}$ smooth map $v\colon \cT\to \cT^{\perp}$, and therefore $\cM$ is a $C^{p-1}$-smooth manifold around $\bar x$. Moreover equalities, $\nabla v(0)=0$ and $T_{\cM}(\bar x)=\cT$, hold.
	\end{lemma}
	\begin{proof}
		Define the $C^{p-1}$-smooth map $F(v,u)=\nabla_v f(u,v)$. Clearly, the Jacobian $\nabla_v F(0,0)=\nabla^2_{vv} f(\bar x)$ is nonsingular on $\cT^{\perp}$. Therefore, the implicit function theorem implies that there exist open neighborhoods $V$ containing $v=0$ and $U$ containing $u=0$ such that for each $u\in U$ there is a unique point $v(u)\in V$ satisfying $F(u,v(u))=0$. Moreover, the implicit map $v(\cdot)$ thus defined is $C^{p-1}$ smooth and satisfies 
		$$\nabla v(0)=-\nabla_v F(0,0)^{-1}\nabla_u F(0,0)=-\nabla^2_{vv} f(\bar x)^{-1}\nabla^2_{vu} f(\bar x)=0.$$
		The last equality follows from the fact that in the coordinate system $\mathcal{T}\times \mathcal{T}^{\perp}$, the block $\nabla^2_{vu} f(\bar x)=0$ is zero. 
		In particular, we see that $\cM$ coincides with the graph of $v(\cdot)$ locally around $\bar x$. Consequently, $\cM$ is a $C^{p-1}$ smooth manifold and its tangent space at $\bar x$ is the graph of the trivial linear map $\nabla v(0) \colon u \rightarrow 0$, which is $\mathcal{T}\times\{0\}$. 
	\end{proof}
	
	The map $v(\cdot)$ in Lemma~\ref{lem:implicit_graph_rep} will be called the {\em graphical representation of $\cM$}. Next, it remains to establish the defining property \eqref{eqn:def_prop_ravine} for the retraction $$R(u,v)=(u,v(u))$$ at $\bar x$. This follows directly from the Morse lemma with parameters~\cite[Lemma C.6.1]{hormander2007analysis}.
	
	\begin{lemma}[Morse lemma with parameters]\label{lem: Morse}
		Let $\cM$ be the Morse ravine of a $C^{p}$-smooth function $f$ {\color{blue}$(p \ge 2)$} at a minimizer $\bar x=0$, and let $v\colon \cT\to \cT^{\perp}$ be a graphical representation of $\cM$. Then the equation holds:
		$$
		f(u,v) = f(u,v(u)) + \tfrac{1}{2}\dotp{\nabla^2_{vv} f(\bar x)w,w},
		$$ 
		where $w = v - v(u) + O(\|v - v(u)\| (\|u\|+ \|v\|))$ is a $C^{p-2}$-smooth function of $(u,v)$ at $(0,0)$.
	\end{lemma}

	The fact that the Morse ravine is a ravine is now immediate.
	
	\begin{corollary}[Existence of Morse ravine]\label{cor:morse_ravin_exist}
		The Morse ravine of a $C^p$-smooth function $f$ {\color{blue} $(p\ge 2)$} at a  minimizer $\bar x=0$ is a $C^{p-1}$-ravine of $f$ at $\bar x$.
	\end{corollary}
	\begin{proof}
		Lemma~\ref{lem:implicit_graph_rep} showed that the Morse ravine $\cM$ of $f$ at $\bar x=0$ is a $C^{p-1}$-smooth manifold around $\bar x$ with $T_{\cM}(\bar x)={\rm Null}(\nabla^2 f(\bar x))$. We let $v(\cdot)$ be the graphical representation of $\cM$ and define the map $R(u,v)=(u,v(u))$. Clearly $R$ is a $C^{p-1}$ retraction onto $\cM$ at $\bar x$.  Setting $\Delta=x-R(x)$, Lemma~\ref{lem: Morse}  implies 
		$$f(x)-f(R(x))=\tfrac{1}{2}\dotp{\nabla^2_{vv} f(\bar x)\Delta,\Delta} +o(1)\cdot \|\Delta\|^2.$$
		Note that $\Delta$ lies in $\cT^{\perp}$.
		Taking into account that $\nabla^2_{vv} f(\bar x)$ is nonsingular on $\mathcal{T}^{\perp}={\rm Range}(\nabla^2 f(\bar x))$, we deduce the estimates $c_1\|\Delta\|^2\leq \tfrac{1}{2}\dotp{\nabla^2_{vv} f(\bar x)\Delta,\Delta}\leq c_2\|\Delta\|^2$ for some constants $c_1,c_2>0$. Thus the proof is complete.
	\end{proof}
	
	It is worth noting that ravines are not unique, and the Morse ravine is just one possibility. For example, consider the function 
	$$f(z)=(\|z\|-1)^2+\left\|\tfrac{z}{\|z\|}-e_2\right\|^4,$$
	with $z\in\R^2$ and $e_2=(0,1)$. One can show that the unit circle $\cM_0=\{z: \|z\|=1\}$ is a ravine for $f$ at $\bar x=(0,1)$. On the other hand, a quick computation shows that the Morse ravine $\cM_1$ consists of all points $z=(x,y)$ satisfying the nonlinear equation:
	$$\|z\|^6y -y \|z\|^5 -x^2\|z\|^3+ (yx^2+2x^2y)\|z\|^2+(x^4-2x^2y^2)\|z\|=2x^4y.$$
	Indeed, the two sets $\cM_0$ and $\cM_1$ intersect only at $\bar x$; see Figure~\ref{fig:comp_ravine} for an illustration. Fortunately, this nonuniqueness will cause no issues for our adaptive gradient descent algorithm, since we will only utilize the existence of a ravine.

	\begin{figure}[H]
		\centering
		\begin{subfigure}[b]{0.55\textwidth}
			\includegraphics[width=\textwidth]{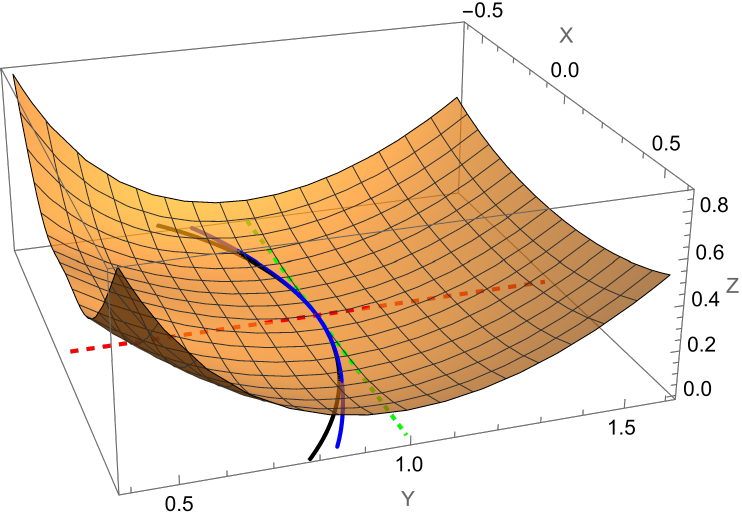}
			\caption{ravine $\cM_0$ (black) differs from the Morse ravine $\cM_1$ (blue); the tangent and normal space are depicted in green and red, respectively.}
			
		\end{subfigure}
		\hfill
		\begin{subfigure}[b]{0.4\textwidth}
			\includegraphics[width=\textwidth]{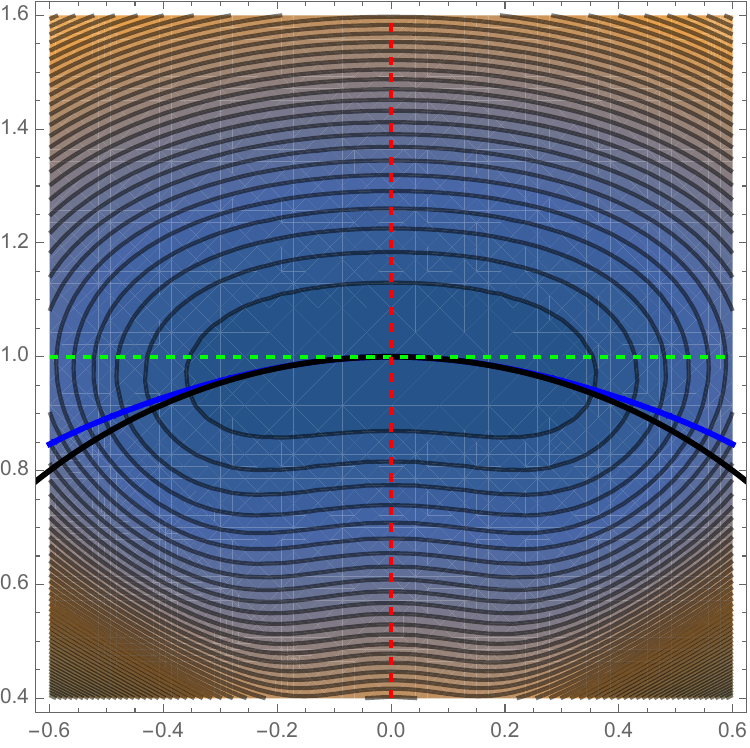}
			\caption{Contour plot}
			\begin{minipage}{.1cm}
				\vfill
			\end{minipage}
		\end{subfigure}
			\caption{The function $f(z)=(\|z\|-1)^2+\left\|\tfrac{z}{\|z\|}-e_2\right\|^4$ for $z\in \R^2$.}
			\label{fig:comp_ravine}
		\end{figure}
		\subsection{Constant rank, uniform ravines, and lower growth}

		When $f$  has multiple minimizers a certain technicality arises since a ravine at a given minimizer may not be a ravine at nearby minimizers. 
		Indeed, this makes the choice of the decomposition $f(x) = f_N(x) + f_T(x)$ ambiguous.
		Nonuniformity of the ravine already arises for the simple quartic function $f(x,y) = x^2y^2$. 
		At the origin, it has a unique ravine -- $\R^2$ -- while at all nearby minimizers, its ravines are one-dimensional. 
		Luckily, we can prevent this possibility by a simple assumption: $\nabla^2 f$ has a constant rank along the solution set.

		\begin{proposition}[Uniform ravines]\label{prop: uniform_ravine}
			Let $\cM$ be a $C^{p}$-ravine $(p\ge 2)$ of a function $f$ at a minimizer $\bar x$ and let $S$ be the set of minimizers of $f$. Then all points $x\in S$ sufficiently close to $\bar x$ lie in $\cM$ and satisfy the inclusion:
			\begin{equation}\label{eqn:null_tang}
				{\rm Null}(\nabla^2 f(x))\subset T_{\cM}(x).
			\end{equation}
			In particular, if the Hessian $\nabla^2 f$ has constant rank on $S$ near $\bar x$, then $\cM$ is a $C^{p}$-ravine of $f$ at all points $x\in S$ near $\bar x$.
		\end{proposition}
		\begin{proof}
			Since $\cM$ is a ravine, all $x\in S$ near $\cM$ satisfy $\Theta(1)\cdot \|x-R(x)\|^2=f^* -f(R(x)) \le 0$. Therefore, we deduce that $x$ coincides with $R(x)$ and hence lies in $\cM$, as claimed. 
			
			Next, by definition, a ravine of $f$ at a minimizer $\bar x$ satisfies the following equality: \break ${\rm Null}(\nabla^2 f(\bar x))= T_{\cM}(\bar x)$.  Suppose now, for the sake of contradiction, that there exists a point $x$ close to $\bar x$ and a unit vector $v\in {\rm Null}(\nabla^2 f(x))$ satisfying $v\notin T_{\cM}(x)$. Then the estimate $\dist(x+tv,\cM)=\Theta(1)\cdot t$ holds. Therefore, we deduce:
			\begin{align}
				o(t^2)=f(x+tv)-f(x)&\geq f(x+tv)-f(R(x+tv))\label{eqn:b1}\\
				&\geq \Theta(1)\cdot\|(x+tv)-R(x+tv)\|^2\label{eqn:b2}\\
				&=\Theta(1)\cdot \dist^2(x+tv,\cM)\label{eqn:b3}\\
				&=\Theta(1)\cdot t^2,
			\end{align}
			where \eqref{eqn:b1} follows from the assumption $\nabla^2 f(x)v=0$, the inequality \eqref{eqn:b2} uses the defining property of \eqref{eqn:def_prop_ravine}, and \eqref{eqn:b3} follows from Proposition~\ref{prop_dist_comp}. Dividing through by $t^2$ and letting $t$ tend to zero yields a contradiction. Thus, the claimed inclusion~\eqref{eqn:null_tang} holds. Finally, suppose that the Hessian $\nabla^2 f$ has a constant rank on $S$ near $\bar x$. Then equalities $\rank \nabla f^2(x)=\rank \nabla f^2(\bar x)=\dim T_{\cM}(\bar x)$ hold for all $x\in S$ near $\bar x$. Consequently, the inclusion \eqref{eqn:null_tang} holds as an equality for all  $x\in S$ near $\bar x$.
		\end{proof}
		
		We end the section by recording an intriguing property that further highlights the key role that the ravine plays on the local behavior of the function: the lower growth of $f$ along the ravine $\cM$ implies the same lower growth for $f$ outside of the ravine. Throughout the rest of the paper, we denote the minimum value of $f$ by $f^*$. 
		
		\begin{lemma}[Extending growth]
			Suppose that a function $f\colon\E\to\R$ admits a $C^2$ ravine at a minimizer $\bar x$. Let $S$ be the set of minimizers of $f$ and suppose that there exist constants 
			$\delta,\dlb > 0$ and $p>2$ such that the estimate 
			\begin{equation}\label{eqn:growth_on_man}
				f(y)  - f^*\ge \dlb\cdot\dist^p(y,S) \qquad \textrm{holds for all }y\in B_{\delta}(\bar x)\cap \cM.
			\end{equation}
			Then there exist constants $\delta',\dlb' > 0$ satisfying
			\begin{align*}
				f(x)  - f^*\ge \dlb'\cdot\dist^p(x,S) \qquad \textrm{holds for all }x\in B_{\delta'}(\bar x).
			\end{align*}	
		\end{lemma}
		\begin{proof}
			For any $x$ sufficiently close to $\bar x$ we compute:
			\begin{align*}
				\dist^p(x,S)&\leq 2^{p-1}\cdot \|x-R(x)\|^p+2^{p-1}\cdot \dist^p(R(x),S)\\
				&\leq o(1) (f(x)-f(R(x))) + \tfrac{2^{p-1}}{\dlb}\cdot (f(R(x))  - f^*)\\
				&\le  \tfrac{2^{p-1}}{\dlb}\cdot (f(x)  - f^* ),
			\end{align*}
			where the first inequality follows from the triangle inequality and the estimate $\frac{1}{2^{p-1}}(a+b)^p\leq a^p+b^p$, and the second inequality follows from the assumption \eqref{eqn:growth_on_man} and the defining property of the ravine \eqref{eqn:def_prop_ravine}.
		\end{proof}

		\section{Ravines: analytic properties}\label{sec: ravine}
		In this section, we derive several useful properties of ravines, with a particular view of how ravines influence algorithmic behavior.

		\subsection{Gradient control in tangent and normal directions.}
		Recall that the key property of ravine \eqref{eqn:def_prop_ravine} implies an upper bound on the Function gap $|f(x)-f(R(x))|= O(\dist^2(x,\cM))$. In this section, we discuss the consequences of this condition on first-order behavior, namely on the gradients $\nabla f$ and $\nabla(f\circ R)$. Estimating the deviations between these two gradients in tangent and normal directions will allow us to connect gradient-based methods for $f$ to gradient-based methods for the function $f\circ P_\cM$. In later sections, this connection will play a central role in designing and analyzing algorithms.

		We begin with the following lemma by ``formally differentiating'' \eqref{eqn:def_prop_ravine} and deducing an analogous bound on the gradients, $\|\nabla f(x)-\nabla (f\circ R)(x))\|= O(\dist(x,\cM))$.
		
		\begin{lemma}[Gradient control]\label{lem:grad_are_close}
			Let $\cM$ be a $C^2$-ravine for a function $f$ at $\bar x$. Then the estimate holds:
			$$\|\nabla f(x)-\nabla (f\circ R)(x)\|=O(\dist(x,\cM))\qquad \textrm{ as }x\textrm{ tends to }\bar x.$$  
		\end{lemma}
		\begin{proof}
			Proposition~\ref{prop_dist_comp} implies that on any sufficiently small neighborhood of $\bar x$ the estimate
			{\color{blue} 	\begin{equation}\label{eqn:sandwich}
					\|x-R(x)\|^2\leq c\cdot\dist^2(x,\cM),
				\end{equation}
				holds for some constants $c>0$.}  Define the function 
			{\color{blue} $$g(x):=f(x)-f(R(x)).$$ 
				The definition of the ravine along with \eqref{eqn:sandwich} imply the estimate
				$$0\leq g(x)\leq C \|x-R(x)\|^2\leq c\cdot C \cdot \dist^2(x,\cM),$$
				holds for some constant $C > 0$ and all $x$ near $\bar x$.} Since $\nabla g$ is $L$-Lipschitz continuous around $\bar x$ for some constant $L$, the standard descent lemma implies
			$$0\leq g\left(x-\tfrac{1}{L}\nabla g(x)\right)\leq g(x)-\tfrac{1}{2L}\|\nabla g(x)\|^2,$$
			for all $x$ near $\bar x$. Thus we conclude $\|\nabla g(x)\|^2=O(g(x))$. Taking into account that $g(x)=O(\dist^2(x,\cM))$, we deduce $\|\nabla g(x)\|=O(\dist(x,\cM))$. {\color{blue} This completes the proof since $\nabla g(x) = \nabla f(x)-\nabla (f\circ R)(x)$.}
		\end{proof}

		\subsection{Orthogonal decomposition of the function}
		In this section, we will pass from the retraction $R(\cdot)$ to the true projection $P_{\cM}(\bar x)$. The reason is simply that the analysis of algorithms typically relies on distances and projections, with retractions serving only as approximations of the true projection. To this end, given a $C^2$-ravine $\cM$ for a function $f$ at a point $\bar x$, we decompose $f$ into  {\em tangent} and {\em normal parts}:
		$$f(x)=f_{N}(x)+f_{T}(x),$$
		where we define $f_N(x):=f(x)-f(P_{\cM}(x))$ and $f_T(x):=f(P_{\cM}(x))$. The following theorem is the main result of the section. In short, it shows that the normal part $f_{N}(x)$ satisfies several desirable regularity conditions relative to the manifold $\cM$, such as quadratic growth and gradient aiming towards $\cM$, but only up to an error term that depends on the gradient of the tangent part $\|\nabla f_T\|$. The precise form of the error terms will be critically important in the sequel.
		
		\begin{theorem}[Key properties of the normal part]\label{thm:key_properties}
			Let $\cM$ be a $C^2$-ravine of a function $f$ at a minimizer $\bar x$ and suppose that $P_{\cM}$ is $C^2$-smooth near $\bar x$. Then, the following properties hold for all $x$ near $\bar x$.
			\begin{enumerate}
				\item {\bf (Projected gradient)}\label{it:strong_a_in_thm} $\qquad\|P_{T_{\cM}(P_{\cM}(x))} \nabla f_{N}(x)\| = o(1)\cdot \dist(x,\cM).$
				\item {\bf (Aiming)}\label{it:aiming} There exists a constant $C_{{\rm a}}>0$ such that the estimates hold:
				\begin{align}
					\langle \nabla f(x),x-R(x)\rangle&\geq C_{\rm a}\cdot \|x-R(x)\|^2,\label{eqn:aimin}\\
					\langle \nabla f_N(x),x-P_{\cM}(x)\rangle&\geq C_{\rm a}\cdot\dist^2(x,\cM)+o(1)\cdot\|\nabla f_T(x)\|\cdot \dist(x,\cM).\label{eqn:aim_main}
				\end{align}
				\item {\bf (Growth)}\label{it:grad_val_bound} 
				\begin{align}
					\|\nabla f_N(x)\|&= O(\dist(x,\cM))+o(1)\cdot\|\nabla f_T(x)\|,\label{eqn:grad_bound_norm}\\
					|f_N(x)|&= O(\dist^2(x,\cM))+o(1)\cdot\|\nabla f_T(x)\|\cdot \dist(x,\cM).\label{eqn:val_bound}
				\end{align}
			\end{enumerate}
		\end{theorem}

		A few comments are in order. The projected gradient expression shows that the gradient of $\nabla f_{N}(x)$ is small in tangent directions $P_{T_{\cM}(P_{\cM}(x))}$. 
		The companion growth bound \eqref{eqn:grad_bound_norm} shows that the gradient $\nabla f_{N}(x)$ in the normal directions $P_{N_{\cM}(P_{\cM}(x))}$ is tightly controlled by $\dist(x,\cM)$ and the gradient of the tangent part $\|\nabla f_T(x)\|$. The aiming condition \eqref{eqn:aim_main} shows that the negative gradient $-\nabla f_{N}(x)$ points towards the projected point $P_{\cM}(x)$ up to an error controlled by the distance $\dist(x,\cM)$ and the gradient of the tangent part $\|\nabla f_{T}(x)\|$.

		\begin{proof}We now establish the claimed properties in  Theorem~\ref{thm:key_properties} in order of appearence. 
			
			\paragraph{\normalfont{\em Item \eqref{it:strong_a_in_thm} (Projected Gradient):}}
			For the ease of notation, we set $y:=P_\cM(x)$  and $\Delta:=x - y$. 
			{\color{blue} We first note that 
				\begin{align}
					\|P_{T_{\cM}(y)} \nabla f_{N}(x)\| &=  \norm{P_{T_\cM(y)}(\nabla f(x) - \nabla P_\cM(x)^\top \nabla f(y))} \notag\\
					&\le   \norm{P_{T_\cM(y)}(\nabla f(x) -   \nabla P_\cM(y)^\top \nabla f(y))} +   \norm{P_{T_\cM(y)}(\nabla P_\cM(x) - \nabla P_\cM(y))^\top \nabla f(y))}\notag\\
					&= \norm{P_{T_\cM(y)}(\nabla f(x) -  \nabla f(y))} + o(1) \cdot \| \Delta\| \label{eq:stronga1_eq1}\\
					&= \norm{P_{T_\cM(y)}\int_{0}^{1} (\nabla^2 f(y+ t\Delta)) \Delta ~dt} + o(1)\cdot\|\Delta\|\notag\\
					&\leq \left\|P_{T_\cM(y)}\int_{0}^{1} \nabla^2 f(y+ t\Delta) ~dt\right\|_{\rm op}\cdot\|\Delta\| + o(1)\cdot\|\Delta\|, \label{eq:stronga1}
				\end{align}
				where the equality~\eqref{eq:stronga1_eq1} follows from the local Lipschitz continuity of $\nabla P_\cM$ and $\|\nabla f(y)\| = o(1)$.
			}
			Next, using the triangle inequality, we deduce
			\begin{align}
				\left\|P_{T_\cM(y)}\int_{0}^{1} \nabla^2 f(y+ t\Delta) ~dt\right\|_{\rm op}&\leq\underbrace{\left\|P_{T_\cM(\bar x)}\int_{0}^{1} \nabla^2 f(y+ t\Delta)~ dt\right\|_{\rm op}}_{{=o(1)}}\notag\\
				&\quad+\underbrace{\| P_{T_{\cM}(y)}-P_{T_\cM(\bar x)}\|_{\rm op}}_{=o(1)}\cdot \underbrace{\left\|\int_{0}^{1} (\nabla^2 f(y+ t\Delta))~ dt\right\|_{\rm op}}_{=O(1)},\label{eqn:esta2}
			\end{align}
			where the under-brace expressions, in order of appearance, follow from the fact that the range of $\nabla^2 f(\bar x)$ is orthogonal to $T_{\cM}(\bar x)$, Lipschitz continuity of $\nabla P_{\cM}(\cdot)$, and boundedness of the Hessian $\nabla^2 f(x)$ near $\bar x$. Combining \eqref{eq:stronga1} and \eqref{eqn:esta2} completes the proof.

			\paragraph{\normalfont{\em Item \eqref{it:aiming} (Aiming):}}
			Since $\bar x$ is a minimizer of $f$, for every $c>0$ there exists a convex neighborhood $U$ of $\bar x$ such that the estimate $\lambda_{\min}(\nabla^2 f(x))\geq -c$ holds for every $x\in U$. Consequently, for all  $x$ sufficiently close to $\bar x$, the estimate holds:
			$$f(R(x))\geq f(x)+\langle \nabla f(x),R(x)-x\rangle-c \|x-R(x)\|^2.$$
			Rearranging and using the defining property of the ravine \eqref{eqn:def_prop_ravine}, we therefore conclude
			\begin{align*}
				\langle \nabla f(x),x-R(x)\rangle&\geq f(x)-f(R(x))-c \|x-R(x)\|^2\\
				&\geq (C_{\rm lb}-c)\|x-R(x)\|^2.
			\end{align*}
			Since $c>0$ is arbitrary, the estimate \eqref{eqn:aimin} follows. 
			
			Next, fix a point $x$ near $\bar x$ and set $y:=P_{\cM}(x)$. Then we compute
			\begin{align}
				\langle \nabla f_N(x), x-P_{\cM}(x)\rangle&=\langle \nabla f(x)-\nabla P_{\cM}(x)^\top \nabla f(y), x-y\rangle\notag\\
				&=\langle \nabla f(x),x-y\rangle-\langle   \nabla f(y), {\nabla P_{\cM}(x) (x-y)}\rangle,\label{pass_tproj}
			\end{align}
			Using  \eqref{eqn:proj_null_needed}, we deduce that the last term in \eqref{pass_tproj} is zero. Therefore, we conclude
			\begin{align*}
				\langle \nabla f_N(x), x-P_{\cM}(x)\rangle&=\langle \nabla f(x),x-y\rangle\\
				&=\langle \nabla f(x),x-R(x)\rangle+\langle \nabla f(x),R(x)-P_{\cM}(x)\rangle\\
				&= \Theta(1)\cdot \underbrace{\|x-R(x)\|^2}_{=\theta(1)\cdot\dist^2(x,\cM))}+\langle \nabla f(x),R(x)-P_{\cM}(x)\rangle,
			\end{align*}
			where the last equality follows from \eqref{eqn:aimin} and the expression in the under-brace follows from Proposition~\ref{prop_dist_comp}.
			Next, we rewrite the remainder term:
			$$\langle \nabla f(x),R(x)-P_{\cM}(x)\rangle=\underbrace{\langle \nabla (f\circ P_{\cM})(x),R(x)-P_{\cM}(x)  \rangle}_{=o(1)\cdot\|\nabla (f\circ P_{\cM})(x)\|\cdot \dist(x,\cM)}+\langle \nabla f(x)-\nabla (f\circ P_{\cM})(x),R(x)-P_{\cM}(x)\rangle,$$
			where the expression in the under-brace follows from Proposition~\ref{prop_dist_comp}.
			Using Proposition~\ref{prop:tangent_close}, we may write $\frac{R(x)-P_{\cM}(x)}{\|R(x)-P_{\cM}(x)\|}=\omega+O(\|R(x)-P_{\cM}(x)\|)$ for some unit tangent vector $\omega\in T_{\cM}(P_{\cM}(x))$. Consequently, using \eqref{it:strong_a_in_thm} we deduce the estimate
			$$\langle \nabla f(x)-\nabla (f\circ P_{\cM})(x),R(x)-P_{\cM}(x)\rangle\leq o(1)\cdot\dist^2(x,\cM),$$
			thereby completing the proof.

			\paragraph{\normalfont{\em Item \eqref{it:grad_val_bound}  (Growth):}}
			We first prove \eqref{eqn:grad_bound_norm}. To simplify notation, set $y=R(x)$. The triangle inequality and Lemma~\ref{lem:grad_are_close} then directly imply
			\begin{align}
				\|\nabla f_N(x)\| &\leq \|\nabla f_N(x)-(\nabla f(x)-\nabla (f\circ R)(x))\|+\|\nabla f(x)-\nabla (f\circ R)(x)\|\notag \\
				&\leq \|\nabla( f\circ P_{\cM})(x)-\nabla(f\circ R)(x)\|+O(\dist(x,\cM))\notag\\
				&= \|\nabla f_T(x)-\nabla R(x)^\top\nabla f(y)\|+O(\dist(x,\cM)).\label{eqn:grad_bound1}
			\end{align}
			Next,  we compute
			\begin{align}
				\nabla R(x)^\top\nabla f(y)&=\nabla R(x)^\top (P_{N_{\cM}(y)}\nabla f(y)+P_{T_{\cM}(y)}\nabla f(y))\notag\\
				&\leq \underbrace{\nabla R(x)^\top P_{N_{\cM}(y)}\nabla f(y)}_{=0}+\nabla R(x)^\top P_{T_{\cM}(y)}\nabla f(y),\label{eqn:grad_bound2}
			\end{align}
			where the underbrace follows from the inclusion ${\rm Range}(\nabla R(x))\subseteq T_{\cM}(y)$. 
			Lipschitz continuity of $\nabla f$ and $\nabla P_{\cM}$ yield the estimate for the remainder: 
			\begin{align}
				\nabla R(x)^\top P_{T_{\cM}(y)}\nabla f(y)&=\nabla R(x)^\top P_{T_{\cM}(y)}\nabla f(P_{\cM}(x))+O(\|y-P_{\cM}(x)\|)\notag\\
				&=\nabla R(x)^\top \nabla P_{\cM}(x)^\top\nabla f(P_{\cM}(x))+O(\dist(x,\cM))\notag\\
				&=\nabla R(x)^\top \nabla f_T(x)+O(\dist(x,\cM))\notag\\
				&= P_{T_\cM(P_{\cM}(x))} \nabla f_T(x)+O(\dist(x,\cM))\notag\\
				&\quad+ O(\underbrace{\|\nabla R(x)-P_{T_\cM(P_{\cM}(x))}\|_{\rm op}}_{=o(1)}\cdot \|\nabla f_T(x)\|).\label{eqn:grad_bound3}
			\end{align}
			{\color{blue} Finally we compute
				\begin{align}
					\|P_{T_\cM(P_{\cM}(x))} \nabla f_T(x)-\nabla f_T(x)\|&= \|P_{N_\cM(P_{\cM}(x))} \nabla P_\cM(x)^\top \nabla f(y)\| \notag \\
					&= \|P_{N_\cM(P_{\cM}(x))} \nabla P_\cM(x)^\top P_{T_\cM(P_{\cM}(x))}\nabla f(P_{\cM}(x))\|\notag \\
					&= O(\dist(x,\cM)) \|P_{T_\cM(P_{\cM}(x))}\nabla f(P_{\cM}(x))\| \notag\\
					&= o(1) \|P_{T_\cM(P_{\cM}(x))}\nabla f(P_{\cM}(x))\| \notag \\
					&= o(1) \|\nabla f_T(P_{\cM}(x))\| \notag \\
					&=o(1)\cdot \|\nabla f_T(x)\| + O(\dist(x,\cM)),\label{eqn:grad_bound4}
				\end{align}
				where the first equality follows from the definition of $f_T$, the second equality follows from Proposition~1, the third equality follows from Lipschitz continuity of $\nabla P_\cM$ and $\nabla P_\cM(P_\cM(x)) = P_{T_\cM(P_{\cM}(x))}$, the fourth equality follows from the fact that $\dist(x,\cM)$ tends to zero as $x$ tends to $\bar x$, the fifth equality follows from the definition of $f_T$, and the final equality follows from the Lipschitz continuity of $\nabla f_T$.
			}
			Combining \eqref{eqn:grad_bound1}, \eqref{eqn:grad_bound2}, \eqref{eqn:grad_bound3}, and \eqref{eqn:grad_bound4} completes the proof of \eqref{eqn:grad_bound_norm}.

			Next, we establish \eqref{eqn:val_bound}. To this end, the definition of the ravine along with Proposition~\ref{prop_dist_comp} implies $|f(x)-f(R(x))|=O(\dist^2(x,\cM))$, and therefore we have
			\begin{align}
				|f_N(x)|&\le |f(R(x))-f(P_{\cM}(x))|+ |f(x)-f(R(x))| \notag\\
				&=|f(R(x))-f(P_{\cM}(x))| +O(\dist^2(x,\cM)).\label{eqn:growth_norm1}
			\end{align}
			Continuing, the smoothness of $f$ implies 
			\begin{equation}\label{eqn:growth_norm2}
				f(R(x))-f(P_{\cM}(x))=\langle \nabla f(P_{\cM}(x)), R(x)-P_{\cM}(x)\rangle+O(\|R(x)-P_{\cM}(x)\|^2).
			\end{equation}
			Using Proposition~\ref{prop:tangent_close},  we may write 
			$\frac{R(x)-P_{\cM}(x)}{\|R(x)-P_{\cM}(x)\|}=\omega+O(\|R(x)-P_{\cM}(x)\|)$ for some unit tangent vector $\omega\in T_{\cM}(P_{\cM}(x))$. Therefore, we compute
			\begin{align}
				\langle \nabla f(P_{\cM}(x)), R(x)-P_{\cM}(x)\rangle&=\langle \nabla f(P_{\cM}(x)),\omega\rangle\cdot \|R(x)-P_{\cM}(x)\|+o(1)\cdot \|R(x)-P_{\cM}(x)\|^2\notag\\
				&=\langle \nabla P_{\cM}(P_{\cM}(x))\nabla f(P_{\cM}(x)),\omega\rangle\cdot \|R(x)-P_{\cM}(x)\|\notag\\
				&~ +o(1)\cdot \|R(x)-P_{\cM}(x)\|^2\notag\\
				&=\langle \nabla f_T(x),\omega\rangle\cdot \|R(x)-P_{\cM}(x)\|\notag\\
				&~ +o(1)\cdot\dist(x,\cM)\cdot \|R(x)-P_{\cM}(x)\|+o(1)\cdot\|R(x)-P_{\cM}(x)\|^2\notag\\
				&=o(1)\cdot \|\nabla f_T(x)\|\cdot \dist(x,\cM)+o(1)\cdot\dist^2(x,\cM),\label{eqn:growth_norm3}
			\end{align}
			where the third equality follows Lipschitz continuity of $\nabla P_{\cM}$ and the fourth equality follows from Proposition~\ref{prop_dist_comp}.
			Combining \eqref{eqn:growth_norm1}, \eqref{eqn:growth_norm2}, and \eqref{eqn:growth_norm3} completes the proof of \eqref{eqn:val_bound}. 	
		\end{proof}

		The previous theorem shows that the normal function $f_N$ behaves very well relative to $\cM$ up to an error that is controlled by the gradient of the tangent part $\|\nabla f_T\|$. This regularity will allow us to argue that the constant stepsize gradient method will rapidly approach the ravine $\cM$ up to the point where the tangent part 
		$f_T$ dominates the normal part $f_N$. At this point, we will argue that taking a long (Polyak) gradient step will move the iterate significantly closer to the set of solutions. In order to facilitate this argument, we need to place some assumptions on $f_T$, which we now record. Most importantly, we will require $f_T$ to have constant order growth on $\cM$ away from the solution set. We record the relevant assumptions next.

		\begin{assumption}[Tangent part]\label{assum:tangent}
			Let $\cM$ be a $C^2$-ravine of a function $f$ at a minimizer $\bar x$, and let $S$ be the set of minimizers of $f$. Suppose that there exist constants $p>1$ and $\dlb, D_\ub>0$ such that the following properties hold for any $y\in \cM$ sufficiently close to $\bar x$:
			\begin{enumerate}
				\item {\bf (Growth condition)}\label{item: localgrowthconditionravine} 		\qquad	$
				\dlb\cdot \dist^p(y, S) \le f_T(y) - f^* \le D_\ub\cdot\dist^p(y, S). 
				$		\item {\bf(Aiming towards solution)}\label{item: localaimingtowardsolutionravine} 
				The estimate 
				$$
				f_T(y) - f^* \le \dotp{\nabla f_T(y), y - \bar y} + o(1)\cdot \dist^p(y, S),
				$$
				holds for any $\bar y\in P_{S}(y)$.
				
				\item {\bf(Size of gradients)}\label{item: localsizeofgdravine} 
				$
				\qquad		\|\nabla f_T(y)\| \le O(\dist^{p-1}(y, S)).
				$
			\end{enumerate}
		\end{assumption}

		At first sight, it may appear difficult to verify Assumption~\ref{assum:tangent}. On the contrary, the following theorem provides a convenient shortcut: if the function $f$ and the ravine $\cM$ are $C^{p+1}$-smooth, then the constant-order growth condition  \ref{item: localgrowthconditionravine}  automatically implies the other two regularity conditions \ref{item: localaimingtowardsolutionravine}  and \ref{item: localsizeofgdravine}. The following Lemma will be useful in proving the theorem.

		\begin{lemma}\label{lem:power_growth}
			Let $g$ be a $C^{p}$-smooth function {\color{blue} $(p\ge 2)$}. Let $S$ denote a closed subset of the minimizers of $g$. Fix $\bar x \in S$ and suppose that we have
			$$
			g(x) - \min g = O(\dist^p(x, S))
			$$
			for all $x$ near $\bar x$.  Then, for all $x$ near $\bar x$ 
			\begin{enumerate}
				\item {\bf(Aiming towards solution)}\label{item: localaimingtowardsolutionravine_sufficient} 
				The estimate 
				$$
				g(x) - \min g = \frac{1}{p}\dotp{\nabla g(x), x- {\color{blue}y} } +  o(1) \cdot \dist^{p}(x, S),
				$$
				{\color{blue}as $x$ tends to $\bar x$ and $y \in P_{S}(x)$.}
				
				\item {\bf(Size of gradients)}\label{item: localsizeofgdravine_sufficient} 
				$
				\qquad		\|\nabla g(x)\| = O(\dist(x, S)^{p-1}).
				$
			\end{enumerate}
		\end{lemma}
		\begin{proof}
			{\color{blue}
				We claim that $D^q g(y) = 0$ for all $y \in S$ near $\bar x$. To show this, first note that for any $y \in \mathbf{E}$, 
				we may form the Taylor expansion up to order $p-1$ as follows:	\begin{equation}\label{eqn:we_need_later}
					0 \leq g(x) - g(y)=\sum_{q=1}^{p-1}\frac{D^q g(y)[(x-y)^{\otimes q}] }{q!}+O(\|x-y\|^{p}).
				\end{equation}
				Now by assumption, there exists $C > 0$ such that for all $x\in \mathbf{E}$ and $y \in S$  near $\bar x$, we have 
				\begin{align}\label{eq:g_growth_rev}
					g(x) - \min g \leq C\dist^p(x,S) \leq C\|x - y\|^p
				\end{align} 
				From this we will deduce that $D^q g(y) [z^{\otimes q}]=0$ for all $q\in \{1,\ldots, p-1\}$ any unit vector $z$. 
				Indeed, define the curve $x(t) := y + tz$. 
				Then divide~\eqref{eqn:we_need_later} through by $\|x(t)-y\|^q$ with $q\in \{1,\ldots, p-1\}$ in increasing order and let $t$ tend to zero.
				From the bound~\eqref{eq:g_growth_rev}, we deduce $D^q g(y) [z^{\otimes q}]=0$. As shown in~\cite[Equation (1.2)]{nesterov2021implementable} and \cite[Theorem 1]{Banach1938}, this further implies that $D^q g(y) = 0$ for all $y \in S$ near $\bar x$.
			} 
			
			We will use this observation throughout the rest of the proof {\color{blue} for $y \in P_S(x)$ with $x$ near $\bar x$. This is justified since $y \rightarrow \bar x$ as $x \rightarrow \bar x$.}
			We  now establish the gradient size condition~\ref{item: localsizeofgdravine_sufficient} . To this end, for any vector $w$, we Taylor expand the gradient:
			{\color{blue}\begin{equation}\label{eqn:allign_taylor}
					\begin{aligned}
						\langle\nabla g(x),w\rangle&=\sum_{q=2}^p \frac{D^{q}g(\bar x)[(x-y)^{\otimes q-1}, w] }{(q-1)!}+ o(1)\dist(x, S)^{p-1}\cdot\|w\|\\
						&= \frac{D^{p}g(y)[(x-y)^{\otimes p-1}, w] }{(p-1)!}+ o(1)\dist(x,S)^{p-1}\cdot\|w\|\\
						&=O(\dist(x,S)^{p-1})\cdot \|w\|.
					\end{aligned}
				\end{equation}
				We thus deduce $\|\nabla g(x)\|=O(\dist(x,S)^{p-1})$.}
			
			Next, we argue the aiming condition \ref{item: localaimingtowardsolutionravine_sufficient}. Plugging in {\color{blue}$w=x-y$} in \eqref{eqn:allign_taylor}, and comparing the right side to a Taylor expansion of $g$ yields the equality
			{\color{blue}\begin{align*}
					\langle\nabla g(x),x-y\rangle&= \frac{D^{p}g(\bar x)[(x-y)^{\otimes p}] }{(p-1)!} +  o(1)\dist(x, S)^p\\
					&=p\cdot (g(x) -g(y))+  o(1)\dist(x,S)^{p}.
			\end{align*}}
			Rearranging gives 
			$$g(x) -g(y)=\frac{1}{p}\langle\nabla g(x),x-y\rangle+ o(1)\cdot\dist(x, S)^{p},$$
			as desired.

		\end{proof}
		
		\begin{theorem}[Growth and high-order smoothness]\label{thm:simplified_assumpt}
			Let $\cM$ be a $C^{p+1}$-ravine of a function $f$ at a minimizer $\bar x$  {\color{blue} $(p\ge 2)$}. Let $S$ be the set of minimizers of $f$. 
			Suppose that there exist constants $\dlb, D_\ub>0$ satisfying the growth condition 
			\begin{equation}\label{eqn:basic_growth_cor}
				\dlb\cdot \dist^p(y, S) \le f_T(y) - f^* \le D_\ub\cdot\dist^p(y, S),
			\end{equation}
			for all $y\in \cM$ near $\bar x$. Then Assumption~\ref{assum:tangent} holds.
		\end{theorem}
		\begin{proof}
			Note that since $S \subseteq \cM$ we have 
			$$
			\dist(P_{\cM}(x), S) \leq \dist(x,S) + \|x - P_{\cM}(x)\| \leq 2\dist(x, S). 
			$$
			Consequently, the function $f_T = f\circ P_\cM$ satisfies
			$$
			f_T(x) -  f^\ast = O(\dist^p(P_\cM(x), S)) = O(\dist^p(x, S)).
			$$
			Thus, by Lemma~\ref{lem:power_growth}, for all $x$ near $\bar x$ and $\hat x \in P_S(x)$,  we have 
			$$
			\|\nabla f_T(x)\| = O(\dist^{p-1}(x, S)) \qquad \text{ and } \qquad f_T(x)  - f^\ast \leq \frac{1}{p}\dotp{\nabla f_T(x), x - \hat x} + {\color{blue}o(1)}\cdot \dist^{p}(x, S).
			$$
			Finally, since $f_T(y) - f^\ast = \Theta(1)\cdot \dist^{p}(y, S)$ for all $y \in \cM$ near $\bar x$, we have $\dotp{\nabla f_T(y), y - \hat y} \geq 0$ for all $\hat y \in P_S(y)$. Consequently, 
			$$
			f_T(y)  - f^\ast \leq \dotp{\nabla f_T(y), y - \hat y} +  o(1)\cdot \dist^{p}(x, S),
			$$
			for all $y \in \cM$ near $\bar x$ and $\hat y \in P_S(y)$.
		\end{proof}

		In the special (and important!) case of $p=4$, meaning fourth-order growth, Assumption \ref{assum:tangent} can be simplified even further. The following proposition shows that if $\nabla^2 f$ has constant rank on the solution set and merely the left inequality holds in \eqref{eqn:basic_growth_cor}, then Assumption \ref{assum:tangent} holds automatically.

		\begin{proposition}[Ravine under fourth order growth]\label{prop:rav_four}
			Let $\cM$ be a $C^{5}$-ravine of a function $f$ at a  minimizer $\bar x$ and let $S$ be the set of minimizers of $f$. Suppose that $\nabla^2 f$ has constant rank on $S$ near $\bar x$ and that there exists a constant $\dlb >0$ satisfying the lower growth condition 
			\begin{equation}\label{eqn:basic_growth_cor_lower}
				\dlb\cdot \dist^4(y, S) \le f_T(y) - f^*,
			\end{equation}
			for all $y\in \cM$ near $\bar x$. Then Assumption~\ref{assum:tangent} holds.
		\end{proposition}
		\begin{proof}
			In light of Theorem~\ref{thm:simplified_assumpt}, it suffices to establish the bound $f_T(y) - f^*=O(\dist^4(y, S))$ for all $y\in \cM$ near $\bar x$. 
			To that end define $g = f\circ P_{\cM} - f^\ast$.
			Fix a point $\bar y\in S$ near $\bar x$. Observe the equalities $g(\overline y) = 0$, 
			$\nabla g(\bar y)=0$,
			and 
			$$\nabla^2g(\bar y)=P_{\cT_\cM(\bar y)}\nabla P_\cM(\bar y)^{\top}\nabla^2 f(\bar y)  \nabla P_\cM(\bar y) P_{\cT_\cM(\bar y)}=0,$$
			where we used the equality  ${\rm Null}(\nabla^2 f(\bar y))=T_{\cM}(\bar y)$ from Proposition~\ref{prop: uniform_ravine} and the inclusion ${\rm Range}(\nabla P(\bar y))\subset T_{\cM}(\bar y)$. 
			{\color{blue} We claim that $D^{3}g(\bar y) = 0$ as well.}
			Indeed, the Taylor expansion of $g$ around $\bar y$ takes the form
			$$g(x) = D^{3}g(\bar y)[(x-\bar y)^{\otimes 3}]+O(\|x-\bar y\|^4),$$
			for all $x$ near $\bar y$.
			Note the equality $D^{3}g(\bar y)[(-u)^{\otimes 3}]=-D^{3}g(\bar y)[u^{\otimes 3}]$ for all $u$. Therefore, taking into account that $\bar y$ is a minimizer of $g$ we deduce that $D^{3}g(\bar y)[u^{\otimes 3}]$ is zero for all $u$. As shown in~\cite[Equation (1.2)]{nesterov2021implementable} and \cite[Theorem 1]{Banach1938}, this further implies that $D^3 g(\bar y) = 0$.
			
			{\color{blue} Thus, we have shown that there exist $C > 0$ such that $g(x) = O(\|x - \bar y\|^4)$ for $x \in \mathbf{E}$ and $\bar y \in S$ near $\bar x$. Since $\bar y \in P_S(x) \rightarrow \bar x$ as $x \rightarrow \bar x$, it follows that $g(x) = O(\dist(x, S)^4)$ for all $x$ near $\bar x$. In particular, for all $y \in \cM$ near $\bar x$, we have
				$$f_T(y)-f^*= g(y)  = O(\dist^4(y, S)),$$
				as claimed.}
		\end{proof}
		
		Proposition~\ref{prop:rav_four} assumes that a ravine exists, but existence is automatic if $f$ is sufficiently smooth, as shown in Corollary~\ref{cor:morse_ravin_exist}. For ease of future reference, we record this corollary now: any sufficiently smooth function with constant rank Hessian on its solution set $S$, and fourth-order growth away from $S$ admits a (Morse) ravine satisfying Assumption~\ref{assum:tangent}.
		
		\begin{corollary}[Morse ravine under fourth order growth]\label{cor:morseravineexistsfourth}
			Let $f$ be a $C^{p+1}$-smooth function for $p\geq 5$ around a minimizer $\bar x$ and let $S$ be the set of minimizers of $f$. Suppose that $\nabla^2 f$ has constant rank on $S$ near $\bar x$ and that there exists a constant $\dlb >0$ satisfying the lower growth condition 
			\begin{equation}\label{eqn:basic_growth_cor_lower2}
				\dlb\cdot \dist^4(x, S) \le f(x) - f^*,
			\end{equation}
			for all $x$ near $\bar x$. Then $f$ admits a $C^p$ Morse ravine around $\bar x$ satisfying 
			Assumption~\ref{assum:tangent}.
		\end{corollary}
		\begin{proof}
			This follows directly from Corollary~\ref{cor:morse_ravin_exist} and Proposition~\ref{prop:rav_four}.
		\end{proof}

		In the final part of this section, we prove that the Assumption~\ref{assum:tangent} holds for $C^3$ semialgebraic functions defined on $\RR^2$, as long as the Hessian is nondegenerate at the minimizer. Recall that a function $f$ is semialgebraic if its graph can be written as a finite union of sets each cut out by finitely many polynomial inequalities. We refer the reader to the introductory lecture notes  \cite{Coste-semi} or the influential monograph \cite{bochnak2013real} on the subject.
		\begin{theorem}[Semialgebraic functions on $\RR^2$ admit a ravine]
			Let $f:\RR^2\to \RR$ be a $C^{\infty}$-smooth semialgebraic function and let $S$ be the set of minimizers of $f$.
			Then around any $\bar x \in S$ such that $\nabla^2 f(\bar x)$ is nonzero, function $f$ admits a $C^{\infty}$ ravine satisfying Assumption~\ref{assum:tangent}.
		\end{theorem}
		\begin{proof}
			The case where $\nabla^2 f(\bar x)$ is full rank is immediate, so we may assume that $\nabla^2 f(\bar x)$ is degenerate. 
			In this case, there exists a $1$-dimensional $C^{\infty}$ Morse ravine $\cM$ of $f$ at $\bar x$. By standard quantifier elimination, the Morse ravine $\cM$ is a semialgebraic set.
			Since $\cM$ is $1$-dimensional, there exists a $C^{\infty}$-smooth semialgebraic curve $\phi \colon \RR \rightarrow \cM$ such that $\phi(0) = \bar x$ and $\nabla \phi(t) \neq 0$ for all $t$ near $0$ that locally parametrizes $\cM$ near $\bar x$.
			In particular, we have 
			$$
			\|\phi(t) - \phi(s)\| = \Theta(1) \cdot |t-s|
			$$
			for all $t,s$ near $0$.
			Consequently, for all $t$ near $0$ and small neighborhood $X$ of $\bar x$ we have 
			\begin{align*}
				\dist(t, \phi^{-1}(S\cap X)) &= \min_{s\in \phi^{-1}(S\cap X)} |t-s| \\
				&=  \Theta(1) \cdot\min_{s\in \phi^{-1}(S\cap X)} \|\phi(t)-\phi(s)\| \\
				&= \Theta(1) \cdot \dist(\phi(t), S)
			\end{align*}
			Now notice that since $f$ is semialgebraic, and $g = f\circ \phi$ is semialgebraic, there exists a neighborhood $U$ of $0$ such that $J:=U\cap \phi^{-1}(S\cap X)$ is an interval containing $\{0\}$. If $0$ lies in the interior of $J$, then $S$ and $\cM$ locally coincide near $\bar x$, so that $f_T = f\circ P_\cM$ is constant near $\bar x$ and therefore Assumption~\ref{assum:tangent} is trivially satisfied. There are two remaining cases, namely that $J=\{0\}$ or that zero is a boundary point of $J$. We focus on the former since the argument for the latter is completely analogous.

			Thus for the remainder of the proof we suppose that  that $U \cap \phi^{-1}(S\cap X) = \{0\}$. Since $g$ is a $C^{\infty}$ semialgebraic function that is minimized at zero, there exists $c > 0$ and an even integer $p > 1$ satisfying 
			\begin{align*}
				g(t) = c t^p + o(1) \cdot t^{p}\\
				g'(t) = cp t^{p-1} + o(1) \cdot t^{p-1},	
			\end{align*}
			for all $t$ in $U$. Consequently, we have
			\begin{align*}
				f(\phi(t)) - f^* &= g(t) = c t^p + o(1) \cdot t^{p} = \Theta(1)\cdot \|\phi(t)\|^p + o(1) \cdot \|\phi(t)\|^{p}
			\end{align*}
			for all $t$ in $U$. We claim that this expansion implies that Assumption~\ref{assum:tangent} holds. Indeed, this equality is clearly a restatement of the growth condition in Part~\ref{item: localgrowthconditionravine} of Assumption~\ref{assum:tangent}. Thus, it remains to verify the final two conditions in Assumption~\ref{assum:tangent}.\footnote{One might be tempted here to apply Theorem~\ref{thm:simplified_assumpt} and conclude that Assumption~\ref{assum:tangent} holds. However, the ravine $\cM$ is not necessarily $C^{p+1}$ smooth, so we cannot apply it directly.}
			
			Second, we verify the gradient bound in Part~\ref{item: localsizeofgdravine} of Assumption~\ref{assum:tangent}. To that end, observe that $T_\cM(\phi(t)) = \range(\nabla \phi(t))$ and that $\sigma_{\min} (\nabla \phi(t))$ is bounded away from zero near for $t$ near zero. Consequently, since $\nabla f_T(\phi(t)) \in T_\cM(\phi(t))$ we have
			\begin{align*}
				\|\nabla f_T(\phi(t))\| &= \|P_{T_{\cM}(\phi(t))}\nabla f(\phi(t))\| \\
				&= O(\|\nabla \phi(t)^\top P_{T_{\cM}(\phi(t))} \nabla f(\phi(t))\|) \\
				&= O(\|\nabla \phi(t)^\top \nabla f(\phi(t))\|) \\
				&=  O(\|g'(t)\|)\\
				&= O(t^{p-1}) \\
				&= O(\|\phi(t)\|^{p-1})\\
				&= O(\dist^{p-1}(\phi(t), S))
			\end{align*}
			for all $t$ near $0$. This proves the gradient bound.
			
			Finally, we verify the aiming condition in Part~\ref{item: localaimingtowardsolutionravine} of Assumption~\ref{assum:tangent}.
			Indeed, we have $g'(t) = cp\cdot t^{p-1} + o(1)\cdot t^{p-1}$, and therefore 
			\begin{align*}
				g(t) &= c t^p + o(1) \cdot t^{p}\leq cp t^p+ o(1) \cdot t^{p} \leq g'(t)t + o(1) \cdot t^{p}.
			\end{align*}
			Thus, to prove the aiming condition, we relate $g'(t)t$ to $\dotp{\nabla f_T(\phi(t)), \phi(t)}$: 
			\begin{align*}
				g'(t)t &= \dotp{\nabla f(\phi(t)), \nabla \phi(t)t} \\
				&= \dotp{P_{\cT_{\cM}(\phi(t))}\nabla f(\phi(t)), \nabla \phi(t)t} \\
				&= \dotp{\nabla f_T(\phi(t)), \nabla \phi(t)t}\\
				&= \dotp{\nabla f_T(\phi(t)), \phi(t)} + \dotp{\nabla f_T(\phi(t)), \nabla \phi(t)t - \phi(t)} \\
				&= \dotp{\nabla f_T(\phi(t)), \phi(t)} + \dist^{p-1}(\phi(t), S) \cdot  O(t^2) \\
				&= \dotp{\nabla f_T(\phi(t)), \phi(t)} + O(\dist^{p+1}(\phi(t), S)).
			\end{align*}
			Putting it all together, we have 
			\begin{align*}
				f(\phi(t)) - f^\ast = g(t)	&\leq g'(t)t + o(1) \cdot t^{p} \\
				&= \dotp{\nabla f_T(\phi(t)), \phi(t)} + O(\dist^{p+1}(\phi(t), S)) + o(1) \cdot t^{p} \\
				&\leq \dotp{\nabla f_T(\phi(t)), \phi(t)} + o(1) \cdot \dist^{p}(\phi(t), S),
			\end{align*}
			for all $t$ near $0$, as desired.
		\end{proof}

		\section{Algorithms and main convergence theorem}
		We are now ready to state the gradient descent algorithm with adaptive stepsizes and analyze its performance. 
		This paper presents two variants of the algorithm: one which assumes knowledge of the optimal value $f^*$ {\color{blue}(i.e., Algorithm~\ref{alg:GD-Polyak})} and one which only assumes knowledge of a lower bound $f_0 \leq f^\ast$. 
		The former has better local oracle complexity than the latter -- $O(\log(1/\varepsilon)^2)$ versus $O(\log(1/\varepsilon)^3)$ -- but may not be implementable in general.
		
		\subsection{Algorithm with knowledge of $f^\ast$}
		
		%
		
		{\color{blue} When $f^\ast$ is known, we apply Algorithm~\ref{alg:GD-Polyak} from the introduction.} The intuition behind the algorithm is as follows. Due to the ravine's defining property, gradient descent approaches the ravine at a linear rate up to a tolerance controlled by the iterates' suboptimality. Once this tolerance is reached, the function behaves similarly to its tangent part $f\circ P_{\cM}$, which is assumed to have constant order growth on the manifold. The Polyak gradient step makes significant progress towards the optimal solution for such functions. To see this key point, let us look at a simplified setting of minimizing a smooth convex function $g$ on $\E$ that satisfies $g(x)- g^*=\Theta(1)\cdot\|x-\bar x\|^p$, where $\bar x$ is its minimizer, and $g^*$ is its minimal value. Then the  Polyak step $x^+:=x-\frac{g(x)-g^*}{\|\nabla g(x)\|^2}\nabla g(x)$ satisfies:
		\begin{align*}
			\tfrac{1}{2}\|x^+-\bar x\|^2&= \tfrac{1}{2}\|x-\bar x\|^2-\tfrac{g(x)-g^*}{\|\nabla g(x)\|^2}\langle \nabla g(x),x-\bar x\rangle+\tfrac{1}{2}\left(\frac{g(x)-g^*}{\|\nabla g(x)\|}\right)^2\\
			&\leq\tfrac{1}{2}\|x-\bar x\|^2-\tfrac{1}{2}\left(\frac{g(x)-g^*}{\|{\color{blue}\nabla  g(x)}\|}\right)^2.
		\end{align*}
		Now constant order growth implies $g(x)-g^*=\Theta(1)\cdot\|x-\bar x\|^p$ and $\|\nabla g(x)\|=O(\|x-\bar x\|^{p-1})$ and therefore there exists a constant $c\in (0,1)$ satisfying $\|x^+-\bar x\|\leq c\cdot \|x-\bar x\|$.
		That is, a single Polyak step shrinks the distance to the solution by a constant factor. Despite the simplicity of this argument, extending it in our setting presents numerous technical challenges: controlling the progress of the constant step gradient method towards the ravine, lack of convexity, constant order growth holding only along the ravine, etc. The end result is the following theorem—the main result of the paper.

		\begin{theorem}[Nearly linear rate]\label{thm: main_thm_intro}
			Let $f$ be a $C^{2}$-smooth function with a set of minimizers $S$. Suppose that for some point $\bar x\in S$,  Assumption~\ref{assum:tangent} holds with respect to a ravine $\cM$ and such that the projection $P_{\cM}$ is $C^2$-smooth near $\bar x$. Then there exist constants $\delta_0,\eta_0,c,C>0$ such that for any initial point $x_0\in {\color{blue} B_{\delta_0}(\bar x)}$, any stepsize $\eta\in (0,\eta_0)$, and any iteration counters $K, I\in \mathbb{N}$, the point $\xout=\gdpk(x_0, \gdstep, \gdrds, I)$ returned by Algorithm~\ref{alg:GD-Polyak} satisfies 
			$$
			f(\xout) - f^* \le Ce^{-c\cdot\eta \cdot\min\{K, I\}},
			$$
			and the total number of gradient and function evaluations is bounded by $I\cdot (K+1)$.
		\end{theorem}
		
		In particular, given a target accuracy $\varepsilon>0$, we may set $I=K=\lceil(c\eta)^{-1}\log( C/\varepsilon)\rceil$, and then Algorithm~\ref{alg:GD-Polyak} will find a point $\xout$ satisfying $f(\xout) - f^* \le \varepsilon$ after using at most $O(\log^2(\frac{1}{\varepsilon}))$ gradient and function evaluations.

		Finally, we note that when the solution set $S$ is compact, the conclusion of Theorem~\ref{thm: main_thm_intro} holds if we initialize in a sufficiently small tube $\dist(x_0, S) < \delta$ around the set of minimizers $S$, rather than in a ball around a fixed minimizer $\bar x$.

		\begin{corollary}[Nearly linear rate under compactness]\label{cor: compactravineconvergence}
			Let $f$ be a $C^{2}$-smooth function with a compact set of minimizers $S$. 
			Suppose that for any point $\bar x\in S$,  Assumption~\ref{assum:tangent} holds with respect to a ravine $\cM_{\bar x}$ and such that the projection $P_{\cM_{\bar x}}$ is $C^2$-smooth near $\bar x$. 
			Then there exist constants $\delta_0,\eta_0,c,C>0$ such that for any initial point $x_0$ satisfying $\dist(x_0,S)\leq \delta_0$, any stepsize $\eta \le \eta_0$, and any iteration counters $K, I\in \mathbb{N}$, the point $\xout=\gdpk(x_0, \gdstep, \gdrds, I)$ returned by Algorithm~\ref{alg:GD-Polyak} satisfies 
			$$
			f(\xout) - f^* \le Ce^{-c\cdot\eta \cdot\min\{K,I\}},
			$$
			and the total number of gradient and function evaluations is bounded by $I\cdot (K+1)$.
		\end{corollary}

		\subsection{Algorithm with a lower bound on $f^\ast$}
		
		While precise knowledge of $f^\ast$ may be unavailable, we often know a lower bound $f_0 \leq f^\ast$ on the minimal objective value.
		For example, in data fitting or machine learning problems, the function $f$ may represent a measurement misfit penalty or a proxy for model accuracy, which always has a lower bound $0 \leq f^\ast$.
		Following the technique introduced in~\cite{hazan2019revisiting}, we now introduce an adaptive version of $\gdpk$, which only requires a lower estimate $f_0 \le  f^\ast$. 
		The algorithm restarts a modified version of $\gdpk$ with a new estimate $f_n$ of $f^\ast$ at each step. 
		The modified version of $\gdpk$ is identical to the original except that we use a Polyak stepsize in which $f^\ast$ is replaced by the current estimate $f_n$ and the overall stepsize is scaled down by 2.
		
		\begin{algorithm}[H]
			\caption{$\gdpklb(x_{0}, \gdstep, \gdrds, I, J, f_0$)}
			\label{alg:GD-Polyak_LB}
			\begin{algorithmic}[1]
				\State {\bfseries Input}   $x_0, \gdstep,  K, I, J, f_0$.
				\For {$j = 1,\ldots, J$}
				\State $x_{0,j-1} = x_0$
				\For {$i = 1,\ldots, I$}
				\State $\tilde x_{i, j-1} = \gd(x_{i-1,j-1}, \gdstep, \gdrds)$
				\State $x_{i,j-1} = \tilde x_{i,j-1} -\frac{f(\tilde x_{i,j-1}) - f_{j-1}}{2\|\nabla f(\tilde x_{i,j-1})\|^2} \nabla f(\tilde x_{i,j-1})$.
				\EndFor
				\State $x_{j} = \argmin\{f(x_{i,j-1}), f(\tilde x_{i,j-1}) \colon i = 1,\ldots, I\}$
				\State $f_j = \frac{f_{j-1} + f(x_j)}{2}$.
				\EndFor
				\State $x_{\mathtt{out}} = \argmin\{f(x_j) \colon j = 1,\ldots, J\}$
				\State \Return $\xout$
			\end{algorithmic}
		\end{algorithm}
		
		Let us briefly explain the principle underlying Algorithm~\ref{alg:GD-Polyak_LB}. 
		First, it is straightforward to check from the proof Theorem~\ref{thm: main_thm_intro}, that if one runs the $\gdpk$ with an approximation of the true Polyak stepsize, which is at most the Polyak stepsize and at least half the Polyak stepsize, then the guarantees of Theorem~\ref{thm: main_thm_intro} continue to hold. 
		Thus, within the context of Algorithm~\ref{alg:GD-Polyak_LB}, if for some $j$, we have
		\begin{align}\label{eq:polyak_gap_estimate}
			\frac{f(\tilde x_{i,j-1}) - f^\ast }{2\|\nabla f(\tilde x_{i,j-1})\|^2}
			\leq \frac{f(\tilde x_{i,j-1}) - f_{j-1}}{2\|\nabla f(\tilde x_{i,j-1})\|^2} \leq \frac{f(\tilde x_{i,j-1}) - f^\ast }{\|\nabla f(\tilde x_{i,j-1})\|^2} \qquad  \text{for $i =1, \ldots, I$},
		\end{align}
		then $f(x_n) - f^* \le Ce^{-c\cdot\eta \cdot\min\{K,T\}}$ for appropriate constants $C, c > 0$.
		
		We claim that even if~\eqref{eq:polyak_gap_estimate} fails for all $j$, we can still infer the following bound: 
		$$
		f(x_{\mathtt{out}}) - f^\ast \leq 2^{-(J-1)}(f^\ast - f_0).
		$$
		Indeed, let us suppose that there is no $j$ for which~\eqref{eq:polyak_gap_estimate} holds. 
		We claim that $f_j \leq f^\ast$ for all $1 \leq j \leq J$. If not, there exists a first index $j$ such that $f_{j} > f^\ast$. 
		By definition, the lower bound in~\eqref{eq:polyak_gap_estimate} holds for index $j-1$ since $f_{j-1} \leq f^\ast$. 
		In addition, since~\eqref{eq:polyak_gap_estimate} fails, there exist $1 \leq i \leq I$ such that 
		$$
		f(\tilde x_{i,j-1}) - f_{j-1} > 2(f(\tilde x_{i,j-1}) - f^\ast).
		$$
		Consequently, $f^\ast > (f(\tilde x_{i,j-1}) + f_{j-1})/2 \geq f_{j}$, which is a contradiction.
		Thus, we have $f_{j} \leq f^\ast$ for all $1 \leq j \leq J$. 
		Now since $f_{j-1} \leq f^\ast \leq f(x_j)$, we have 
		$
		0 \leq f^\ast - f_{j}  \leq (f^\ast - f_{j-1})/2
		$
		for each $n$.
		Iterating this bound, we find that 
		$$
		f(x_{\mathtt{out}}) \leq f(x_J) = 2f_{J} - f_{J-1} \leq f^\ast + (f^\ast - f_{J-1}) \leq f^\ast + 2^{-(J-1)}(f^\ast - f_{0}).
		$$
		Thus, we have proved the desired bound.

		Therefore, this argument shows that Algorithm~\ref{eq:polyak_gap_estimate} exhibits the following local behavior. We omit the formal proof for simplicity. 
		
		\begin{theorem}[Nearly linear rate]\label{thm: main_thm_intro_adaptive}
			Let $f$ be a $C^{2}$-smooth function with a set of minimizers $S$. Fix a lower bound $f_0 \leq f^\ast$. Suppose that for some point $\bar x\in S$,  Assumption~\ref{assum:tangent} holds with respect to a ravine $\cM$ and such that the projection $P_{\cM}$ is $C^2$-smooth near $\bar x$. Then there exist constants $\delta_0,\eta_0,c,C>0$ such that for any initial point $x_0\in B_{\delta}(\bar x)$, any stepsize $\eta\in (0,\eta_0)$, and any iteration counters $K,I,J\in \mathbb{N}$, the point $\xout=\gdpklb(x_0, \gdstep, \gdrds, I, J, f_0)$ returned by Algorithm~\ref{alg:GD-Polyak_LB} satisfies 
			$$
			f(\xout) - f^* \le \max\{Ce^{-c\cdot\eta \cdot\min\{K,I\}}, 2^{-(J-1)}(f^\ast - f_0)\},
			$$
			and the total number of gradient and function evaluations is bounded by $K\cdot J\cdot (I+1)$.
		\end{theorem}
		
		In addition to the above theorem, we could state and prove a result similar to Corollary~\ref{cor: compactravineconvergence} for Algorithm~\ref{alg:GD-Polyak_LB}. For brevity, we omit the statement.

		The rest of the section is devoted to proving Theorem~\ref{thm: main_thm_intro} and Corollary~\ref{cor: compactravineconvergence}. 
		
		\subsection{Key ingredients and proof of Theorem~\ref{thm: main_thm_intro} and Corollary~\ref{cor: compactravineconvergence}}
		This section contains the key ingredients we will need for the proof of Theorem~\ref{thm: mainconvergencelocal}. {\color{blue}In order to state the theorem, we must first state the following corollary, which includes key constants appearing in the theorem. The corollary is essentially a combination of Assumption~\ref{assum:tangent} and Theorem~\ref{thm:key_properties}.} Namely every occurrence of $\|\nabla f_T(x)\|$ in may be replaced by $\dist^{p-1}(P_{\cM}(x),S)$, due to Assumption~\ref{assum:tangent}. We record this observation in the following proposition for ease of reference.
		
		\begin{corollary}\label{cor: localconditionsforgdpolyak}(Ravines satisfying Assumption A)
			Let $f:\RR^d \rightarrow \RR$ be a function with a non-empty set of minimizer $S$. Suppose that Assumption~\ref{assum:tangent} holds for a fixed minimizer $\bar x \in S$ and a ravine $\cM$. Assume, moreover, that $P_\cM$ is $C^2$ smooth near $\bar x$. Then there exists a neighborhood $U$ of $\bar x$ such that the following holds:
			\begin{enumerate}
				\item {\bf (Aiming towards ravine)}\label{item: localaimingtowardmanifold} There exists $\gammalb >0$ such that for any $x\in U$,
				\begin{align}\label{eqn: aiming_to_ravine_cor_I}
					\dotp{\nabla f_N(x), x - P_\cM(x)} \ge \gammalb \dist^2(x,\cM) + o(1) \cdot\dist(x,\cM)\cdot\dist^{p-1}(P_\cM(x),S)
				\end{align}
				and 
				\begin{align}\label{eqn: aiming_to_ravine_cor_II}
					\dotp{\nabla f_N(x), x - P_\cM(x)} \ge \gammalb \dist^2(x,\cM) + o(1) \cdot\dist^{2p-2}(P_\cM(x),S).
				\end{align}
				\item {\bf(Aiming towards solution)}\label{item: localaimingtowardsolution} For any $ y \in U \cap \cM$, we have
				\begin{equation}\label{eqn:est_taylor}
					f_T(y) - f^* \le \dotp{\nabla f_T(y), y - \bar y} + o(1)\cdot \dist^p(y, S),
				\end{equation}
				holds for any $\bar y\in P_{S}(y)$.
				\item {\bf(Size of gradients)}\label{item: localsizeofgd} There exists $\gammaub >0$ such that for any $x\in U$, it holds:
				$$
				\|\nabla f_N(x)\| \le \gammaub \dist(x,\cM) + o(1)\cdot \dist^{p-1}(P_\cM(x), S).
				$$
				There exist $\betalb,\betaub >0$ such that for any $y \in U \cap \cM$, we have
				\begin{equation}\label{eqn:control_tan_grad}
					\betalb \dist^{p-1}(y, S)\le \|\nabla f_T(y)\| \le \betaub \dist^{p-1}(y, S).
				\end{equation}
				\item {\bf(Projected gradient)}\label{item: localstronga} For any $x \in U$, we have
				$$
				\|P_{T_{\cM(P_\cM(x))}} \nabla f_N(x)\| = o(1)\cdot  \dist(x,\cM).
				$$
				\item {\bf (Growth condition)}\label{item: localgrowthcondition} There exist $C_\ub > 0$ such that for any $x \in U$  we have
				$$ |f_N(x)| \le C_\ub\dist^2(x, \cM) +o(1)\cdot\dist(x,\cM)\dist^{p-1}(P_\cM(x), S). 
				$$
				There exist $\dlb$ and $D_\ub$ such that for any $y \in U\cap \cM$  we have
				\begin{equation}\label{eqn:func_control}
					\dlb\dist^p(y, S)\leq f_T(y) - f^* \le D_\ub\dist^p(y, S). 
				\end{equation}
			\end{enumerate}
		\end{corollary}
		\begin{proof}
			First setting $y=P_{\cM}(x)$ and using boundedness of $\nabla P_{\cM}$ we compute
			\begin{align}
				\|\nabla f_T(x)\|&=\|\nabla P_{\cM}(x)\nabla f(y)\|\notag\\
				&\leq \|\nabla f_T(y)\|+\|(\nabla P_{\cM}(y)-\nabla P_{\cM}(x))\nabla f(y)\|\label{eqn:triangle2}\\
				&= \|\nabla f_T(y)\|+\|(\nabla P_{\cM}(y)-\nabla P_{\cM}(x)) \nabla f_T(y)\|\label{eqn:insert_proj}\\
				&=O(\|\nabla f_T(y)\|)\notag\\
				&=O(\dist^{p-1}(y,S)), \label{eqn:dist_soln_grad}
			\end{align}
			where \eqref{eqn:triangle2} follows from the triangle inequality, equation \eqref{eqn:insert_proj} follows from the expressions $\nabla P_{\cM}(y)N_{\cM}(y)=\nabla P_{\cM}(x)N_{\cM}(y)=0$, and \eqref{eqn:dist_soln_grad}
			follows from Assumption~\ref{assum:tangent}(\ref{item: localsizeofgdravine}). Combining this estimate,   Assumption~\ref{assum:tangent}, and Theorem~\ref{thm:key_properties} yields all of the claims, except~\eqref{eqn: aiming_to_ravine_cor_II} and the left side of \eqref{eqn:control_tan_grad}. Let us argue these two claims separately. The claim \eqref{eqn: aiming_to_ravine_cor_II} follows directly from equation~\eqref{eqn: aiming_to_ravine_cor_I} and Young's inequality:
			$$
			\dist(x,\cM) \dist^{p-1}(P_\cM(x),S) \le \frac{1}{2}\paren{\dist^2(x,\cM) + \dist^{2p-2}(P_\cM(x),S)}.
			$$
			The left side of 	\eqref{eqn:control_tan_grad} follows from Item~\ref{item: localaimingtowardsolution}, the Cauchy-Schwartz inequality, and the assumption $f(y)-f^*=\Theta(1)\cdot \dist^p(y,S)$.
			Thus, the proof is complete.
		\end{proof}
		
		{\color{blue} We are now ready to state our main convergence result (Theorem~\ref{thm: mainconvergencelocal}), which immediately implies the announced Theorem~\ref{thm: main_thm_intro}. A complication in the proof is that we must argue that the iterates stay near the initial point, which we do through a careful inductive argument. 
			
			\begin{theorem}[Main convergence theorem]\label{thm: mainconvergencelocal}
				Suppose that Assumption~\ref{assum:tangent} holds for some fixed minimizer $\bar x$ of $f$ and $P_\cM$ is $C^2$ smooth near $\bar x$. Let $\gammalb, \gammaub,\betaub, C_\ub, \dlb, \dub$ be parameters from Corollary~\ref{cor: localconditionsforgdpolyak}.  Let $\deltaU >0$ be such that  $U = B(\bar x, \deltaU)$ satisfies the requirements from Lemma~\ref{lem: innerlooplocal}. Let $\ckl$ be the constant  from Lemma~\ref{lem: KLlocal}, and let $\lipf$ and $L$ be the Lipschitz constants of $f$ and $\nabla f$ on $U$, respectively.  Let $K,T,\eta$ be our algorithm parameters and $\tilde x_i, x_i, \xout$ be defined as in Algorithm~\ref{alg:GD-Polyak} when we set input to be $(x_0, \eta, K,I)$.  Suppose that $\eta \le \min\left\{\frac{\gammalb}{2\gammaub^2}, \frac{1}{\betaub}, \frac{1}{L}\right\}$. Then there exists a constant $\deltagdpk > 0$ such that if $x_0 \in B(\bar x, \deltagdpk)$, 
				at least one of the following holds.
				\begin{enumerate}
					\item \label{item: mainthmitem1local}$f(\xout) - f^* \le 162C_\ub\paren{1- \frac{\eta\gammalb}{4}}^{2K} \deltagdpk^2 + D_\ub\paren{1-\frac{\eta \gammalb}{4}}^{\frac{Kp}{p-1}} \paren{\frac{1800\gammaub\deltagdpk}{\dlb}}^{\frac{p}{p-1}}.$
					\item \label{item: mainthmitem2local} $f(\xout) - f^* \le  2\cdot 3^{2p-2} \cdot C_\ub\paren{\frac{\dlb}{200\gammaub}}^2 \paren{1- \frac{\dlb^2}{20\betaub^2}}^{(2p-2)(I-1)}\deltagdpk^{2p-2} + 3^p\cdot D_\ub \paren{1- \frac{\dlb^2}{20\betaub^2}}^{p(I-1)}\deltagdpk^p$.
				\end{enumerate}
			\end{theorem}
			
			With this result in hand, the proof of Corollary~\ref{cor: compactravineconvergence} is now immediate.
			
			\begin{proof}[Proof of Corollary~\ref{cor: compactravineconvergence}]
				By Theorem~\ref{thm: mainconvergencelocal}, for any $\bar x\in S$, there exists a radius $\delta_{\bar x}$ such that when $\|x_0 -\bar x\| \le  \delta_{\bar x}$, the desired conclusion holds with constants $C_{\bar x}$, $c_{\bar x}$, and $\eta_{\bar x}$. By compactness of $S$, there exists a finite index set $I\subset S$ and $\delta>0$ such that 
				$$
				S \subset \{x \colon \dist(x,S) < \delta\} \subset \bigcup_{\bar x \in I} B(\bar x, \delta_{\bar x}).
				$$
				We take $C = \max_{\bar x \in I} C_{\bar x} > 0$, $c = \min_{\bar x \in I} c_{\bar x} >0$, and $\eta_0 = \min_{\bar x \in I} \eta_{\bar x} > 0$. For any $x_0$ with $\dist(x_0, S) < \delta$, there exists some $\bar x \in I$ such that $\|x_0 - \bar x\| \le \delta_{\bar x}$. Therefore, for any $\eta \le \eta_0 \le \eta_{\bar x}$, we have
				$$
				f(\xout) - f^* \le C_{\bar x} e^{-c_{\bar x} \eta \min\{K,I\}} \le C e^{-c \eta \min\{K,I\}},
				$$
				as desired. 
			\end{proof}
			
			We now turn our attention to the proof of Theorem~\ref{thm: mainconvergencelocal}. In our proof of the result, we need several auxiliary lemmas, which we now state. Afterwards, we prove Theorem~\ref{thm: mainconvergencelocal} in Section~\ref{sec:proofmaintheorem}
		}
		\subsubsection{Auxiliary lemmas for Theorem~\ref{thm: mainconvergencelocal}}
		
		In the rest of this section, we assume the setting of Corollary~\ref{cor: localconditionsforgdpolyak} and fix all the constants and neighborhoods in the Corollary. We also assume that $P_\cM$ is $C^2$ smooth on $U$, and we let $C_\cM$ be the Lipschitz constant of $\nabla P_\cM$ on $U$ with respect to the operator norm. We begin with the following useful lemma. The proofs of the rest of the results are deferred to Section~\ref{sec:proofs_technical_lemmas}.

		\begin{lemma}\label{lem: propertyofdecomplocal}
			
			The following properties hold after shrinking $U$ if necessary.
			\begin{enumerate}
				
				\item \label{item: propertyofdecomplipschitzgradientlocal}For any $x \in U$, we have 
				$$
				\|\nabla f_T(x) - \nabla f_T(P_\cM(x))\| \le C_\cM\cdot \dist(x,\cM)\cdot \|\nabla f_T(P_\cM(x))\|.
				$$
				\item  \label{item: propertyofdecompsizeofgradientlocal} For any $x \in U$, we have 
				$$\frac{99}{100} \|\nabla f_T(P_\cM(x))\| \le \|\nabla f_T(x)\| \le  \frac{101}{100} \|\nabla f_T( P_{\cM}(x))\|.$$
				\item \label{item: propertyofdecompbregularitylocal}For any $y \in U\cap \cM$, we have $\betaub \ge \frac{9}{10} \dlb$ and
				$$
				\|\nabla f_T(y)\| \|y - \bar y\|\ge \dotp{\nabla f_T(y), y - \bar y} \ge  \frac{9}{10}(f(y) - f^*)
				$$
				for any $\bar y\in P_S(y)$.
				\item \label{it:lower_bound_grad_norm} For any $x \in U$, we have
				$$ \dist(x,\cM) \le \Theta(1) \|\nabla f_N(x)\| + o(1) \dist^{p-1} (P_{\cM}(x),S).$$
				
				\item\label{it:norm_comp}  For any $x \in U$, we have\begin{equation*}
					\tfrac{1}{1+o(1)}\cdot \|\nabla f(x)\|^2 = \|\nabla f_N(x)\|^2 + \|\nabla f_T(x)\|^2.
				\end{equation*}
			\end{enumerate}
			
		\end{lemma}
		
		\begin{lemma}\label{lem: KLlocal}
			The function $f$ satisfies {\L}ojasiewicz inequality with exponent $\frac{p-1}{p}$                                                                                  at $\bar x$, that is there exists a constant $\ckl$ such that by shrinking $U$ if necessary, we have 
			\begin{align*}
				\ckl \|\nabla f(x)\| \ge (f(x) - f^*)^{\frac{p-1}{p}}, \qquad \forall x\in U.
			\end{align*}
		\end{lemma}

		The following lemma shows that the Polyak step can effectively reduce the distance from $y$ to the solution set when a point is close to the ravine. The Polyak step, however, can overshoot in the normal direction to the ravine. The lemma also shows that the iterate will not go too far from the manifold.
		\begin{lemma}\label{lem: progressinypolyaklocal}
			For any $x \in U$, define the points $x_+ = x -   \frac{f(x) - f^*}{\|\nabla f(x)\|^2} \nabla f(x)$, $y = P_\cM(x)$ and  $y_+ = P_\cM(x_+)$. Suppose that the inequality $\|\nabla f_N(x)\|\le \frac{1}{100}\|\nabla f_T(y)\|$ holds. Then after shrinking $U$ if necessary,  the following estimates hold.
			\begin{enumerate}
				\item (Progress in tangent direction) \label{item: progressinypolyaklocal}	$$\dist(y_+, S) \le  \paren{1-  \frac{1}{10} \frac{\dlb^2}{\betaub^2}}\dist(y, S).$$
				\item  (Bound in normal direction) \label{item: xMafterpolyaklocal}
				$$
				\dist(x_+, \cM) \le  3\cdot\dist(y, S).
				$$
			\end{enumerate}
		\end{lemma}

		%
		%
		%
		%

		The following lemma shows that the distance to the ravine shrinks geometrically until the norms of the gradients of $f_N(x)$ and $f_T(x)$ balance out. Meanwhile, the constant size gradient descent steps do not interfere much with the progress toward the set of minimizers.
		\begin{lemma}\label{lem: innerlooplocal}
			For all sufficiently small $\deltaU > 0$, the following holds.\footnote{Specifically, one must choose $\deltaU$ small enough that the conclusions of Lemmas~\ref{lem: propertyofdecomplocal},~\ref{lem: KLlocal},~\ref{lem: progressinypolyaklocal},~\ref{lem: polyakstepsizeboundlocal},~\ref{lem: GDshrinksdisttoMlocal},~\ref{lem: ychangeingdlocal},~\ref{lem: finitelength} hold.} L Let $U = B(\bar x, \deltaU)$, let $\ckl$ be the constant  from Lemma~\ref{lem: KLlocal}, and let $L$ be the Lipschitz constant of $\nabla f$ in $U$. Choose a stepsize $\eta \le \min\left\{\frac{\gammalb}{2\gammaub^2}, \frac{1}{\betaub}, \frac{1}{L}\right\}$. Consider the sequence $\{x_k\}_{k=0}^{K}$ generated by gradient descent with stepsize $\eta$. Define the projected points $y_k = P_\cM(x_k)$. Then there exists $\rho_0 >0$ such that for any $x_0 \in B(\bar x, \rho_0)$, the following statements are true.
			\begin{enumerate}
				\item \label{item: stayinneighborhood}
				For each $0\le k \le K$, the iterate $x_k$ lies in $U$ and satisfies
				\begin{equation}\label{eqn:KL_bound_needed}
					\|x_k - x_0\| \le \frac{2}{L} \|\nabla f(x_0)\| + 4\ckl p \cdot (f(x_0) - f^*)^{1/p}.
				\end{equation}
				\item \label{item: induction1local} For each $0 \le k \le K$, at least one of the following is true. 
				\begin{enumerate}
					\item \label{item: induction1partalocal}$\dist(x_k, \cM) \le  \frac{\dlb}{200\gammaub} \dist^{p-1}(y_k,S)$.
					\item \label{item: induction1partblocal}$\dist(x_k, \cM) \le \paren{1 - \frac{\eta \gammalb}{4}}^k \dist(x_0,\cM)$.
				\end{enumerate}
				\item \label{item: induction2local}For any $1\le k \le K$, we have $\dist(y_k, S) \le \dist(y_0,S) + \frac{1}{60}\frac{\dlb^2}{\betaub^2} \dist(x_0,\cM)$. 
			\end{enumerate}
			Moreover, item \eqref{item: induction1partalocal} implies $\|\nabla f_N(x_k)\| \le \frac{1}{100} \|\nabla f_T(y_k)\|$.
		\end{lemma}

		We are now ready to prove Theorem~\ref{thm: mainconvergencelocal}.
		\subsubsection{Proof of Theorem~\ref{thm: mainconvergencelocal}}\label{sec:proofmaintheorem}
		\begin{proof}
			By Item~\ref{item: localgrowthcondition} of Corollary~\ref{cor: localconditionsforgdpolyak}, decreasing $\deltaU$ if necessary, we assume that 
			\begin{align}\label{eqn: nonasymptoticupperf1local}
				f_N(x) \le C_\ub \dist^2(x,\cM) + \frac{C_\ub \dlb}{200\gammaub} \dist(x,\cM) \dist^{p-1}(P_\cM(x),S), \qquad \forall x \in B(\bar x, \deltaU).
			\end{align}
			Let $\rho_0$ be from Lemma~\ref{lem: innerlooplocal}. Note that by Lemma~\ref{lem: KLlocal}, for any $x \in U$, we have
			\begin{align}\label{eqn: KLpolyak}
				\frac{f(x) - f^*}{\|\nabla f(x)\|} \le  \ckl (f(x) - f^*)^{1/p}.
			\end{align}
			So there exists $\rho_1 >0$ such that for any $x \in B(\bar x, \rho_1)$ the point $x_+ = x - \frac{f(x) - f^*}{\|\nabla f(x)\|^2}  \nabla f(x)$ lies in  $B(\bar x, \rho_0)$. We set $\deltagdpk$ small enough so that  
			\begin{align}
				3\deltagdpk + 4\ckl p\lipf^{1/p}\deltagdpk^{1/p} + \sum_{l=0}^{\infty}\left( 10\ckl^{1/p}\lipf^{1/p+ 1/p^2}   \paren{1- \frac{\dlb^2}{20\betaub^2}}^{l/p^2}(6\deltagdpk)^{1/p^2}\right) < \min\{\rho_1, \deltaU\}.\label{eqn:delta_small}
			\end{align}
			Suppose that Item~\ref{item: mainthmitem1local} does not hold. {\color{blue} By our choice of $\xout$ in Algorithm~\ref{alg:GD-Polyak}, failure of Item~\ref{item: mainthmitem1local} implies that for any $1\le i \le I$,
				\begin{align}\label{eqn:contradiction_funciton_gap_lb}
					f(\tilde x_i) - f^* > 162C_\ub\paren{1- \frac{\eta\gammalb}{4}}^{2K} \deltagdpk^2 + D_\ub\paren{1-\frac{\eta \gammalb}{4}}^{\frac{Kp}{p-1}} \paren{\frac{1800\gammaub\deltagdpk}{\dlb}}^{\frac{p}{p-1}}.
				\end{align}
			} 
			We will show that Item~\ref{item: mainthmitem2local} holds, which will complete the proof. We set $\tilde y_i := P_\cM(\tilde x_i)$ and $y_i := P_\cM(x_i)$ throughout the proof. 	We will now apply induction to show that for any $1\le i\le I$, we have
			\begin{align}\label{eqn: inductioncontraction}
				\dist(\tilde x_i, \cM) \le  \frac{\dlb}{200\gammaub} \dist^{p-1}(\tilde y_{i},S) \leq	\dist(\tilde y_i, S) \le 3 \paren{1- \frac{\dlb^2}{20\betaub^2}}^{i-1}\deltagdpk,
			\end{align}
			
			\begin{align}\label{eqn: inductiongradientcomparison}
				\|\nabla f_N(\tilde x_i)\| \le \frac{1}{100}\|\nabla f_T(\tilde y_i)\|,
			\end{align}
			and 
			\begin{align}\label{eqn: inductionfinitelength}
				\|\tilde x_i - \bar x\| &\le  3\deltagdpk + 4\ckl p \lipf^{1/p}\deltagdpk^{1/p} + \sum_{l=0}^{i-2} \left(10\ckl^{1/p}\lipf^{1/p+ 1/p^2}  \paren{1- \frac{\dlb^2}{20\betaub^2}}^{l/p^2}(6\deltagdpk)^{1/p^2}\right) \notag\\
				& < \rho_1. 
			\end{align}
			Let us verify the base case $i=1$. To this end, suppose for the sake of contradiction that the first inequality in \eqref{eqn: inductioncontraction} fails, that is:
			\begin{align}\label{eqn: mainfirstlooplocal}
				\dist(\tilde x_1, \cM) > \frac{\dlb}{200\gammaub} \dist^{p-1}(\tilde y_1,S).
			\end{align} 
			Then by Lemma~\ref{lem: innerlooplocal}, we have
			$$
			\dist(\tilde x_1, \cM) \le \paren{1- \frac{\eta\gammalb}{4}}^{K} \dist(x_0,\cM) \le \paren{1- \frac{\eta\gammalb}{4}}^{K} \deltagdpk,
			$$
			and therefore \eqref{eqn: mainfirstlooplocal}  implies
			$$
			\dist(\tilde y_1, S) \le \paren{\frac{200\gammaub\dist(\tilde x_1,\cM)}{\dlb}}^{\frac{1}{p-1}} \le \paren{1-\frac{\eta \gammalb}{4}}^{\frac{K}{p-1}} \paren{\frac{200\gammaub\deltagdpk}{\dlb}}^{\frac{1}{p-1}}.
			$$
			As a result,  we have 
			\begin{align*}
				f(\tilde x_1)- f^* &= f_N(\tilde x_1) + f_T(\tilde x_1) - f^*\\
				&\le C_\ub \dist^2(\tilde x_1, \cM) +\frac{C_\ub \dlb}{200\gammaub} \dist(\tilde x_1,\cM) \dist^{p-1}(\tilde y_1,S) + D_\ub \dist^{p}(\tilde y_1, S)\\
				&\le 2C_\ub\paren{1- \frac{\eta\gammalb}{4}}^{2K} \deltagdpk^2 + D_\ub\paren{1-\frac{\eta \gammalb}{4}}^{\frac{Kp}{p-1}} \paren{\frac{200\gammaub\deltagdpk}{\dlb}}^{\frac{p}{p-1}},
			\end{align*}
			where the first inequality follows from \eqref{eqn: nonasymptoticupperf1local} and \eqref{eqn:func_control}. This implies that {\color{blue} the estimate~\eqref{eqn:contradiction_funciton_gap_lb}} holds, which is a contradiction. So~\eqref{eqn: mainfirstlooplocal} indeed fails to hold. On the other hand, Item~\ref{item: induction2local} of Lemma~\ref{lem: innerlooplocal} and the bound $D_{\lb}^2/60\beta_\ub^2 \leq 1$ from Item~\ref{item: propertyofdecompbregularitylocal} of Lemma~\ref{lem: propertyofdecomplocal}, we have
			$$
			\dist(\tilde y_1, S) \le \dist(y_0, S) + \frac{D_{\lb}^2}{60\beta_\ub^2}\dist(x_0,\cM) \le 3\deltagdpk.
			$$
			Therefore by the failure of~\eqref{eqn: mainfirstlooplocal} and by shrinking $\deltaU$ if necessary, we have 
			$$
			\dist(\tilde x_1, \cM) \le  \frac{\dlb}{200\gammaub} \dist^{p-1}(\tilde y_1,S)  \le \dist(\tilde y_1, S) \le 3\deltagdpk,
			$$
			thereby verifying \eqref{eqn: inductioncontraction}.
			Note moreover that Lemma~\ref{lem: innerlooplocal} ensures the inequality
			$$
			\|\nabla f_N(\tilde x_1)\|  \le \frac{1}{100}\|\nabla f_T(\tilde y_1)\|,
			$$
			thereby verifying \eqref{eqn: inductiongradientcomparison}.
			Moreover, using the triangle inequality, we deduce
			\begin{align*}
				\|\tilde x_1 - \bar x\|&\le \|x_0 -\bar x\| + \|\tilde x_1 - x_0\|\\
				&\le \deltagdpk + \frac{2}{L}\|\nabla f(x_0)\| + 4\ckl p \cdot (f(x_0) - f^*)^{1/p}\\
				&\le 3\deltagdpk + 4\ckl p \lipf^{1/p}\deltagdpk^{1/p},
			\end{align*}
			where the second inequality follows from \eqref{eqn:KL_bound_needed} and the last inequality follows from Lipschitz continuity of $f$ and $\nabla f$. Thus, we have verified all the claims \eqref{eqn: inductioncontraction},\eqref{eqn: inductiongradientcomparison}, and \eqref{eqn: inductionfinitelength}  for the base case $i=1$.
			
			Suppose now as the inductive assumption that \eqref{eqn: inductioncontraction},\eqref{eqn: inductiongradientcomparison}, and \eqref{eqn: inductionfinitelength} hold for any $1\le i \leq i_0$ and will show that they continue to hold for $i=i_0+1$. By the inductive assumption $\|\tilde x_{i_0} - \bar x\| \le \rho_1$ and our choice of $\rho_1$, we have $\|x_{i_0} - \bar x\| \le \rho_0$. So Lemma~\ref{lem: innerlooplocal} applies to the gradient descent sequence initialized at $x_{i_0}$. 
			Consequently, we deduce
			\begin{align*}
				\dist(\tilde y_{{i_0}+1}, S) &\le \dist(y_{i_0},S) + \frac{1}{60}\frac{\dlb^2}{\betaub^2} \dist(x_{i_0},\cM)\\
				&\le \paren{1- \frac{1}{10} \frac{\dlb^2}{\betaub^2}} \dist(\tilde y_{i_0}, S) + \frac{1}{20}\frac{\dlb^2}{\betaub^2} \dist(\tilde y_{i_0},S)\\
				&\le  \paren{1- \frac{1}{20} \frac{\dlb^2}{\betaub^2}} \dist(\tilde y_{i_0}, S)\\
				&\le  3 \paren{1- \frac{\dlb^2}{20\betaub^2}}^{{\color{blue}i_0}}\deltagdpk,
			\end{align*}
			where the first inequality follows from Lemma~\ref{lem: innerlooplocal}, the second inequality follows from Lemma~\ref{lem: progressinypolyaklocal}, and the last inequality follows by the inductive hypothesis \eqref{eqn: inductioncontraction}.
			By Lemma~\ref{lem: progressinypolyaklocal} and the inductive hypothesis \eqref{eqn: inductioncontraction}, we have $\dist(x_{i_0},\cM) \le 3\dist(\tilde y_{i_0}, S)\le 9 \deltagdpk$. Applying exactly the same argument that established the first inequality in \eqref{eqn: inductioncontraction} in the base case but using the inductive assumption instead yields
			$$
			\dist(\tilde x_{i_0+1}, \cM) \le  \frac{\dlb}{200\gammaub} \dist^{p-1}(\tilde y_{i_0+1},S)  \le \dist(\tilde y_{i_0+1}, S)
			$$
			and 
			$$
			\|\nabla f_N(\tilde x_{i_0+1})\| \le \frac{1}{100}\|\nabla f_T(\tilde y_{i_0+1})\|.
			$$
			Therefore, both~\eqref{eqn: inductioncontraction} and~\eqref{eqn: inductiongradientcomparison} hold for $i_0+1$. To show~\eqref{eqn: inductionfinitelength} for $i_0+1$, we note that by shrinking $U$ if necessary, we have
			\begin{align}\label{eqn: xtSandtildextS}
				\dist(x_{i_0}, S)&\le \dist(\tilde x_{i_0}, S) + \|x_{i_0} - \tilde x_{i_0}\| \notag \\
				&\le \dist(\tilde x_{i_0}, S) + \ckl (f(\tilde x_{i_0}) - f^*)^{1/p} \notag\\
				&\le \dist(\tilde x_{i_0}, S) + \ckl \lipf^{1/p} \dist^{1/p}(\tilde x_{i_0},S) \notag\\
				&\le 2\ckl \lipf^{1/p} \dist^{1/p}(\tilde x_{i_0},S),
			\end{align}
			where the second equality follows from~\eqref{eqn: KLpolyak}, the third inequality follows from Lipschitz continuity of $f$, and the last inequality follows by shrinking $U$ if necessary. Therefore, using the triangle inequality, we successively estimate 
			\begin{align*}
				\dist(\tilde x_{i_0+1},S) &\le \dist(\tilde x_{i_0},S) + \|x_{i_0} - \tilde x_{i_0}\| + \|\tilde x_{i_0+1} - x_{i_0}\|\\
				&\le \dist(\tilde x_{i_0},S) + \ckl (f(\tilde x_{i_0}) -f^*)^{1/p} + \frac{2}{L}\|\nabla f(x_{i_0})\| + 4\ckl p \cdot (f(x_{i_0}) - f^*)^{1/p}\\
				&\le \dist(\tilde x_{i_0},S) + \ckl \lipf^{1/p}\dist^{1/p}(\tilde x_{i_0}, S) + 2\dist(x_{i_0},S) + 4\ckl p \lipf^{1/p} \dist^{1/p}(x_{i_0},S)\\
				&\le \dist(\tilde x_{i_0},S) + 5\ckl\lipf^{1/p}\dist^{1/p}(\tilde x_{i_0}, S) +  8p\ckl^{1/p + 1/p^2} \lipf^{1/p+ 1/p^2}\dist^{1/p^2}(\tilde x_{i_0}, S)\\
				&\le \dist(\tilde x_{i_0},S) + 10p\ckl^{1/p+1/p^2} \lipf^{1/p+ 1/p^2}\dist^{1/p^2}(\tilde x_{i_0}, S),
			\end{align*}
			where the second inequality follows from~\eqref{eqn: KLpolyak} and \eqref{eqn:KL_bound_needed}, the fourth inequality follows from~\eqref{eqn: xtSandtildextS}, and the last inequality follows by shrinking $U$ if necessary. 
			Consequently,
			\begin{align*}
				\dist(\tilde x_{i_0+1},S) &\le  \dist(\tilde x_{i_0},S) +  10p\ckl^{1/p+1/p^2} \lipf^{1/p+ 1/p^2}  \paren{1- \frac{\dlb^2}{20\betaub^2}}^{(i_0-1)/p^2}(6\deltagdpk)^{1/p^2}\\
				&\le 3\deltagdpk + 4\lipf^{1/p} +  \sum_{l=0}^{i_0-1}  10p\ckl^{1/p+1/p^2} \lipf^{1/p+ 1/p^2}  \paren{1- \frac{\dlb^2}{20\betaub^2}}^{l/p^2}(6\deltagdpk)^{1/p^2},
			\end{align*}
			where the first inequality follows from $\dist(\tilde x_{i_0}, S) \le \dist(\tilde x_{i_0}, \cM) + \dist(\tilde y_{i_0}, S)$ and the inductive hypothesis, and the second inequality follows from the inductive hypothesis. The above term is smaller than $\rho_1$ by the initial estimate \eqref{eqn:delta_small}. Thus, \eqref{eqn: inductionfinitelength} holds, and the induction is complete. 
			In particular, we have 
			\begin{align}\label{eqn: finalbound}
				\dist(\tilde x_I, \cM) \le \frac{\dlb}{200\gammaub} \dist^{p-1}(\tilde y_I, S)\leq\dist(\tilde y_I, S) \le 3 \paren{1- \frac{\dlb^2}{20\betaub^2}}^{I-1}\deltagdpk.
			\end{align}
			{\color{blue}	Finally,  we have
				\begin{align*}
					f(\xout) - f^\ast &\le f(\tilde x_I) - f^\ast \\
					& \le C_\ub \dist^2(\tilde x_I, \cM) + \frac{C_\ub \dlb}{200\gammaub} \dist(\tilde x_I,\cM) \dist^{p-1}(\tilde y_{I},S) + D_\ub \dist^p(\tilde y_I, S) \\
					&\le 2C_\ub\paren{\frac{\dlb}{200\gammaub}}^2\dist^{2p-2}(\tilde y_{I},S)+ D_\ub \dist^p(\tilde y_I, S),
				\end{align*}
				where the second inequality follows from~\eqref{eqn: nonasymptoticupperf1local} and Item~\ref{item: localgrowthconditionravine} of Assumption~\ref{assum:tangent}, and the last inequality follows from~\eqref{eqn: finalbound}. Item~\ref{item: mainthmitem2local} follows by replacing $\dist(\tilde y_{T},S)$ with the upper bound we obtained in equation~\eqref{eqn: finalbound}.
			}
		\end{proof}

		%
		
		\section{Proof of technical lemmas}\label{sec:proofs_technical_lemmas}
		\subsection{Proof of Lemma~\ref{lem: propertyofdecomplocal}}
		\begin{proof}
			We first prove Item \ref{item: propertyofdecomplipschitzgradientlocal}. To this end,  setting $y = P_\cM(x)$, we compute
			\begin{align*}
				\|\nabla f_T(x)-\nabla f_T(y)\|&=\|(\nabla P_{\cM}(x)-\nabla P_{\cM}(y)) \nabla f(y)\|\\
				&=\|(\nabla P_{\cM}(x)-\nabla P_{\cM}(y)) \nabla f_T(y)\|\\
				&\leq C_{\cM}\|x-y\|\cdot \|\nabla f_T(y)\|.
			\end{align*}
			where the second equality follows from the identities $\nabla P_{\cM}(y)N_{\cM}(y)=\nabla P_{\cM}(x)N_{\cM}(y)=0$. Thus, item  \ref{item: propertyofdecomplipschitzgradientlocal} is proved.
			Item~\ref{item: propertyofdecompsizeofgradientlocal} follows directly from Item \ref{item: propertyofdecomplipschitzgradientlocal}. Item~\ref{item: propertyofdecompbregularitylocal} follows from the Cauchy-Schwarz inequality, Item \ref{item: localaimingtowardsolution} of Corollary~\ref{cor: localconditionsforgdpolyak}, and the assumption $f_T(y)-f^*=\Theta(1)\cdot \dist^p(y,S)$. Item~\ref{it:lower_bound_grad_norm} follows directly from Item~\ref{item: localaimingtowardmanifold} of Corollary~\ref{cor: localconditionsforgdpolyak} and the Cauchy-Schwarz inequality.
			
			It remains to verify Item~\ref{it:norm_comp}. We begin by writing
			$$\|\nabla f(x)\|^2 = \|\nabla f_N(x)\|^2 + \|\nabla f_T(x)\|^2 + 2\dotp{\nabla f_N(x), \nabla f_T(x)}.$$
			Next,  we successively compute 
			\begin{align*}
				|\dotp{\nabla f_N(x), \nabla f_T(x)}| &\leq \|P_{T_\cM(y)}\nabla f_N(x)\|\|\nabla f_T(x)\|\notag \\
				&\leq o(1)\cdot \dist(x,\cM) \|\nabla f_T(x)\| \notag \\
				&\leq o(1)\cdot (\dist^2(x,\cM) + \|\nabla f_T(x)\|^2) \\
				&\leq o(1)\cdot (\|\nabla f_N(x)\|^2+ \|\nabla f_T(x)\|^2),
			\end{align*}
			where the first inequality follows from the inclusion $\nabla f_T(y) \in T_{\cM}(y)$ and Cauchy-Schwarz, the second inequality follows from Item~\ref{item: localstronga} of Corollary~\ref{cor: localconditionsforgdpolyak}, the third inequality follows from Young's inequality, and the final inequality follows from Item~\ref{it:lower_bound_grad_norm} of the present lemma, using the estimates $\|\nabla f_T(x)\|=\Theta(1)\cdot \|\nabla f_T(y)\|=\Theta(1)\cdot \dist^{p-1}(y,S)$, which follow directly from Item \ref{item: propertyofdecompsizeofgradientlocal} and equation \eqref{eqn:control_tan_grad}. This completes the proof.

		\end{proof}
		\subsection{Proof of Lemma~\ref{lem: KLlocal}}
		To simplify notation, we write $y = P_{\cM}(x)$. Using Item~\ref{it:norm_comp} of Lemma~\ref{lem: propertyofdecomplocal}  we obtain 
		\begin{equation}\label{eqn:yeahboy}
			\|\nabla f(x)\|^2 = (1+ o(1)) \|\nabla f_N(x)\|^2 + (1+o(1)) \|\nabla f_T(x)\|^2.\notag
		\end{equation}
		Note the estimates $\|\nabla f_T(x)\|=\Theta(1)\cdot \dist^{p-1}(y,S) =\Theta(1) \cdot (f_T(x) - f^*)^{(p-1)/p}$.
		Moreover, Item \ref{it:lower_bound_grad_norm} of Lemma~\ref{lem: propertyofdecomplocal} reads as: 
		\begin{align*}
			\|\nabla f_N(x)\|&\geq \Theta(1)\cdot \|x-y\|+o(1)\cdot\dist^{p-1}(y,S).
		\end{align*}
		Therefore, continuing with \eqref{eqn:yeahboy} we deduce
		\begin{align*}
			\|\nabla f(x)\|^2&\geq \Theta(1)\left(\|x-y\|^2+ o(1)\|x-y\|\dist^{p-1}(y,S)\right)+\Theta(1)\dist^{2p-2}(y,S)\\
			&\geq \Theta(1)(\|x-y\|^2+\dist^{2p-2}(y,S))\\
			&\geq \Theta(1)(|f_N(x)|+(f_T(x)-f^*)^{\frac{2p-2}{p}})\\
			&\geq \Theta(1)\left(|f_N(x)|^{\frac{2p-2}{p}}+(f_T(x)-f^*)^{\frac{2p-2}{p}}\right)\\
			&\geq \Theta(1) (f(x) - f^*)^{\frac{2p-2}{p}},
		\end{align*}
		where the second inequality follows from Young's inequality, and the third follows from Item~\ref{item: localgrowthcondition} of Corollary~\ref{cor: localconditionsforgdpolyak} and Young's inequality. Taking square roots of both sides completes the proof.

		\subsection{Proof of Lemma~\ref{lem: progressinypolyaklocal}}
		{\color{blue}
			In order to prove this lemma, we will first state and prove the following auxiliary lemma, which shows that the Polyak stepsizes induced by $f$ and $f_T$ are almost the same when the point is close to the ravine.
			\begin{lemma}\label{lem: polyakstepsizeboundlocal}
				For any $x \in U$,  define $y := P_{\cM}(x)$.  Suppose that the inequality $\|\nabla f_N(x)\| \le \frac{1}{100}\|\nabla f_T(y)\|$ holds.  Then, by shrinking $U$ if necessary,  the inequalities hold:
				\begin{align}
					\|x-y\|&=O( \dist^{p-1}(y,S))\label{eqn:place_back_later}\\
					\tfrac{39}{40}(f_T(y) - f^*) &\le f(x) - f^* \le \tfrac{41}{40}(f_T(y) - f^*),\label{it:1_bit_suffice}\\
					\tfrac{19}{20}\|\nabla f_T(y)\|^2 &\le \|\nabla f(x)\|^2 \le \tfrac{21}{20}\|\nabla f_T(y)\|^2. \label{it:2_bit_suffice}
				\end{align}
				Consequently, the Polyak stepsize satisfies 
				\begin{equation}\label{eqn:polyak_comp}
					\frac{9}{10}\frac{f_T(y) - f^*}{\|\nabla f_T(y)\|^2} \le \frac{f(x) - f^*}{\|\nabla f(x)\|^2} \le \frac{11}{10}\frac{f_T(y) - f^*}{\|\nabla f_T(y)\|^2}.
				\end{equation}
			\end{lemma}
			\begin{proof}
				The estimate \eqref{eqn:polyak_comp} follows immediately from \eqref{it:1_bit_suffice} and \eqref{it:2_bit_suffice} through elementary algebraic manipulations. Therefore, we focus on proving the estimates \eqref{eqn:place_back_later}-\eqref{it:2_bit_suffice}.
				
				We begin by proving \eqref{eqn:place_back_later} and \eqref{it:1_bit_suffice}. To this end, we note that
				\begin{equation}\label{eqn:ref_later}
					\begin{aligned}
						\gammalb \| x-y\| +o(1)\dist^{p-1}(y,S) &\le \|\nabla f_N(x)\|\\
						&\le \frac{1}{100}\|\nabla f_T(y)\|\\
						&\le \frac{\betaub}{100} \dist^{p-1}(y,S),
					\end{aligned}
				\end{equation}
				where the first inequality follows from Item~\ref{item: localaimingtowardmanifold} of Corollary~\ref{cor: localconditionsforgdpolyak} and Cauchy-Schwarz, and the last inequality follows from \eqref{eqn:control_tan_grad}. Therefore, the claim \eqref{eqn:place_back_later} holds, that is $\|x-y\| = O(\dist^{p-1}(y,S))$.
				As a result, by shrinking $U$ if necessary, we have
				\begin{align}
					|f_N(x)| &\le C_\ub \|x -y\|^2 + o(1)\|x-y\|\dist^{p-1}(y,S)\notag\\
					&=O(\dist^{2p-2}(y,S))\notag\\
					&\le \frac{\dlb}{40} \dist^p(y,S)\notag\\
					&\le \frac{1}{40}(f_T(y) - f^*),\label{eqn:broop}
				\end{align}
				where the first and the last inequalities follow from Item~\ref{item: localgrowthcondition} of Corollary~\ref{cor: localconditionsforgdpolyak}. Writing $f=f_T+f_N$ and using the estimate \eqref{eqn:broop} directly yields
				$$\frac{39}{40}(f_T(y) - f^*)\leq f(x) - f^* \le \frac{41}{40}(f_T(y) - f^*),$$
				thereby completing the proof of Claim~\ref{it:1_bit_suffice}.
				
				Next, we prove Claim~\ref{it:2_bit_suffice}. By shrinking $U$ if necessary, we assume that Item~\ref{item: propertyofdecompsizeofgradientlocal} of Lemma~\ref{lem: propertyofdecomplocal} holds. Combining this with the assumption $\|\nabla f_N(x)\| \le \frac{1}{100}\|\nabla f_T(y)\|$, we have
				\begin{align*}
					\|\nabla f(x)\| \ge  \| \nabla f_T(x)\| - \|\nabla f_N(x)\|
					\ge \frac{49}{50} \|\nabla f_T(y)\|.
				\end{align*}
				Squaring both sides, we obtain the lower bound. Similarly, by Item~\ref{item: propertyofdecompsizeofgradientlocal} of Lemma~\ref{lem: propertyofdecomplocal} and our assumption that $\|\nabla f_N(x)\| \le \frac{1}{100}\|\nabla f_T(y)\|$, we have
				\begin{align}\label{eqn: boundongradusingnablaf2}
					\|\nabla f(x)\|\le \|\nabla f_N(x)\| + \|\nabla f_T(x)\|
					\le \frac{51}{50} \|\nabla f_T(y)\|.
				\end{align}
				Squaring both sides, we obtain the upper bound. 
		\end{proof}}

		{\color{blue}We now prove Lemma~\ref{lem: progressinypolyaklocal}.}
		
		\begin{proof}[Proof of Lemma~\ref{lem: progressinypolyaklocal}]
			To simplify notation, we set $\etax := \frac{f(x) - f^*}{\|\nabla f(x)\|^2}$ and $\etay := \frac{f(y) - f^*}{\|\nabla f_T(y)\|^2}$. By shrinking $U$ if necessary, we suppose that the conclusions of  Lemma~\ref{lem: polyakstepsizeboundlocal} hold. Consequently, we have
			\begin{align}\label{eqn: etaxbound}
				\frac{9}{10}\etay \le \etax \le \frac{11}{10}\etay,  \qquad  \|\nabla f(x)\|^2 \le \frac{21}{20}\|\nabla f_T(y)\|^2.
			\end{align}
			By shrinking $U$ if necessary, we assume that the conclusions of Lemma~\ref{lem: propertyofdecomplocal} hold as well. Choosing any $\bar y\in P_S(y)$, we now  estimate:
			\begin{align*}
				\dist^2(y - \etax \nabla f_T(y), S) &\le \|y -  \etax \nabla f_T(y) - \bar y\|^2\\
				&\le \|y - \bar y\|^2 - \frac{9}{5}\etay \dotp{\nabla f_T(y), y - \bar y} + \frac{121}{100}\etay^2\|\nabla f_T(y)\|^2\\
				&\le \|y - \bar y\|^2 -\frac{81}{50}\etay(f(y)- f^*) + \frac{121}{100}\etay^2\|\nabla f_T(y)\|^2\\
				&\le \|y - \bar y\|^2 -\frac{2}{5}  \frac{(f(y) - f^*)^2}{\|\nabla f_T(y)\|^2} \\
				&\le \paren{1- \frac{2}{5} \frac{\dlb^2}{\betaub^2}} \| y -\bar y\|^2,
			\end{align*}
			where the second inequality follows from~\eqref{eqn: etaxbound} and the fact that $\dotp{\nabla f_T(y), y - \bar y}$  is positive, the third inequality follows from Item~\ref{item: propertyofdecompbregularitylocal} of Lemma~\ref{lem: propertyofdecomplocal}, the fourth inequality follows from the definition of $\etay$, and the last inequality follows from Item~\ref{item: localsizeofgd} and Item~\ref{item: localgrowthcondition} of Corollary~\ref{cor: localconditionsforgdpolyak}. Taking the square root of both sides, we have
			\begin{align}\label{eqn: dist(y,s)boundintermediate}
				\dist(y - \etax \nabla f_T(y), S) \le \paren{1-  \frac{1}{5} \frac{\dlb^2}{\betaub^2}} \dist(y, S).
			\end{align}
			By shrinking $U$ if necessary, we assume that $P_\cM$ is well-defined and $C^1$ at $x - \etax \nabla f(x)$ for any $x\in U$. A first-order expansion of $P_\cM$ at $x$ yields 
			\begin{align}
				P_\cM(x - \etax \nabla f(x)) = y -  \etax P_{T_\cM(y)}(\nabla f_T(x)) -  \etax P_{T_\cM(y)}(\nabla f_N(x)) + \mathtt{error} \label{eqn: ypluesexpansionpolyak}.
			\end{align}
			Here, we have
			\begin{align*}
				\|\mathtt{error}\| &\le   \etax \|(\nabla P_\cM(x) -P_{T_\cM(y)})\nabla f(x)\|+  C_\cM \|   \etax \nabla f(x) \|^2\\
				& \le 2\etay C_\cM \| x-y\| \|\nabla f_T(y)\| + 2 \etay^2 C_\cM \|\nabla f_T(y)\|^2\\
				&= 2C_\cM \|x-y\| \frac{f(y) - f^*}{\|\nabla f_T(y)\|} + 2C_\cM \frac{(f(y) - f^*)^2}{\|\nabla f_T(y)\|^2}\\
				&\le \frac{20 C_\cM }{9} \| x-y\| \dist(y, S) + \frac{200 C_\cM}{81} \dist^2(y,S)\\
				&\le o(1)\dist(y, S),
			\end{align*}
			where the second inequality follows from~\eqref{eqn: etaxbound}, the third inequality follows from~\eqref{eqn: etaxbound} and Item~\ref{item: propertyofdecompbregularitylocal} of  Lemma~\ref{lem: propertyofdecomplocal}. In addition, we have
			\begin{align}
				\etax \|P_{T_\cM(y)}(\nabla f_T(x) - \nabla f_T(y))\| &\le \frac{11}{10}\etay C_\cM \| x-y\| \|\nabla f_T(y)\|\notag\\
				&= \frac{11C_\cM}{10} \| x-y\| \frac{f_T(y) - f^*}{\|\nabla f_T(y)\|}\notag\\
				&\le \frac{11C_\cM}{9} \| x-y\| \dist(y, S)\notag\\
				&= o(1)\dist(y, S),\label{eqn:getirdone}
			\end{align}
			where the first inequality follows from Item~\ref{item: propertyofdecomplipschitzgradientlocal} of Lemma~\ref{lem: propertyofdecomplocal} and~\eqref{eqn: etaxbound}, the second inequality follows from Item~\ref{item: propertyofdecompbregularitylocal} of Lemma~\ref{lem: propertyofdecomplocal}. Moreover, by \eqref{eqn:place_back_later} we have $\|x-y\| = O(\dist^{p-1}(y,S))$. Consequently, we deduce
			\begin{align}
				\etax \|P_{T_{\cM}(y)} \nabla f_N(x)\| & =o(1)\etay \|x-y\|\notag\\
				&= o(1)\dist^{p-1}(y,S)\frac{f_T(y) - f^*}{\|\nabla f_T(y)\|^2} \notag \\
				&=  o(1)\dist(y, S),\label{eqn:wegonnaneed}
			\end{align}
			where the first inequality follows from Item~\ref{item: localstronga} of Corollary~\ref{cor: localconditionsforgdpolyak}, and the second and third follow from \eqref{eqn:control_tan_grad} and \eqref{eqn:func_control}.
			Combining all the estimates we have
			\begin{align*}
				\dist(y_+,S) &= \dist(P_\cM(x-  \etax \nabla f(x)), S)\\
				&\le \dist(y - \etax \nabla f_T(y), S) + \etax\|P_{T_\cM(y)}(\nabla f_N(x)) \| +  \etax\|P_{T_\cM(y)}(\nabla f_T(x)) - \nabla f_T(y)\|  + \|\mathtt{error}\|\\
				&\le \paren{1-  \frac{1}{5} \frac{\dlb^2}{\betaub^2}}\dist(y, S) + o(1)\dist(y,S),
			\end{align*}
			where the first inequality follows from the triangle inequality and \eqref{eqn: ypluesexpansionpolyak}, and the second inequality follows from \eqref{eqn: dist(y,s)boundintermediate}, \eqref{eqn:getirdone}, and \eqref{eqn:wegonnaneed}. Therefore, Item~\ref{item: progressinypolyaklocal} holds by shrinking $U$ if necessary. Moreover,  
			\begin{align*}
				\|\etax \nabla f(x)\| &\le \frac{3}{2} \etay\|\nabla f_T(y)\|\\
				&=\frac{3}{2}\frac{f_T(y) - f^*}{\|\nabla f_T(y)\|}\\
				&\le 2\dist(y,S),
			\end{align*}
			where the first inequality follows from~\eqref{eqn: etaxbound} and the second inequality follows from Item~\ref{item: propertyofdecompbregularitylocal} of Lemma~\ref{lem: propertyofdecomplocal}. Thus, we deduce
			$$
			\dist(x_+,\cM) \le \dist(x,\cM) + \|\etax \nabla f(x)\| \leq 3\dist(y,S),
			$$
			thereby completing the proof of Item~\ref{item: xMafterpolyaklocal}.
		\end{proof}

		\subsection{Proof of Lemma~\ref{lem: innerlooplocal}}
		
		{\color{blue} Before proving Lemma~\ref{lem: innerlooplocal}, we state and prove three auxiliary lemmas. 
			
			\subsubsection{Auxiliary lemmas}
			The following lemma shows that constant size gradient descent steps can shrink the distance to the ravine at a linear rate up to the tolerance $\dist(x,\cM) = o(1)\cdot \dist^{p-1} (P_{\cM}(x), S)$.
			\begin{lemma}\label{lem: GDshrinksdisttoMlocal}
				Suppose that $\eta \le \frac{\gammalb}{2\gammaub^2}$. By shrinking $U$ if necessary, there exists $\delta >0$ such that for any $x \in U$, we have 
				$$
				\dist\left(x - \eta \nabla f(x), \cM\right) \le \left(1- \frac{\eta \gammalb}{2}\right) \dist(x, \cM) +o(1)\cdot\eta\cdot \dist^{p-1}(P_\cM(x),\cM).
				$$
			\end{lemma}
			\begin{proof}
				Setting $y:= P_\cM(x)$, we first estimate:
				\begin{align*}
					\| x - \eta \nabla f_N(x) - y\|^2 &= \|x - y\|^2 - 2\eta \dotp{\nabla f_N(x), x-y} + \eta^2 \|\nabla f_N(x)\|^2 \\
					&\le (1- 2\eta\gammalb+2\eta^2 \gammaub^2) \|x - y\|^2  + o(1)\eta^2\dist^{2p-2}(y,S)\\
					&\le (1- \eta \gammalb)\|x-y\|^2 + o(1)\eta^2\dist^{2p-2}(y,S),
				\end{align*}
				where the first inequality follows from Item~\ref{item: localaimingtowardmanifold} and Item~\ref{item: localsizeofgd} of Corollary~\ref{cor: localconditionsforgdpolyak}, and the last inequality follows from our assumption that $\eta \le \frac{\gammalb}{2\gammaub^2}$.  Taking square root, we have
				\begin{equation}\label{eqn:main_contr}
					\| x - \eta \nabla f_N(x) - y\| \le \paren{1- \frac{\eta \gammalb}{2}} \dist(x,\cM) + o(1)\cdot \eta\cdot\dist^{p-1}(y,S),
				\end{equation}
				where we use the inequality $\sqrt{1-x} \leq \sqrt{1-x/2}$.
				On the other hand, since the projection $\nabla P_\cM$ is $C_{\cM}$-Lipschitz,  $P_\cM(y) = y$, and $\nabla P_\cM(y)\nabla f_T(y) = \nabla f_T(y)$, we have
				\begin{align}\label{eqn: projectiontaylor}
					\|P_\cM(y - \eta \nabla f_T(y)) - (y - \eta \nabla f_T(y))\| \le \frac{C_\cM \eta^2}{2} \|\nabla f_T(y)\|^2.
				\end{align}
				Combining, we have 
				\begin{align*}
					\dist(x - \eta \nabla f(x), \cM) &\le \| x- \eta \nabla f_N(x) - \eta \nabla f_T(x) - P_\cM(y- \eta \nabla f_T(y)) \|\\
					&\le \|x- \eta \nabla f_N(x) - \eta \nabla f_T(x)  -(y - \eta \nabla f_T(y))\| + \frac{C_\cM \eta^2}{2} \|\nabla f_T(y)\|^2\\
					& \le \|x- \eta \nabla f_N(x) - y\| + \eta \|\nabla f_T(x) - \nabla f_T(y)\| + \frac{C_\cM \eta^2}{2} \|\nabla f_T(y)\|^2\\
					&\le \left(1- \frac{\eta \gammalb}{2}\right)\|x-y\|+o(1)\cdot\eta\cdot\dist^{p-1}(y,S))\\
					&\quad + C_\cM\cdot\eta\cdot \|x-y\|\cdot \|\nabla f_T(y)\|+ \frac{C_{\cM}\eta^2}{2}\|\nabla f_T(y)\|^2\\
					&= \left(1- \frac{\eta \gammalb}{2}\right)\|x-y\|+ o(1)\cdot \eta\cdot\dist^{p-1}(y,S),
				\end{align*}
				where the second inequality follows from~\eqref{eqn: projectiontaylor}, the fourth inequality follows from \eqref{eqn:main_contr} and Item \ref{item: propertyofdecomplipschitzgradientlocal} of Lemma~\ref{lem: propertyofdecomplocal}, and the last equality follows from Item~\ref{item: localsizeofgd} of Corollary~\ref{cor: localconditionsforgdpolyak}.
			\end{proof}
			
			Lemma~\ref{lem: GDshrinksdisttoMlocal} shows that gradient descent with constant stepsize rapidly approach the ravine. In the process, however, the iterates may move away from the set of minimizers $S$. The following lemma provides a bound on this negative effect.
			\begin{lemma}\label{lem: ychangeingdlocal}
				Suppose that we set $\eta \le \frac{1}{\betaub}$. Then for any $x$ sufficiently close to $\bar x$, the points $y = P_\cM(x)$ and $y_+ = P_\cM(x- \eta \nabla f(x))$  satisfy
				\begin{equation}\label{eqn:base_ineq_srror}
					\dist(y_+, S) \le \dist(y,S)\paren{1 - \tfrac{\eta \dlb}{2} \dist^{p-2}(y,S)} + o(1)\cdot\eta \dist(x, \cM) ,
				\end{equation}
				and 
				\begin{equation}\label{eqn:lower_bound_lem}
					\dist(y_+, S) \ge \paren{1- \eta o(1)}\cdot\dist(y,S) + o(1) \cdot \gdstep \dist(x,\cM).
				\end{equation}
			\end{lemma}
			\begin{proof}
				By shrinking $U$ if necessary, we may assume that all items in Lemma~\ref{lem: propertyofdecomplocal} hold. 
				We will first show the estimate:
				\begin{equation}\label{eqn:intermediate}
					\dist(y - \eta \nabla f_T(y), S) \le \dist(y,S)\paren{1- 0.9\eta \dlb\dist(y,S)^{p-2}} + o(1)\eta \dist^{p-1}(y,S).
				\end{equation}
				To this end, for any $\bar y\in P_S(y)$ we successively estimate
				\begin{align}
					\dist^2(y - \eta \nabla f_T(y), S) &\le \|y - \eta \nabla f_T(y) - \bar y\|^2\notag\\
					&= \|y - \bar y\|^2 - 2\eta \dotp{\nabla f_T(y), y -\bar y}  + \eta^2 \|\nabla f_T(y)\|^2 \notag\\
					&\le \|y - \bar y\|^2 - \frac{9}{5}\eta (f_T(y) - f^*)  + \eta^2 \|\nabla f_T(y)\|^2 \label{eqn: ychangeeq2}\\
					&\le \|y - \bar y\|^2 - \frac{9}{5}\eta \dlb \| y - \bar y\|^p +\eta^2 \betaub^2\|y- \bar y\|^{2p-2}\label{eqn: ychangeeq3}\\
					&= \|y - \bar y\|^2\left(1- \frac{9}{5}\eta \dlb \| y-\bar y\|^{p-2} + o(1)\eta\|y- \bar y\|^{p-2}\right),\notag 
				\end{align}
				where the estimate~\eqref{eqn: ychangeeq2} follows from Item~\ref{item: propertyofdecompbregularitylocal} of Lemma~\ref{lem: propertyofdecomplocal}, and the estimate~\eqref{eqn: ychangeeq3} follows from \eqref{eqn:control_tan_grad} and \eqref{eqn:func_control} and the inequality $p> 1$.
				By shrinking $U$ if necessary and taking the square root, we have
				$$
				\dist(y - \eta \nabla f_T(y), S) \le \|y - \bar y\| \paren{1 - 0.9 \eta \dlb\|y - \bar y\|^{p-2}} + o(1)\eta \dist^{p-1}(y,S),
				$$
				thereby verifying the claimed estimate \eqref{eqn:intermediate}. Additionally, Item \ref{item: propertyofdecompsizeofgradientlocal} of Lemma~\ref{lem: propertyofdecomplocal} yields:
				\begin{align}\label{eqn: upperboundnabla2f(x)}
					\|\nabla f_T(x)\| =O(\|\nabla f_T(y)\|).
				\end{align}
				Next, a Taylor expansion of $P_\cM$ at $x$ yields
				\begin{align}
					P_\cM(x -\eta \nabla f(x)) = y -  \eta P_{T_\cM(y)}(\nabla f_T(x))-\eta P_{T_\cM(y)}(\nabla f_N(x)) + \mathtt{error} \label{eqn: ypluesexpansion}.
				\end{align}
				Here,
				\begin{align*}
					\|\mathtt{error}\| &\le \eta \|(\nabla P_\cM(x) -P_{T_\cM(y)})\nabla f(x)\|+  \frac{\eta^2C_\cM}{2} \| \nabla f(x) \|^2\\
					& \le \eta C_\cM \| x-y\|\|\nabla f(x)\| + \eta^2C_\cM  (\|\nabla f_N(x)\|^2 + \|\nabla f_T(x)\|^2)\\
					&= o(1)\eta \|x-y\| + o(1)\eta\dist^{p-1}(y,S),
				\end{align*}
				where the second inequality follows from Lipschitz continuity of $\nabla P_\cM$ and the equality follows from~\eqref{eqn: upperboundnabla2f(x)} and Item~\ref{item: localsizeofgd} of Corollary~\ref{cor: localconditionsforgdpolyak} Additionally, using Item \ref{item: propertyofdecomplipschitzgradientlocal} of Lemma~\ref{lem: propertyofdecomplocal} we estimate
				\begin{align*}
					\|P_{T_\cM(y)}(\nabla f_T(x)) - \nabla f_T(y)\|\leq  \|\nabla f_T(x) - \nabla f_T(y)\|&\le C_\cM \| x-y\| \|\nabla f_T(y)\|\\
					&= o(1)\dist^{p-1}(y, S).
				\end{align*}
				Moreover, Item~\ref{item: localstronga} of Corollary~\ref{cor: localconditionsforgdpolyak} shows $
				P_{T_\cM(y)}(\nabla f_N(x)) =o(1)\|x-y\|.
				$
				Combining with \eqref{eqn: ypluesexpansion} and using the triangle inequality, we successively compute:
				\begin{align*}
					\dist(y_+,S) &= \dist(P_\cM(x- \eta \nabla f(x)), S)\\
					&\leq \dist(y-\eta P_{T_\cM(y)}(\nabla f_T(x)), S)+\eta\|P_{T_\cM(y)}(\nabla f_N(x))\|+\|\mathtt{error}\|\\
					&\le \dist(y - \eta \nabla f_T(y), S) + \eta\|P_{T_\cM(y)}(\nabla f_N(x)) \| + \eta \|P_{T_\cM(y)}(\nabla f_T(x)) - \nabla f_T(y)\|  + \|\mathtt{error}\|\\
					&\le \|y - \bar y\| \paren{1 - 0.9 \eta \dlb \|y - \bar y\|^{p-2}} + o(1)\eta \| x-y\| + o(1)\eta \dist^{p-1}(y, S)\\
					&\le \|y - \bar y\| \paren{1 - \frac{\eta \dlb}{2} \|y - \bar y\|^{p-2}} + o(1)\eta \| x-y\|
				\end{align*}
				thus completing the proof of \eqref{eqn:base_ineq_srror}.	
				
				We next prove \eqref{eqn:lower_bound_lem}. To this end, using \eqref{eqn: ypluesexpansion} again we compute 
				\begin{align*}
					\dist(y_+,S) &= \dist(P_\cM(x- \eta \nabla f(x)), S)\\
					&\ge \dist(y - \eta \nabla f_T(y), S) - \eta\|P_{T_\cM(y)}(\nabla f_N(x)) \| -\eta \|P_{T_\cM(y)}(\nabla f_T(x)) - \nabla f_T(y)\|  - \|\mathtt{error}\|\\
					&\ge \dist(y,S) - \eta \|\nabla f_T(y)\| +  o(1)\eta \| x-y\| + o(1)\eta \dist^{p-1}(y, S)\\
					&\ge \paren{1 - \eta o(1) }\dist(y,S) + o(1)\eta \| x-y\|,
				\end{align*}
				which completes the proof of \eqref{eqn:lower_bound_lem}.
			\end{proof}

			An important step of our proof is to show that when initialized sufficiently close to a minimizer $\bar x$ of $f$, the iterates of Algorithm~\ref{alg:GD-Polyak} stay with a neighborhood of $\bar x$. This can be guaranteed by the {\L}ojasiewicz  property (Lemma~\ref{lem: KLlocal}) and a standard finite length gurarantee~\cite[Lemma 2.6]{attouch2013convergence}, which we now record. We provide a proof sketch for completeness.

			\begin{lemma}[Finite length]\label{lem: finitelength}
				Suppose that the $C^1$-smooth function $f$ has {\L}ojasiewicz property with exponent $\alpha$ at a minimizer $\bar x \in \RR^d$, that is there exists a constant $\ckl >0$ and a neighborhood $U$ of $\bar x$ such that for all $x \in U$, we have 
				\begin{align}\label{eqn: KLdef}
					\ckl\|\nabla f(x)\| \ge  (f(x) - f^*)^{\alpha}\qquad \forall x\in U.
				\end{align} 
				Suppose, moreover, that $\nabla f$ is $L$-Lipschitz continuous on some neighborhood $U$ where~\eqref{eqn: KLdef} is satisfied. Let $\{x_k\}_{k\geq 0}$ be a sequence generated by gradient descent with constant stepsize $\eta \le \frac{1}{L}$. Then there exists a constant $\rho_0$ such that for all initialization $x_0 \in B(\bar x, \rho_0)$, the following holds for all $k \ge 1$:
				\begin{align}\label{eqn: finitelength}
					\|x_k - x_0\| \le \frac{2}{L} \|\nabla f(x_0)\| + \frac{4\ckl}{1-\alpha} (f(x_0) - f^*)^{1-\alpha}.
				\end{align}
				Moreover, $x_i \in U$ for all $k\ge 1$.
			\end{lemma}
			\begin{proof}
				Without loss of generality, we assume that $U = B(\bar x, \delta)$ for some $\delta >0$. Let $\rho_1$ be small enough so that $\|\nabla f(x)\| \le \frac{\delta L}{2}$ for all $x \in B(\bar x, \rho_1)$. Let $\rho_0$ be small enough such that for any $x \in B(\bar x, \rho_0)$, we have
				\begin{align}\label{eqn: initialization}
					\rho_0 + \frac{2}{L} \|\nabla f(x)\| + \frac{4\ckl}{1-\alpha}(f(x) - f^*)^{1-\alpha} < \min\left\{\rho_1, \frac{\delta}{2}\right\}.
				\end{align}
				We prove the result by induction on $k$. First, we have 
				\begin{align*}
					\|x_1 - \bar x\| &\le \|x_0 - \bar x\| + \|x_1 - x_0\|\\
					&\le \rho_0  + \eta \|\nabla f(x_0)\|\\
					&\le  \delta,
				\end{align*}
				where the last inequality follows from $\eta \le \frac{1}{L}$ and~\eqref{eqn: initialization}. So~\eqref{eqn: finitelength} holds for $k = 1$. 
				
				Now suppose that~\eqref{eqn: finitelength} holds for all $k \le k_0$. Note that for any $k \le k_0$, we automatically have $x_k \in B(\bar x, \rho_1)$ and	
				\begin{align*}
					\|x_{k_0+1} - \bar x\| &\le \|x_{k_0} - \bar x\| + \|x_{k_0+1}- x_{k_0}\|\\
					&\le  \rho_0 + \frac{2}{L} \|\nabla f(x_0)\| + \frac{4\ckl}{1-\alpha}(f(x_0) - f^*)^{1-\alpha} + \frac{1}{L}\|\nabla f(x_{k_0})\|\\
					&\le \delta,
				\end{align*}
				where the last inequality follows from~\eqref{eqn: initialization} and our choice of $\rho_1$. Therefore, $x_{k_0+1} \in B(\bar x, \delta)$. By the $L$-smoothness, for any $0\le k \le k_0$, we have
				\begin{align}\label{eqn: sufficientdescent}
					f(x_{k+1}) &\le f(x_k) - \eta \|\nabla f(x_k)\|^2 + \frac{L\eta^2}{2}\|\nabla f(x_k)\|^2 \notag \\
					&\le f(x_k) - \frac{\eta}{2} \|\nabla f(x_k)\|^2 \notag \\
					&= f(x_k) - \frac{1}{2\eta} \|x_{k+1} - x_k\|^2,
				\end{align} 
				where the  second inequality follows from $\eta \le \frac{1}{L}$. 
				Additionally, for any $0\le k\le k_0$, we have
				\begin{align}\label{eqn: relativeerror}
					\|\nabla f(x_{k+1})\| &\le L\|x_{k+1} - x_k\| + \|\nabla f(x_k)\| = \left(L + \frac{1}{\eta}\right) \|x_{k+1} - x_k\|.
				\end{align}
				Repeating the same argument as~\cite[Lemma 2.6]{attouch2013convergence}, we obtain 
				\begin{align*}
					\sum_{k=1}^{k_0} \|x_{k+1} - x_k\| &\le \|x_1 - x_0\| + \frac{2(L\eta + 1)\ckl}{1-\alpha} (f(x_1) - f^*)^{1-\alpha}\\
					&\le \|x_1 - x_0\| + \frac{4\ckl}{1-\alpha} (f(x_0) - f^*)^{1-\alpha},
				\end{align*}
				where the last inequality follows from $\eta \le \frac{1}{L}$ and $f(x_1) \le f(x_0)$. Using the triangle inequality, we obtain
				\begin{align*}
					\|x_{k_0+1} - x_0\| &\le \|x_1 -x_0\| + \sum_{k=1}^{k_0}\|x_{k+1} - x_k\|\\
					&\le   2\|x_1 - x_0\| + \frac{4\ckl}{1-\alpha} (f(x_0) - f^*)^{1-\alpha}\\
					&\le \frac{2}{L} \|\nabla f(x_0)\| + \frac{4\ckl}{1-\alpha} (f(x_0) - f^*)^{1-\alpha}.
				\end{align*}
				The result now follows by induction.
			\end{proof}
		}
		
		\subsubsection{Proof of Lemma~\ref{lem: innerlooplocal} using auxiliary lemmas}
		\begin{proof}
			First note that by applying Item~\ref{item: localaimingtowardsolution}, Item~\ref{item: localsizeofgd}, and Item~\ref{item: localgrowthcondition} of Corollary~\ref{cor: localconditionsforgdpolyak} and decreasing $\deltaU$ if necessary,  Item~\ref{item: induction1partalocal}  implies $\|\nabla f_N(x_k)\| \le \frac{1}{100} \|\nabla f_T(y_k)\|$ as claimed.
			Next, observe that  Item~\ref{item: stayinneighborhood} holds by Lemma~\ref{lem: KLlocal} and Lemma~\ref{lem: finitelength}. If $\dist(x_k, \cM) \le \frac{\dlb}{400\gammaub} \dist^{p-1}(y_k,S)$ holds for some $0\le k \le K$, we let $k_0$ be the smallest such index. Otherwise, we define $k_0 = K$. The rest of the proof consists of two steps.
			\begin{enumerate}
				\item[(i).]\label{item:astepinproof} We show that Item~\ref{item: induction1partblocal} and Item~\ref{item: induction2local} hold for any $0\le k \le k_0$. The conclusion holds trivially for $k = 0$. Suppose that the conclusion holds for all indices less or equal to $k$ for some $0 \le k \le k_0-1$. By Item~\ref{item: propertyofdecompbregularitylocal} in Lemma~\ref{lem: propertyofdecomplocal}, we have $\frac{\dlb}{\betaub} \le \frac{10}{9}$.  By Lemma~\ref{lem: GDshrinksdisttoMlocal} and decreasing $\deltaU$ if necessary, we have 
				\begin{align}\label{eqn: inductiondistshrinkifbig}
					\dist\left(x_{k+1}, \cM\right) &\le \left(1- \frac{\eta \gammalb}{2}\right) \dist(x_k, \cM) +o(1)\cdot \eta \dist^{p-1}(y_k,S) \notag\\
					&\le  \left(1- \frac{\eta \gammalb}{4}\right) \dist(x_k, \cM),
				\end{align}
				where the last inequality follows from our assumption that $\dist(x_k, \cM) \ge \frac{\dlb}{400\gammaub} \dist^{p-1}(y_k,S)$. On the other hand, by decreasing $\deltaU$ if necessary, we have 
				\begin{align}
					\dist(y_{k+1}, S) &\le \dist(y_k, S) + \gdstep\frac{\dlb^2 \gammalb}{240\betaub^2}\dist(x_k,\cM)\notag\\
					&\le \dist(y_0, S) + \sum_{l=0}^{k} \gdstep\frac{\dlb^2 \gammalb}{240\betaub^2}\dist(x_l,\cM) \notag\\
					&\le \dist(y_0, S) + \sum_{l=0}^{k} \gdstep\frac{\dlb^2 \gammalb}{240\betaub^2}\paren{1- \frac{\eta \gammalb}{4}}^l\dist(x_0,\cM)\notag \\
					&\le \dist(y_0,S) + \frac{1}{60}\frac{\dlb^2}{\betaub^2} \dist(x_0,\cM), \label{eqn:firsthalfdistance}
				\end{align}
				where the first inequality follows from Lemma~\ref{lem: ychangeingdlocal}, the second inequality follows by applying the first inequality recursively, the third inequality follows from the induction hypothesis, and the last inequality follows by bounding the geometric series.
				\item[(ii).]  We show that Item~\ref{item: induction1partalocal}  and Item~\ref{item: induction2local} hold for any $k_0\le k \le K$.  By part (i) of this proof and the definition of $k_0$, the conclusion holds trivially for $k = k_0$. Suppose that the conclusion holds for all indices less or equal to $k$ for some $k_0 \le k \le K-1$. By applying equation~\eqref{eqn:lower_bound_lem} in Lemma~\ref{lem: ychangeingdlocal} to $y= y_k$, our induction hypothesis that $\dist(x_k, \cM) \le  \frac{\dlb}{200\gammaub} \dist^{p-1}(y_k,S)$, and decreasing $\deltaU$ if necessary, we have
				\begin{align}\label{eqn: lbdist(y,S)}
					\dist(y_{k+1}, S) \ge \max\left \{\paren{1 - \frac{\eta \gammalb}{4} }^{1/(p-1)} , \left(\frac{2}{3}\right)^{1/(p-1)}\right\}\dist(y_k,S).
				\end{align}
				(The motivation for the precise terms in this maximum will be clear momentarily.) We now consider two cases:
				\begin{enumerate}
					\item Suppose $\dist(x_k, \cM) \ge \frac{\dlb}{400\gammaub} \dist^{p-1}(y_k,S)$. Then we have
					\begin{align*}
						\dist(x_{k+1},\cM) &\le \paren{1 - \frac{\eta \gammalb}{4} }\dist(x_k, \cM)\\
						&\le  \paren{1 - \frac{\eta \gammalb}{4} }\frac{\dlb}{200\gammaub} \dist^{p-1}(y_k,S)\\
						&\le \frac{\dlb}{200\gammaub} \dist^{p-1}(y_{k+1},S),
					\end{align*}
					where the first inequality follows by Lemma~\ref{lem: GDshrinksdisttoMlocal} and decreasing $\deltaU$ if necessary, the second inequality follows from the induction hypothesis that $\dist(x_i, \cM) \le  \frac{\dlb}{200\gammaub} \dist^{p-1}(y_i,S)$, and the last inequality follows from~\eqref{eqn: lbdist(y,S)}. 
					\item  Suppose that $\dist(x_k, \cM) \le  \frac{\dlb}{400\gammaub} \dist^{p-1}(y_k,S)$. Then we have
					\begin{align*}
						\dist(x_{k+1},\cM) &\le \dist(x_k, \cM) + 	o(1)\dist^{p-1}(y_k,S)\\
						&\le \frac{\dlb}{300\gammaub} \dist^{p-1}(y_k,S)\\
						&\le \frac{\dlb}{200\gammaub} \dist^{p-1}(y_{k+1},S),
					\end{align*}
					where the first inequality follows from  Lemma~\ref{lem: GDshrinksdisttoMlocal}, the second inequality follows by decreasing $\deltaU$ if necessary, and the last inequality follows from equation~\ref{eqn: lbdist(y,S)} and the estimate
					$$\dist^{p-1}(y_{k+1},S)  \ge \frac{2}{3}\dist^{p-1}(y_k,S).$$ 
				\end{enumerate}
				Consequently, Item~\ref{item: induction1partalocal} holds for index $k+1$. To see Item~\ref{item: induction2local}, we note that when $\deltaU$ is sufficiently small, by Lemma~\ref{lem: ychangeingdlocal}, for any $k_0\le l\le k$, we have
				$$
				\dist(y_{l+1}, S) \le \dist(y_l, S)\paren{1- \frac{\eta \dlb}{4} \dist(y_l,S)^{p-2}} + \gdstep \gammaub \dist(x_l, \cM).
				$$
				Now recall that we have $\dist(x_l, \cM) \le \frac{\dlb}{200\gammaub} \dist^{p-1}(y_l,S)$ for any $k_0\le l \le k$ by our induction hypothesis. As a result, 
				$$
				\dist(y_{l+1}, S) \le \dist(y_{l}, S)\paren{1- \frac{\eta \dlb}{4} \dist(y_l,S)^{p-2}} + \frac{\eta \dlb}{200} \dist^{p-1}(y_l,S) \le \dist(y_l,S).
				$$
				Therefore, we have $\dist(y_{k+1},S) \le \dist(y_{k_0}, S) \le \dist(y_0,S) + \frac{1}{60}\frac{\dlb^2}{\betaub^2} \dist(x_0,\cM),$ where the final inequality follows from~\eqref{eqn:firsthalfdistance}. Thus, Item~\ref{item: induction2local} holds.
			\end{enumerate}	
			The proof is complete.
			
		\end{proof}
		
		\section{Examples}
		In this section, we show that $\gdpk$ (Algorithm~\ref{alg:GD-Polyak}) has a local nearly linear rate for the two main examples discussed in the introduction: overparameterized matrix sensing and for learning a single neuron. More precisely, we will show that Assumption~\ref{assum:tangent} holds, and therefore, our main result (Theorem~\ref{thm: main_thm_intro}) is applicable.

		\subsection{Overparametrized matrix factorization}\label{section: matrix factorization}
		We begin with a simplified problem of overparameterized matrix factorization to build intuition for the more complicated matrix sensing problem. Namely, overparameterized matrix factorization is the optimization problem:
		\begin{align}\label{eqn: matrixfactorization}
			\min_{B\in \RR^{d\times k}}~f(B) = \|BB^\top -X \|_F^2,
		\end{align}
		where $X\in \R^{d\times d}$ is a symmetrix positive definite rank $r$ matrix, for some $r<k$. In particular, the optimal value of the problem \eqref{eqn: matrixfactorization} is zero. The main difficulty of the optimization problem \eqref{eqn: matrixfactorization} is that it exhibits a mixture of quadratic and quartic growth. To see this, without loss of generality, we may assume that $X$ takes the form  
		$X=
		\begin{pmatrix}
			D & 0 \\
			0 & 0
		\end{pmatrix},
		$
		where $D \in \RR^{r\times r}$ is a diagonal matrix with positive diagonal elements. We let $\sigma_1$ and $\sigma_r$ be the largest and smallest eigenvalues of $D$. We write the variable $B$ in block form $B = \begin{pmatrix}
			P \\ Q
		\end{pmatrix}$, for  $P \in \RR^{r\times k}$ and $Q \in \RR^{(d-r)\times k}$. With this notation, the objective function takes the form
		\begin{align*}
			f(B) = \norm{\begin{pmatrix}
					PP^\top - D & PQ^\top \\
					QP^\top & QQ^\top
			\end{pmatrix}}_F^2 = \|	PP^\top - D\|_F^2 + 2 \|PQ^\top\|_F^2 + \|QQ^\top\|_F^2.
		\end{align*}
		Clearly, the set of minimizers has the form
		$$S := \left\{\begin{pmatrix}
			P\\
			0
		\end{pmatrix} \colon PP^\top = D \right\}.$$ 
		Define the following set 
		$$\cM := \left\{B = \begin{pmatrix}
			P\\
			Q
		\end{pmatrix} \colon PP^\top = D, PQ^\top=0 \right\}.$$
		Clearly, $\cM$ contains $S$ and is, in fact, a smooth manifold. The proofs of all results in this section appear in Section~\ref{sec:proof_overparam_sensing}.
		
		\begin{theorem}[Smoothness]\label{thm:ravineissmooth}
			The set $\cM$ is a $C^{\infty}$ smooth manifold.
		\end{theorem}

		Next, the following theorem shows that $\cM$ is a ravine with respect to the nearest point projection $P_{\cM}$. 
		
		\begin{theorem}[Ravine]\label{thm: factorization_uniform_ravine}
			There exists a constant $\delta >0$ such that the estimate
			\begin{align*}
				\frac{\sigma_r}{8}\|B -P_\cM(B)\|_F^2 \le f(B) - f(P_\cM(B)) \le  18 \sigma_1\|B -P_\cM(B)\|_F^2,
			\end{align*}
			holds for any $B$ with $\dist_F(B,S) < \delta$.
			In particular, $\cM$ is a $C^{\infty}$ ravine for $f$ at any $\bar B\in S$.
		\end{theorem}

		Next, we verify the constant-order quartic growth of $f$ on $\cM$.
		
		\begin{lemma}[Constant order growth on $\cM$]\label{lem: matrixfac_fourthorder}
			For any $B \in \cM$, we have 
			\begin{equation}\label{eqn:const_growth_man}
				\frac{1}{k} \dist^4(B, S) \le f(B) \le  \dist^4(B, S),
			\end{equation}
			and consequently Assumption~\ref{assum:tangent} holds with $p=4$. 
		\end{lemma}

		Thus, Corollary~\ref{cor: compactravineconvergence} directly applies and shows that Algorithm~\ref{alg:GD-Polyak} converges almost linearly when initialized sufficiently close to the solution set $S$.

		\subsection{Overparametrized matrix sensing}
		As discussed in the introduction, the symmetric matrix sensing problem is given by 
		\begin{equation}\label{eqn: matrixsensingobj}
			\min_{B\in \R^{d\times k}}f(B) = \frac{1}{m}\sum_{i=1}^{m} (y_i - \dotp{A_i, BB^\top})^2,
		\end{equation}
		where $A_1,\ldots, A_m \in \RR^{d\times d}$ are fixed (measurement) matrices and equalities $y_i = \dotp{A_i, X}$ hold for some symmetric positive semi-definite matrix $X\in \RR^{d\times d}$ with rank $r$. We assume that the problem is rank overparameterized, that is $k \ge r$.
		Our main result will hold under the standard restricted isometry property. Namely, define the sensing linear map by 
		$$
		\cA(X) := [m^{-1/2}\dotp{A_i, X} ]_{1\le i \le m}.
		$$
		
		\begin{definition}[Restricted Isometry Property(\cite{candes2008restricted})]\label{def: RIP}{\rm
				The  map $\cA: \RR^{n\times n} \rightarrow \RR^{m}$ satisfies the {\em Restricted Isometry Property (RIP) of rank $l$ with constant $\delta >0$} if the estimate	$$
				(1-\delta) \|Z\|_F^2 \le \|\cA(Z)\|^2 \le (1+\delta)\|Z\|_F^2.
				$$
				holds for all matrices $Z$ with rank at most $l$.}
		\end{definition}
		
		For various random measurement models (e.g. Gaussian), RIP holds with high probability~\cite{candes2011tight, recht2010guaranteed} and with $\delta$ arbitrarily small.   The following theorem is the main result of the section. The proof appears in Section~\ref{sec:proof_mat-sense}.
		\begin{theorem}[Ravine for matrix sensing]\label{thm: sensinggrowth}
			Suppose that for measurement operator $\cA$ satisfies RIP of rank $l = k+r$ and with some $\delta  \le \frac{1}{2}$. Then the set of minimizer of $f$ in~\eqref{eqn: matrixsensingobj} is given by $S = \{B\colon BB^\top =X\}$. Moreover, $\nabla^2 f$ has constant rank on $S$ and there exist constants $\delta_0 >0$ and $\dlb >0$ such that for any $B$ with $\dist(B,S) \le \delta_0$, we have
			\begin{equation}\label{eqn:quart_lower_mat_sense}
				f(B) \ge \dlb \dist^4 (B,S).
			\end{equation}
			Consequently, Assumption~\ref{assum:tangent} holds for the $C^{\infty}$ Morse ravine at any  $\bar B\in S$ with $p=4$.
		\end{theorem}

		Thus Corollary~\ref{cor: compactravineconvergence} directly applies and shows that Algorithm~\ref{alg:GD-Polyak} converges almost linearly when initialized sufficiently close to the solution set $S$.
		
		\subsection{Overparametrized neural network}
		Our final example is the problem of a single neuron in the overpartametrized regime. That is, following~\cite{xu2023over}, we consider the problem
		$$
		f(w) = \EE_{x\sim N(0,I)}\left[ \frac{1}{2}\left(\sum_{i=1}^{2}[w_i^\top x]_+ - [v^\top x]_+ \right)^2\right],
		$$
		where $w = (w_1^\top, w_2^\top)^\top \in \RR^{2d}$ denotes the parameter vector. The paper~\cite{xu2023over} showed that gradient descent with constant stepsize converges at a sublinear rate. In contrast, we will now show that  one can achieve a local nearly linear rate of convergence by using the adaptive gradient method (Algorithm~\ref{alg:GD-Polyak}). Define the set
		\begin{align}\label{eq:solutionsetdu}
			S = \left\{w \colon w_1+w_2 = v, \dotp{w_i, v} = \|w_i\|\|v\| \text{, and } \frac{\|v\|}{8} \le \|w_i\| \le  2\|v\| \text{ for $i=1,2$} \right\}.
		\end{align}
		It is known that $S$ is a strict subset of minimizers of $f$. We work with this subset rather than the entire set of minimizers because the three-phase analysis in~\cite{xu2023over} shows that gradient descent from small random initialization converges to the vicinity of $S$ and the local geometry at $S$ slows down gradient decent.  	Note that \cite[Lemma 1]{safran2018spurious} shows that the function $f$ is $C^2$-smooth on a neighborhood of $S$. As we will see, the linear subspace $\cM = \{w \colon w_1 + w_2 = v\}$ comprises a ravine around any minimizer $w\in \cM\cap S$. The following is the main result of the section. The proof appears in Section~\ref{sec:proofs_overparamNNs}.

		\begin{theorem}[Ravine for learning a single layer neural network]\label{thm: simonduravine}
			For any $\bar w \in S$, the set $\cM$ is a $C^2$ ravine for $f$ at $\bar w$. Moreover, the projection map $P_\cM$ is $C^\infty$ smooth near $\bar w$, and Assumption~\ref{assum:tangent} holds for $\cM$ at $\bar w$ with $p=3$.
		\end{theorem}
		
		Thus Corollary~\ref{cor: compactravineconvergence} directly applies and shows that Algorithm~\ref{alg:GD-Polyak} converges almost linearly when initialized sufficiently close to any point in $S$.

		\bibliographystyle{unsrt}
		\bibliography{bibliography}
		
		\appendix

		\section{Proofs for overparametrized matrix factorization}\label{sec:proof_overparam_sensing}

		\subsection{Proof of Theorem~\ref{thm:ravineissmooth}}
		The proof proceeds by an application of the implicit function theorem. To this end, we may write $\cM$ as the zero set of the map $F: \RR^{d\times k} \rightarrow \mathbb{S}^{r\times r} \times \RR^{r\times (d-r)}$ defined by setting $F(P,Q)
		= \begin{pmatrix}
			PP^\top - D
			\\ PQ^\top
		\end{pmatrix}$.  We claim that $\nabla F$ is surjective at any point $\begin{pmatrix}
			P\\
			Q
		\end{pmatrix} \in \cM$. To see this, a simple calculation yields the expression	$$\nabla F{(P,Q)}\left[\begin{pmatrix}
			R\\
			S
		\end{pmatrix}\right] = \begin{pmatrix}
			PR^\top + RP^\top \\ PS^\top + RQ^\top
		\end{pmatrix}. $$
		Note that the first block on the right side depends only on $R$ and not on $S$.	
		A quick computation shows that for any symmetric matrix $A \in \mathbb{S}^{r\times r}$, equality
		$A=PR^\top + RP^\top$ holds for the matrix $R=\frac{1}{2} AD^{-1}P$. Thus the map $R \rightarrow PR^\top + RP^\top$ is surjective from $\RR^{r\times k}$ to $\mathbb{S}^{r\times r}$. Looking at the second block, the equality $PP^\top= D$ implies that $P$ has full row-rank and therefore $S \rightarrow PS^\top$ is surjective from $\RR^{(d-r)\times k}$ to $\RR^{(d-r)\times r}$. 
		Thus we conclude that the Jacobian $\nabla F{\begin{pmatrix}
				P\\
				Q
		\end{pmatrix}}$ is surjective. Note that $F$ is $C^\infty$ and $\nabla F$ is surjective at every point in $F^{-1}(0) = \cM$, and therefore $\cM$ is a $C^\infty$-smooth manifold.		
		
		\subsection{Proof of Theorem~\ref{thm: factorization_uniform_ravine}}\label{sec:proof_of tangent}
		Denote the projection of $B$ onto $\cM$ by 
		$P_\cM(B)=\begin{pmatrix}
			\hat P \\ \hat Q
		\end{pmatrix}$. 
		We begin by writing
		\begin{align}
			f(B) - f(P_\cM(B)) &=f(P,Q)-f(\hat P,\hat Q)\notag\\
			&= \|PP^\top -D\|_F^2 + 2 \|PQ^\top\|_F^2 + \|QQ^\top\|_F^2\notag \\
			&~~-(\|\hat P\hat P^\top -D\|_F^2 + 2 \|\hat P \hat Q^\top\|_F^2 + \|\hat Q \hat Q^\top\|_F^2)\notag\\
			&=\|PP^\top -D\|_F^2 + 2 \|PQ^\top\|_F^2 + \|QQ^\top\|_F^2-\|\hat Q \hat Q^\top\|_F^2,\label{eqn:multline}
		\end{align}
		where  \eqref{eqn:multline} follows from the definition of $\cM$. The remainder of the proof estimates each term on the right side of \eqref{eqn:multline}.

		We denote the nearest-point projection of $P$ onto $\{P \colon PP^\top = D\}$ by $\tilde P$ and the projection of $Q$ onto $\{Q \colon \tilde{P}Q^\top = 0\}$ by $\tilde Q$.	 The result~\cite[Lemma 35]{ma2018implicit} shows that the matrix $\tilde P^\top P$ is symmetric and positive semidefinite. By taking $\delta$ sufficiently small, we may assume that $\sigma_r(P) \ge \sqrt{\frac{\sigma_r(D)}{2}} =  \sqrt{\frac{\sigma_r}{2}}$, and $\max\{\sigma_1(P), \sigma_1(\hat P), \sigma_1(Q)\} \le 2\sqrt{\sigma_1}.$ Consequently, we deduce
		\begin{align}\label{eqn: PtotildeP}
			\|PP^\top - D\|_F^2 &= \|P(P - \tilde P)^\top + (P-\tilde P) \tilde P^\top\|_F^2 \notag\\
			&= \fnorm{P(P - \tilde P)^\top}^2 + \fnorm{(P-\tilde P) \tilde P^\top}^2 + 2\cdot \trace{P(P - \tilde P)^\top \tilde P (P-\tilde P)^\top} \notag\\
			&\ge 2\fnorm{P(P - \tilde P)^\top}^2 + 2\cdot \trace{P(P - \tilde P)^\top  P (P-\tilde P)^\top}+ o(1)\|P-\tilde P\|_F^2 \notag\\
			&\ge 2\fnorm{P(P - \tilde P)^\top}^2 + o(1)\|P-\tilde P\|_F^2 \notag\\
			&\ge \frac{\sigma_r}{2} \fnorm{P-\tilde P}^2,
		\end{align}
		where the second inequality follows from the observation that $\tilde P^\top P$ is symmetric.
		On the other hand, we also have:
		\begin{align}\label{eqn: PP-D_upper_bound}
			\fnorm{PP^\top -D} &= \|P(P - \hat P)^\top + (P-\hat P) \hat P^\top\|_F \le  4 \sqrt{\sigma_1} \|P - \hat P\|_F.
		\end{align}
		Next, note that we may write
		\begin{align}
			\|PQ^\top \|_F^2 &= \|(P- \tilde P)Q^\top + \tilde{P}(Q-\tilde Q)^\top\|_F^2  \label{eqn: STlowerboundeq1}\\
			&\ge  \sigma_r^2(\tilde P)\|Q-\tilde Q\|_F^2 - 2\opnorm{\tilde P} \opnorm{ Q} \|P-\tilde P\|_F\|Q-\tilde Q\|_F \label{eqn: STlowerboundeq2} \\
			&= \sigma_r \|Q - \tilde Q\|_F^2 + o(1) \|P-\tilde P\|_F\|Q-\tilde Q\|_F\label{eqn: STlowerboundeq3}\\
			&=\sigma_r \|Q - \tilde Q\|_F^2 + o(1) (\|P-\tilde P\|_F^2+\|Q-\tilde Q\|_F^2) \label{eqn: STlowerbound}
		\end{align}
		where the equation~\ref{eqn: STlowerboundeq1} follows from the defining equation $\tilde P \tilde Q^\top =0$, the estimate~\ref{eqn: STlowerboundeq2} follows from expanding the square, the equation~\ref{eqn: STlowerboundeq3} follows from the estimate $Q=o(1)$, and  the equation~\ref{eqn: STlowerbound} follows from Young's inequality.
		Note that $\hat Q$ is the projection of $Q$ onto the subspace $\{Q \colon \hat P  Q^\top = 0\}$. In particular, we project each row of $Q$ onto $\ker(\hat P)$  to obtain $\hat Q$, which implies   $\|\hat Q\|_F\le \|Q\|_F$ and $(Q-\hat Q){\hat Q}^\top = 0$. Therefore, we deduce
		\begin{align}
			\abs{\|QQ^\top\|_F^2 - \|\hat Q {\hat Q}^\top\|_F^2} &= \abs{\|\hat Q \hat Q^\top + (Q-\hat Q){\hat Q}^\top + Q(Q-\hat Q )^\top\|_F^2 - \|\hat Q{\hat Q}^\top\|_F^2} \notag \\
			&=\abs{2\langle \hat Q \hat Q^\top,(Q-\hat Q){\hat Q}^\top + Q(Q-\hat Q )^\top\rangle+\|(Q-\hat Q){\hat Q}^\top + Q(Q-\hat Q )^\top\|^2_F}\notag\\
			&\le \|Q(\hat Q - Q)^\top\|_F^2  + 2\abs{\dotp{\hat Q{\hat Q}^\top,   Q(Q-\hat Q)^\top}} \label{eqn:jeb1}\\
			&= o(1)\|Q-\hat Q\|_F^2  +  2\abs{\dotp{\hat Q{\hat Q}^\top,(Q-\hat Q)(Q - \hat Q)^\top}} \label{eqn:jeb2}\\
			&= o(1)\|Q-\hat Q\|_F^2,\label{eqn: boundonTT}
		\end{align}
		where \eqref{eqn:jeb1} and \eqref{eqn:jeb2} follow from the equation $(Q-\hat Q){\hat Q}^\top = 0$. 
		Therefore, by taking $\delta$ sufficiently small, returning to \eqref{eqn:multline} we 
		compute
		\begin{align*}
			f(B) - f(P_\cM(B))	&=\|PP^\top -D\|_F^2 + 2 \|PQ^\top\|_F^2 + \|QQ^\top\|_F^2-\|\hat Q \hat Q^\top\|_F^2\\
			&\ge \frac{\sigma_r}{4} (\| P -\tilde P\|_F^2 + \|Q-\tilde  Q\|_F^2) + o(1)\|Q-\hat Q\|_F^2 \\
			&\ge \frac{\sigma_r}{8}\|B -P_\cM(B)\|_F^2,
		\end{align*}
		where the first inequality follows from equations~\eqref{eqn: PtotildeP}, \eqref{eqn: STlowerbound} and~\eqref{eqn: boundonTT}, and the second inequality follows from the fact that $(\tilde P,\tilde Q)$ lies in $\cM$ and we have
		$\fnorm{Q -\hat Q}\le \|B- P_\cM(B)\|_F \le \fnorm{B - \begin{pmatrix}
				\tilde P\\
				\tilde Q
		\end{pmatrix}}$. To see the reverse inequality, we compute
		\begin{align*}
			f(B) - f(P_\cM(B)) &= \|PP^\top -D\|_F^2 + 2 \|PQ^\top\|_F^2 + \|QQ^\top\|_F^2-\|\hat Q \hat Q^\top\|_F^2\\
			&\le 16 \sigma_1 \|P - \hat P\|_F^2 +2\|(P- \hat P)Q^\top + \hat P(Q-\hat Q)^\top\|_F^2+ o(1)\|Q -\hat Q\|_F^2\\
			&\le 18\sigma_1 \|P-\hat P\|_F^2 + 2\sigma_1 \|Q-\hat Q\|_F^2 + o(1)\|Q -\hat Q\|_F^2\\
			&\le 18\sigma_1 \|B - P_\cM(B)\|_F^2, 
		\end{align*}
		where the first inequality follows from equations~\eqref{eqn: PP-D_upper_bound} and \eqref{eqn: boundonTT}.
		The proof is complete.
		
		\subsection{Proof of Lemma~\ref{lem: matrixfac_fourthorder}}
		\begin{proof}
			Firstly, we note that for any $B \in \cM$, we have $\dist(B,S) = \|Q\|_F$. Let $\sigma_i(Q)$ be the $i$-th largest singular values of $Q$. We compute 
			\begin{align*}
				f(B)=\|QQ^\top\|_F^2 &= \sum_{i=1}^{k} \sigma_i^4(Q)\ge \frac{1}{k} \paren{\sum_{i=1}^{k}  \sigma_i^2(Q)}^2= \frac{1}{k}\|Q\|_F^4,
			\end{align*}
			Therefore, for any $B \in \cM$, we have
			\begin{align*}
				f(B) = \|QQ^\top\|_F^2 \ge \frac{1}{k} \|Q\|_F^4 = \frac{1}{k}\dist^4(B,\cS). 
			\end{align*} 
			Conversely, we have
			\begin{align*}
				f(B) = \|QQ^\top\|_F^2 \le \fnorm{Q}^4 = \dist^4(B,S),
			\end{align*} 
			which completes the proof of \eqref{eqn:const_growth_man}. Appealing to Theorem~\ref{thm:simplified_assumpt}, we conclude that Assumption~\ref{assum:tangent} indeed holds.
		\end{proof}
		
		\subsection{Explicit estimates in Assumption~\ref{assum:tangent} and constant rank}
		In this section, we verify directly items 2 and 3 of Assumption~\ref{assum:tangent} with explicit constants and show that the Hessian $\nabla^2 f$ has constant rank on $S$. These results are not needed for the rest of the arguments, and we include them here for completeness.

		\begin{lemma}[Aiming towards solution and the gradient bound]\label{lem: matrixfacbregularity}
			For any $B \in \cM$, we have 
			\begin{align}
				\dotp{\nabla f(B), B-P_S(B)} &\ge f(B),\label{eqn:firs_aim}\\
				4 \cdot \dist^3(B,S)&\geq \|\nabla f(B)\|_F.\label{eqn:grad_bound}
			\end{align}
		\end{lemma}
		\begin{proof}
			First, elementary calculations show that for any $B$, we have
			\begin{align*}
				\nabla f(B) = 4 \cdot \begin{pmatrix}
					(PP^\top -D)P + PQ^\top Q \\
					Q P^\top P + QQ^\top Q 
				\end{pmatrix}.
			\end{align*}
			In particular, for any $B \in \cM$, we have $
			\nabla f(B) =  \begin{pmatrix}
				0 \\
				4QQ^\top Q 
			\end{pmatrix}$. On the other hand, for $B \in \cM$, we know that $B - P_S(B) = (0,Q)$. Consequently, we deduce
			\begin{align*}
				\dotp{\nabla f(B), B- P_\cS(B)} &= 4 \dotp{QQ^\top Q, Q}= 4 \|QQ^\top\|_F^2=4 f(B),
			\end{align*}
			which completes the proof of \eqref{eqn:firs_aim}. Finally, we compute
			$$
			\|\nabla f(B)\|_F = 4\|QQ^\top Q\|_F \le 4\|Q\|_F^3 = 4\cdot \dist^3(B, S). 
			$$
			thereby verifying \eqref{eqn:grad_bound}.
		\end{proof}

		We next show that $\nabla^2 f$ has constant rank on $S$.

		\begin{theorem}[Constant rank]\label{thm:factorization_fourth_growth}
			Fix a $C^2$-smooth function $h\colon\R^{d\times d}\to \R$ and define the function 
			$g\colon\R^{d\times k}\to \R$ by setting $g(B)=h(BB^\top)$. Then the Hessian $\nabla^2 g$ has constant rank on any set of the form $S=\{B:BB^\top=X\}$.
		\end{theorem}
		\begin{proof}
			Note that the group of orthogonal $k\times k$ matrices acts transitively on $S$ by right multiplication, and moreover, $g$ is invariant under this group action. Define for any $k\times k$ orthogonal matrix $V$ the map $\varphi_V(B)=BV$. Clearly, $\varphi_V$ is a linear isomorphism, and equality holds:
			$$f=f\circ\varphi_V.$$
			Consequently, differentiating twice yields the expression
			$\nabla^2 f(B)= \varphi_V^*\nabla^2 f(BV)\varphi_V$, and therefore $\nabla^2 f$ has constant rank on $S$ as claimed.
		\end{proof}

		\section{Proofs for overparametrized matrix sensing}\label{sec:proof_mat-sense}
		\subsection{Proof of Theorem~\ref{thm: sensinggrowth}}
		\begin{proof}
			By Definition~\ref{def: RIP} and the fact that $\rank(BB^\top - X) \le r+k$, we have
			$$
			\frac{1}{2}\|BB^\top - X\|_F^2 \le f(B) \le \frac{1}{2} \|BB^\top - X\|_F^2.
			$$
			Consequently, the set $S=\{B: BB^\top=X\}$ coincides with the set of minimizers of $f$. Moreover, the quartic growth \eqref{eqn:quart_lower_mat_sense} follows immediately from the lower bound	$f(B) \ge \frac{1}{2}\|BB^\top - X\|_F^2 $ and Lemma~\ref{lem: matrixfac_fourthorder}. The fact that $\nabla^2 f$ has constant rank on $S$ follows from Theorem~\ref{thm:factorization_fourth_growth}. Finally, applying Corollary~\ref{cor:morseravineexistsfourth} we deduce that Assumption~\ref{assum:tangent} holds for the $C^{\infty}$ smooth ravine around any point in $S$. 
		\end{proof}

		\section{Proofs for overparametrized neural network}\label{sec:proofs_overparamNNs}
		We begin with some notation. Let $\theta_{12}$ denote the angle between $w_1$ and $w_2$, and $\theta_i$ denotes the angle between $w_i$ and $v$ for $i=1,2$. For each $w_i$, we decompose it into $w_i = w_i^\parallel + w_i^\perp$, where $w_i^\parallel= \proj_{\spann\{v\}}(w_i)$ and $w_i^\perp = \proj_{\spann\{v\}^\perp}(w_i)$. Moreover, we denote the normal and tangent space of $\cM$ by $N_\cM$ and $T_\cM$, which do not depend on the base point since $\cM$ is affine. It has been shown in~\cite{safran2018spurious} that the closed form expression for $f$ is 
		\begin{align}\label{eqn: dufunction}
			f(w) = \frac{1}{4}\|w_1 + w_2 - v\|^2 + \frac{1}{2\pi}\left[ (\sin \theta_{12} - \theta_{12}\cos \theta_{12}) \|w_1\|\|w_2\| - \sum_{i=1}^{2} (\sin \theta_i - \theta_i \cos \theta_i) \|w_i\| \|v\|\right],
		\end{align}
		Also, the closed-form expression for the gradient $\nabla f$ is the following:
		\begin{align}\label{eqn: gradientformula}
			\nabla_{w_i} f(w) = \frac{1}{2} \left(\sum_{i=1}^{n} w_i - v\right) + \frac{1}{2\pi}\left[ \left( \sum_{j\neq i}\|w_j\|\sin \theta_{ij} -\|v\|\sin \theta_i \right) \bar w_i  - \sum_{j\neq i} \theta_{ij} w_{j} +\theta_i v\right],
		\end{align}
		where $\bar w_i = \frac{w_i}{\|w_i\|}$. We start with several technical lemmas that serve as stepping stones for proving the main theorem.
		\begin{lemma}[{\cite[Lemma 18]{xu2023over}}]\label{lem: anglelb}  For any $w$, we have
			$$
			\|w_i\|^2 \theta_i^3 \le 30\pi f(w),  \quad \text{for $i=1,2$.}
			$$
		\end{lemma}
		\begin{lemma}\label{lem: thetaandorthogonalpart}
			For $w$ sufficiently close to $S$, we have $\theta_i = \Theta(\|w_i^\perp\|)$.
		\end{lemma}
		\begin{proof}
			For $w$ sufficiently close to $S$, the angle $\theta_i$ is sufficiently small and we have 
			$$\theta_i = \Theta(\tan \theta_i) = \Theta\paren{\frac{\|w_i^\perp\|}{\|w_i^\parallel\|}} = \Theta(\|w_i^\perp\|),$$
			which completes the proof. 
		\end{proof}
		
		\begin{lemma}[{\cite[Lemma 19]{xu2023over}}]\label{lem: llbbyr}
			There exists $\delta>0$ and $C>0$ such that for all $w$ with $\frac{\|v\|}{16} \le \|w_i\| \le 4\|v\|$ for $i=1,2$ and $L(w) \le \delta$, we have 
			$$
			\left\|w_1 + w_2 - v \right\| \le C f^{\frac{1}{2}}(w).
			$$
		\end{lemma}
		\begin{proof}
			The proof is the same as the proof~\cite[Lemma 19]{xu2023over}, except that we have a slightly more relaxed range of $w_i$. 
		\end{proof}

		\begin{lemma}[Gradients on manifold]\label{lem: gradientonmanifold}
			For any $w \in \cM$ sufficiently close to $S$, we have 
			\begin{align}\label{eqn: nngradonmanifold}
				\nabla_{w_1} f(w) = \frac{1}{2\pi}\paren{\theta_1 w_1 - \theta_2 w_2}, \qquad \nabla_{w_2}f(w) = \frac{1}{2\pi}\paren{\theta_2 w_2 - \theta_1 w_1}.
			\end{align}
			In particular, we have $\nabla f(w) \in  T_\cM$.
		\end{lemma}
		\begin{proof}
			
			By the gradient formula, we have 
			$$
			\nabla_{w_1}f(w) = \frac{1}{2} \left(\sum_{i=1}^{2} w_i - v\right) + \frac{1}{2\pi}\left[ \left( \|w_2\|\sin \theta_{12} -\|v\|\sin \theta_1 \right) \bar w_1  -  \theta_{12} w_{2} +\theta_1 v\right],
			$$
			and
			$$
			\nabla_{w_2} f(w) = \frac{1}{2} \left(\sum_{i=1}^{2} w_i - v\right) + \frac{1}{2\pi}\left[ \left( \|w_1\|\sin \theta_{12} -\|v\|\sin \theta_2 \right) \bar w_2  -  \theta_{12} w_{1} +\theta_2 v\right].
			$$
			For any $w \in \cM$ sufficiently close to $S$, we have $w_1 + w_2 = v$, so we have $\theta_{12} = \theta_1 + \theta_2$,
			\begin{align}\label{eqn: propertyonmanifoldnormal}
				\|w_1\|\sin \theta_1 = \|w_2\| \sin \theta_2
			\end{align}
			and
			\begin{align}\label{eqn: propertyonmanifoldtangent}
				\|w_1 \|\cos \theta_1 + \|w_2\|\cos \theta_2 = v.
			\end{align}
			We successively compute
			\begin{align*}
				\|v\| \sin \theta_1 \bar w_1 &= (\|w_1 \|\cos \theta_1 + \|w_2\|\cos \theta_2)\sin \theta_1 \frac{w_1}{\|w_1\|}\\
				&= \cos \theta_1 \sin \theta_1 w_1 + \frac{\|w_2\|}{\|w_1\|}\cos \theta_2 \sin \theta_1 w_1\\
				&= \frac{\|w_2\|}{\|w_1\|}\cos \theta_1 \sin \theta_2 w_1 + \frac{\|w_2\|}{\|w_1\|}\cos \theta_2 \sin \theta_1 w_1\\
				&= \|w_2\| \sin (\theta_1 + \theta_2) \bar w_1,
			\end{align*}
			where the first equality follows from~\eqref{eqn: propertyonmanifoldtangent} and the third equality follows from~\eqref{eqn: propertyonmanifoldnormal}. As a result, we obtain 
			$$
			\nabla_{w_1}f(w) = \frac{1}{2\pi}\paren{\theta_1 w_1 - \theta_2 w_2}.
			$$
			By the same argument, we can show 
			$$
			\nabla_{w_2}f(w) = \frac{1}{2\pi}\paren{\theta_2 w_2 - \theta_1 w_1}.
			$$
			Combining, we have
			$$
			\nabla_{w_1}f(w) + \nabla_{w_2}f(w)  = 0,
			$$
			which implies the inclusion $\nabla f(w) \in T_\cM$.
		\end{proof}

		\begin{lemma}\label{lem: nnthirdordergrowth}
			For any $w \in \cM$ sufficiently close to $S$, we have 
			$$f(w) = \Theta(\|w - P_S(w)\|^3).$$
		\end{lemma}
		\begin{proof}
			For  $ w \in \cM$ sufficiently close to $S$, we have $w_1 + w_2 = v$ and $\theta_{12} = \theta_1 + \theta_2$. Note that 
			\begin{align*}
				f(w) &= \frac{1}{2\pi}\left[ (\sin \theta_{12} - \theta_{12}\cos \theta_{12}) \|w_1\|\|w_2\| - \sum_{i=1}^{2} (\sin \theta_i - \theta_i \cos \theta_i) \|w_i\| \|v\|\right]\\
				&= O(\theta_{12}^3 + \theta_1^3 + \theta_2^3)\\
				&= O(\theta_1^3 + \theta_2^3)\\
				&= O(\|w_1^\perp\|^3 + \|w_2^\perp\|^3)\\
				&= O(\dist^3(w,S)),
			\end{align*}
			where the second equality follows from the bound on the size of $\|w_1\|$ and $\|w_2\|$ implied by the definition of $S$ and  $\sin \theta -\theta \cos \theta = \Theta(\theta^3)$ as $\theta \rightarrow 0$, the fourth equality follows from Lemma~\ref{lem: thetaandorthogonalpart}, and the last equality follows from the fact that $w - P_S(w) = \begin{pmatrix}
				w_1^\perp\\
				w_2^\perp
			\end{pmatrix}.$
			On the other hand, by Lemma~\ref{lem: anglelb}, for $w \in \cM$ sufficiently close to $S$, we have $\theta_1^3 + \theta_2^3 = O(f(w))$. By a similar calculation as above, we obtain
			$$
			\dist^3 (w,S) = O(\theta_1^3 + \theta_2^3) = O(f(w)).
			$$
			Therefore, for $w \in \cM$ sufficiently close to $S$, we have 
			$$
			f(w) = \Theta(\dist^3(w,S)),
			$$
			which completes the proof.
		\end{proof}
		
		\begin{lemma}\label{lem: nnbregularity}
			For any $w\in U \cap \cM$ sufficiently close to $S$, we have 
			$$
			\dotp{\nabla f(w), w - P_S(w)} \ge f(w).
			$$
		\end{lemma}
		\begin{proof}
			First, we note that for $w \in U\cap \cM$,
			$$
			w - P_S(w) = \begin{pmatrix}
				w_1^\perp\\
				w_2^\perp
			\end{pmatrix}.
			$$
			Note that
			\begin{align*}
				2\pi\dotp{\nabla f(w), w - P_S(w)} &= \theta_1 \|w_1^\perp\|^2 + \theta_2 \|w_2^\perp\|^2 + (\theta_1+\theta_2) \|w_1^\perp\|\|w_2^\perp\|\\
				&= \theta_1 \sin^2 \theta_1 \|w_1\|^2 + \theta_2 \sin^2 \theta_2 \|w_2\|^2 + (\theta_1+\theta_2)\sin \theta_1 \sin \theta_2 \|w_1\|\|w_2\|,
			\end{align*}
			where the first equality follows from~Lemma~\ref{lem: gradientonmanifold} and the second  follows from the definition of $\theta_1$ and $\theta_2$. On the other hand, by equation~\ref{eqn: dufunction},  for $w \in \cM$ sufficiently close to $S$, we have
			\begin{align*}
				2\pi L(w) &= \paren{\sin(\theta_1+\theta_2) - (\theta_1+\theta_2)\cos(\theta_1+\theta_2)}\|w_1\|\|w_2\| - \sum_{i=1}^{2}\paren{\sin\theta_i - \theta_i \cos \theta_i}\|w_i\|\|v\|\\
				&=  \paren{\sin(\theta_1+\theta_2) - (\theta_1+\theta_2)\cos(\theta_1+\theta_2)}\|w_1\|\|w_2\|\\
				&\qquad  - \sum_{i=1}^{2}\paren{\sin\theta_i - \theta_i \cos \theta_i}\|w_i\|(\|w_1\|\cos \theta_1 + \|w_2\|\cos \theta_2)\\
				&= \paren{\sin(\theta_1+\theta_2) - (\theta_1+\theta_2)\cos(\theta_1+\theta_2)}\|w_1\|\|w_2\|\\
				&\qquad- \paren{\sin \theta_1 \cos \theta_2 +\sin \theta_2 \cos \theta_1 - \theta_1 \cos \theta_1\cos \theta_2 - \theta_2 \cos\theta_2\cos\theta_1} \|w_1\|\|w_2\| \\
				&\qquad  - \paren{\sin \theta_1 \cos \theta_1- \theta_1\cos^2 \theta_1} \|w_1\|^2 - \paren{\sin \theta_2 \cos \theta_2- \theta_2\cos^2 \theta_2} \|w_2\|^2\\
				&= (\theta_1+\theta_2)\sin \theta_1 \sin\theta_2 \|w_1\|\|w_2\|\\
				&\qquad  - \paren{\sin \theta_1 \cos \theta_1- \theta_1\cos^2 \theta_1} \|w_1\|^2 - \paren{\sin \theta_2 \cos \theta_2- \theta_2\cos^2 \theta_2} \|w_2\|^2,
			\end{align*}
			where the second equality follows from $\|v\| = \|w_1\|\cos \theta_1 + \|w_2\|\cos \theta_2$, the third equality follows from direct expansion, and the fourth equality follows from basic properties of $\sin$ and $\cos$ functions. Consequently, we have
			\begin{align*}
				2\pi \dotp{\nabla L(w), w -P_S(w)} - 2\pi L(w) = \sum_{i=1}^{2} \paren{\theta_i \sin^2 \theta_i+ \sin\theta_i \cos \theta_i - \theta_i \cos^2 \theta_i} \|w_i\|^2.
			\end{align*}
			By Taylor expansion of $\sin \theta$ and $\cos \theta$, we can easily show that the right-hand side is nonnegative when $\theta_i$ is small. So, the result follows.
		\end{proof}
		
		\begin{lemma}\label{lem: nngdsizebound}
			For any $w \in \cM$ sufficiently close to $S$, we have 
			$$
			\|\nabla L(w)\| = O(\|w - P_S(w)\|^2).
			$$
		\end{lemma}
		\begin{proof}
			First, we note that for $w \in U\cap \cM$, we have
			$$
			w - P_S(w) = \begin{pmatrix}
				w_1^\perp\\
				w_2^\perp
			\end{pmatrix}.
			$$
			Therefore, we deduce 
			\begin{align*}
				\|\nabla_{w_1} f(w)\| &= \|\theta_1 w_1 - \theta_2 w_2\|\\
				&=   \|\theta_1 w_1^\parallel - \theta_2 w_2^\parallel\| + O(\|w_1^\perp\|^2 + \|w_2^\perp\|^2) \\
				&=  \|\tan(\theta_1) w_1^\parallel - \tan(\theta_2)w_2^\parallel\| + O(\|w_1^\perp\|^2 + \|w_2^\perp\|^2) \\
				&= |\|w_1^\perp\| - \|w_2^\perp\|| + O(\|w_1^\perp\|^2  + \|w_2^\perp\|^2)\\
				&= O(\|w - P_S(w)\|^2),
			\end{align*}
			where the first equality follows from Lemma~\ref{lem: gradientonmanifold},the second equality follows from Lemma~\ref{lem: thetaandorthogonalpart}, the third equality follows from Taylor expansion of $\tan(\theta)$ and Lemma~\ref{lem: thetaandorthogonalpart}, the fourth equality follows from the definition of $\theta$, and the fifth equality follows from the fact that $w_1^\perp + w_2^\perp = 0$ on $\cM$.
			By the same argument, we can show that $\|\nabla_{w_1} f(w)\|^2 = O(\|w - P_S(w)\|^2).$ The result follows.
		\end{proof}

		\subsection{Proof of Theorem~\ref{thm: simonduravine}}
		We fix $\bar w$ in the proof. For the reader's convenience, we use $w_\cM$ to represent points on the manifold $\cM$. Recall that $f_T = f\circ P_\cM$ and $f_N = f -f_T$.  Note that $f$ is $C^2$ in a neighborhood of set $S$ by \cite[Lemma 1]{safran2018spurious}. In addition, the manifold $\cM$ is an affine space, so the projection map $P_\cM$ is $C^\infty$ smooth. By Lemma~\ref{lem: gradientonmanifold}, for any $w_\cM \in \cM$ sufficiently close to $\bar w$, we have
		\begin{align*}
			\nabla f_N(w_\cM) = \nabla f(w_\cM) - P_{T_\cM} \nabla f(w_\cM) = 0.
		\end{align*} 
		Since $f$ is $C^2$ near $S$, by Taylor's theorem, for any $w$ sufficiently close to $\bar w$, we have 
		$$f_N(w) = O( \dist^2(w,\cM)).$$
		On the other hand, by Lemma~\ref{lem: llbbyr}, for $w$ sufficiently close to $\bar w$, we have
		\begin{align}\label{eqn:quadraticlb}
			f(w) \ge \Theta(1)\|w_1 + w_2 - v\|^2 = \Theta(1) \dist^2(w, \cM). 
		\end{align}
		Note also that $\nabla f(\bar w) =0$, so equation~\eqref{eqn:quadraticlb} implies that $\nabla f^2(\bar w)$ is positive definite when restricted onto $N_\cM$. By continuity of $\nabla^2 f$, for all $w$ sufficiently close to $\bar w$, the Hessian $\nabla^2 f(w)$ is positive definite when restricted onto $N_\cM$ and its eigenvalues are bounded away from zero. Since $\nabla^2 f_T = P_{T_\cM} \nabla^2 fP_{T_\cM}$, we have that $P_{N_\cM}\nabla^2 f_N(w)P_{N_\cM} = P_{N_\cM}\nabla^2 f(w)P_{N_\cM}$, and thus for any $w_\cM \in \cM$ sufficiently close to $\bar w$, $\nabla^2 f_N(w_\cM)$ is positive definite when restricted onto $N_\cM$ and its eigenvalues are bounded away from zero. Combining with the result that $\nabla f_N(w_\cM) = 0$ and $f_N(w_\cM) = 0$, for any $u \in N_\cM$, we have 
		\begin{align*}
			f_N(w_\cM+u) =  u^\top \nabla^2 f_N(w_\cM) u + o(1) \|u\|^2,
		\end{align*}
		where $o(1) \rightarrow 0$ as $\|u\| \rightarrow 0$ uniformly in $w_\cM$ and $u$ for $w_\cM$ near $\bar w$ since $\nabla^2 f_N$ is uniformly continuous near $\bar w$. 
		In other words, for $w$ sufficiently close to $\bar w$, we have 
		$$
		f_N(w) = \Theta(1) \dist^2(w,\cM).
		$$
		This proves that $\cM$ is a $C^2$ ravine at $\bar x$. The Item~\ref{item: localgrowthconditionravine}, Item~\ref{item: localaimingtowardsolutionravine}, and Item~\ref{item: localsizeofgdravine} of Assumption~\ref{assum:tangent} follow from Lemma~\ref{lem: nnthirdordergrowth}, Lemma~\ref{lem: nnbregularity}, and Lemma~\ref{lem: nngdsizebound}, respectively.
\end{document}